\documentclass[10pt,a4paper]{amsart}
\usepackage{amsmath,amsthm,amssymb}
\usepackage{times}
\numberwithin{equation}{section}
\usepackage{bbm}
\usepackage{enumitem}
\usepackage[left=1.15in,right=1.15in,top=1.15in,bottom=1.15in]{geometry} 
\usepackage[utf8]{inputenc}
\usepackage{textcomp}
\usepackage{calc}
\usepackage{ifthen,tabularx,graphicx,multirow}
\usepackage[table]{xcolor}
\usepackage{rotating}
\usepackage{stackengine}
\usepackage{MnSymbol}
\usepackage{mathtools}
\usepackage{subfig}
\newcolumntype{Y}{>{\centering\arraybackslash}X}
\usepackage{hyperref}
\usepackage{cite}

\newtheorem{theorem}{Theorem}
\newtheorem{lemma}[theorem]{Lemma}
\newtheorem{corollary}[theorem]{Corollary}
\newtheorem{definition}[theorem]{Definition}
\theoremstyle{remark}
\newtheorem{remark}[theorem]{Remark}
\numberwithin{theorem}{section}

\newcommand{\Wal}{\operatorname{Wal}}

\usepackage{scalerel}
\usepackage{stackengine}
\stackMath
\def\hatgap{2pt}
\def\subdown{-2pt}
\newcommand\reallywidehat[2][]{%
	\renewcommand\stackalignment{l}%
	\stackon[\hatgap]{#2}{%
		\stretchto{%
			\scalerel*[\widthof{$#2$}]{\kern-.6pt\bigwedge\kern-.6pt}%
			{\rule[-\textheight/2]{1ex}{\textheight}}
		}{0.5ex}
		_{\smash{\belowbaseline[\subdown]{\scriptscriptstyle#1}}}%
}}

\mathtoolsset{showonlyrefs}

\newcommand{\RM}{\mathcal{R}_M}

\newcommand{\SN}{\mathcal{S}_N}

\newcommand{\N}{\mathbb{N}}
\newcommand{\R}{\mathbb{R}}

\newcommand{\bfM}{\mathbf M}
\newcommand{\bfN}{\mathbf N}
\newcommand{\bfs}{\mathbf s}

\newcommand{\PNk}{P_{N_k}^{N_{k-1}}}
\newcommand{\PMl}{P_{M_l}^{M_{l-1}}}

\graphicspath{{figures/}}

\begin{document}

	\title[]{Non-uniform recovery guarantees for binary measurements and infinite-dimensional compressed sensing}

	\author{L. Thesing} 
	\address{DAMTP, University of Cambridge}
	\email{lt420@cam.ac.uk}

	\author{A. C. Hansen} 
\address{DAMTP, University of Cambridge}
\email{ach70@cam.ac.uk}

\begin{abstract}
Due to the many applications in Magnetic Resonance Imaging (MRI), Nuclear Magnetic Resonance (NMR), radio interferometry, helium atom scattering etc., the theory of compressed sensing with Fourier transform measurements has reached a mature level. However, for binary measurements via the Walsh transform, the theory has long been merely non-existent, despite the large number of applications such as fluorescence microscopy, single pixel cameras, lensless cameras, compressive holography, laser-based failure-analysis etc. Binary measurements are a mainstay in signal and image processing and can be modelled by the Walsh transform and Walsh series that are binary cousins of the respective Fourier counterparts. We help bridging the theoretical gap by providing non-uniform recovery guarantees for infinite-dimensional compressed sensing with Walsh samples and wavelet reconstruction. The theoretical results demonstrate that compressed sensing with Walsh samples, as long as the sampling strategy is highly structured and follows the structured sparsity of the signal, is as effective as in the Fourier case. However, there is a fundamental difference in the asymptotic results when the smoothness and vanishing moments of the wavelet increase. In the Fourier case, this changes the optimal sampling patterns, whereas this is not the case in the Walsh setting. 
\end{abstract}


\maketitle

\section{Introduction}
Since Shannon's classical sampling theorem \cite{Shannon, Shannon50}, sampling theory has been a widely studied field in signal and image processing. Infinite-dimensional compressed sensing \cite{GSinfCS, breaking, Nature_sci_rep, Clarice, PoonFrames, KutyniokLimShearletFourier, CS1} is part of this rich theory and offers a method that allows for infinite-dimensional signals to be recovered from undersampled linear measurements. This gives a non-linear alternative to other methods like generalized sampling \cite{2DCase, GS, sharpBounds, PolySSR, Hansen_JAMS,AHPWavelet, Ma} and the Parametrized-Background Data-Weak (PBDW)-method \cite{deVore1, deVore2, deVore3,PBDW, maday2, maday3} that reconstruct infinite-dimensional objects from linear measurement. However, these methods do not allow for subsampling, and hence are dependent on consecutive samples of, for example, the Fourier transform. Infinite-dimensional compressed sensing, on the other hand, is similar to generalized sampling and the PBDW-method but utilises an $\ell^1$ optimisation problem that allows for subsampling.

Beside the typical flagship of modern compressed sensing, namely MRI \cite{MRI, Unser_MRI}, there is also a myriad of other applications, like fluorescence microscopy \cite{fluorescence, Candes_PNAS}, single pixel cameras \cite{SinglePixel}, medical imaging devices like computer tomography \cite{Stanford_CT}, electron microscopy \cite{leary2013etcs}, lensless cameras \cite{LenslessImaging}, compressive holography \cite{compressiveHolography} and laser-based failure-analysis \cite{laserBasedFailureAnalysis} among others. The applications divide themselves in three different groups: those that are modelled by Fourier measurements, like MRI \cite{MRI}, those that are based on the Radon transform, as in CT imaging \cite{Stanford_CT, quinto2006xrayradon}, and those that are represented by binary measurements, which are named above. For Fourier measurements there exists a large history of research. However, for Radon measurements, the theoretical results are scarce and for binary measurements results have only evolved recently. In this paper we consider binary measurements and provide the first non-uniform recovery guarantees in one dimension for infinite-dimensional compressed sensing with the reconstruction with boundary corrected Daubechies wavelets. 

The setup of infinite-dimensional compressed sensing is as follows. We consider an orthonormal basis $\left\{\varphi_j\right\}_{j \in \N}$ of a Hilbert space $\mathcal{H}$ and an element $f \in \mathcal H$ with its representation
\[
f = \sum_{j \in \N}x_j \varphi_j \in \mathcal{H}, \quad x_j \in \mathbb{C},
\]
to be recovered from measurements given by linear operators working on $f$. That is, we have another orthonormal basis $\left\{\omega_i\right\}_{i \in \N}$
of $\mathcal{H}$ and we can access the linear measurements given by $l_i(f) = \langle f,  \omega_i \rangle$. 
Although the Hilbert space can be arbitrary, we will in applications mostly consider function spaces. Hence, we will often 
refer to the object $f$ as well as the basis elements as functions. We call the functions $\omega_i$, $i \in \N$ sampling functions and the space spanned by them $\mathcal S = \overline{\operatorname{span}}\left\{\omega_i : i \in \N \right\}$ sampling space. Accordingly, $\varphi_j$, $j \in \N$ are called reconstruction functions and $\mathcal R = \overline{\operatorname{span}}\left\{ \varphi_j : j \in \N \right\}$ reconstruction space. Generalized sampling \cite{GS,GS2,GSinfCS} and the PBDW-method \cite{PBDW} use the change of basis matrix $U = \{u_{i,j}\}_{i,j \in \mathbb{N}} \in \mathcal B ( \ell^2(\N))$ with $u_{i,j} = \langle \omega_i, \varphi_j \rangle$ to find a solution to the problem of reconstructing coefficients in the reconstruction space from measurements in the sampling space. This is also the case in infinite-dimensional compressed sensing.  In particular, we consider the following \emph{reconstruction problem}. Let $\Omega \subset \left\{ 0,1, \ldots, N_r \right\}$ be the subsampling set, where the elements are ordered canonically when needed, and let $P_\Omega$ the orthogonal projection onto the elements indexed by $\Omega$ and $||z||_2 \leq \delta$ some additional noise. For a fixed signal $f = \sum_j x_j \varphi_j$ and the measurements $g = P_\Omega U f + z$ the reconstruction problem is to find a minimiser of
\begin{equation}\label{Eq:ReconstructionProblem}
\min_{\xi \in \ell^1(\N)} \|\xi\|_{1} \text{ subject to } \|P_\Omega U \xi - g \|_2 \leq \delta.
\end{equation}

\section{Preliminaries}\label{Ch:NonRecGen}

\subsection{Setting and Definitions}

In this section we recall the settings from \cite{breaking} that are needed to establish the main results. 
First, note that we will use $a \lesssim b$ to describe that $a$ is smaller $b$ modulo a constant, i.e. there exists some $C >0$ such that $a \leq Cb$. Moreover, for a set $\Omega \subset \N$ the orthogonal projection corresponding to the elements of the canonical bases of $\ell^2(\N)$ with the indices of $\Omega$ is denoted by $P_\Omega$. Similar, for $N \in \N$ the orthogonal projection onto the first $N$ elements of the canonical basis of $\ell^2(\N)$ and $\ell^1(\N)$ is represented by $P_N$ and the projection on the orthogonal complement by $P_N^\perp$. If we project onto the sampling space $\SN$ this is denoted by $P_{\SN}$ and as before the complement by $P_{\SN}^\perp$. Finally, $P_b^a$ stands for the orthogonal projection onto the basis vectors related to the indices $\left\{ a+1, \ldots, b \right\}$.

Note that \eqref{Eq:ReconstructionProblem} is an infinite-dimensional optimisation problem, however, in practice \eqref{Eq:ReconstructionProblem} is replaced by 
\begin{equation*}
\min_{\xi \in \ell^1(\N)} \|\xi\|_{1} \text{ subject to } \|P_\Omega UP_L \xi - g \|_2 \leq \delta.
\end{equation*}
As $L \rightarrow \infty$ one recovers minimisers of \eqref{Eq:ReconstructionProblem} (see \S 4.3. in \cite{GSinfCS} for details).

X-lets such as wavelets \cite{WaveletDef}, Shearlets\cite{Shearlets2,Shearlets3} and Curvelets \cite{Curvelets,Curvelets2,Curvelets3} yield a specific sparsity structure. The construction includes a scaling function which allows to divide them into several levels. The same levels also dominate the sparsity structure. To describe this phenomena the notation of $(\mathbf s, \mathbf M)$-sparsity is introduced instead of global sparsity.
\begin{definition}[Def. 3.3 \cite{breaking}]\label{Def:SparsityInLevel}
	Let $x \in \ell^2(\N)$. For $r \in \N$ let $\mathbf M = (M_1, \ldots, M_r) \in \N$ with $1 \leq M_1 < \ldots < M_r$ and $\mathbf s = (s_1, \ldots , s_r) \in \N^r$, with $s_k \leq M_k-M_{k-1}$, $k=1,\ldots,r$ where $M_0 =0$. We say that $x$ is $(\mathbf s, \mathbf M)$-sparse if, for each $k=1, \ldots,r$,
	\begin{equation}
	\Delta_k := \operatorname{supp}(x) \cap \left\{M_{k-1}+1, \ldots, M_k\right\},
	\end{equation}
	satisfies $| \Delta_k| \leq s_k$. We denote the set of $(\mathbf s, \mathbf M)$- sparse vectors by $\Sigma_{\mathbf s, \mathbf M}$.
\end{definition}

Most natural signals are not perfectly sparse. But they can be represented with a small tail in the X-let bases, or with the according ordering in other sparsifying representation systems. Hence, in a large number of applications it is unlikely to ask for sparsity but compressibility.

\begin{definition}[Def. 3.4 \cite{breaking}]\label{Def:sMtermApproximation}
	Let $f = \sum_{j \in \N} x_j \varphi_j$, where $x= (x_j)_{j\in \N} \in \ell^2(\N)$. We say that $f$ is $(\mathbf s, \mathbf M)$- compressible with respect to $\left\{ \varphi_j \right\}_{j \in \N}$ if $\sigma_{\mathbf s,\mathbf M}(f)$ is small, where
	\begin{equation}
	\sigma_{\mathbf s,\mathbf M}(f) = \min_{\eta \in \Sigma_{\mathbf s, \mathbf M} } || x- \eta||_{1}.
	\end{equation}
\end{definition}

In terms of this more detailed description of the signal instead of classical sparsity it is possible to adapt the sampling scheme accordingly. Complete random sampling will be substituted by the setting of multilevel random sampling. This allows us later to treat the different levels separately. Moreover, this represents sampling schemes that are used in practice.

\begin{definition}[Def 3.2 \cite{breaking}]\label{Def:multilevelrandomsampling}
	Let $r \in \N, \mathbf N = (N_1, \ldots , N_r) \in \N^r$ with $1 \leq N_1 < \ldots <N_r$, $\mathbf m = (m_1, \ldots, m_r) \in \N^r$, with $m_k \leq N_k - N_{k-1}$, $k=1, \ldots,r$, and suppose that
	\begin{equation}\label{Eq:multilevelsampling1}
		\Omega_k \subset \left\{N_{k-1}+1, \ldots, N_k\right\}, \quad |\Omega_k| = m_k, \quad k=1, \ldots, r,
	\end{equation}
	are chosen uniformly at random without replacement, where $N_0 =0$. We refer to the set
	\begin{equation}
		\Omega = \Omega_{\mathbf N, \mathbf m} = \Omega_1 \cup \ldots \cup \Omega_r
	\end{equation}
	as an $(\mathbf N , \mathbf m)$- multilevel sampling scheme.
\end{definition}

\begin{remark}
	To avoid pathological examples, we assume as in \S 4 in \cite{breaking} that we have for the total sparsity $s = s_1 + \ldots s_r \geq 3$. This results in the fact that $\log(s) \geq 1$ and therefore also $m_k\geq 1$ for all $k=1,\ldots,r$.
\end{remark}

\section{Main results: Non-uniform recovery for the Walsh-wavelet case}\label{Ch:NonUniRec}
In this paper we focus on the reconstruction of one-dimensional signals from binary measurements, which can be modelled as inner products of the signal with functions that take only values in $\left\{0,1\right\}$. This arises naturally in examples like those mentioned in the introduction.  
We focus on the setting of recovering data in $L^2([0,1])$. However, the theory from \cite{breaking} also applies to general results for Hilbert spaces. Linear measurements are typically represented by inner products between sampling functions and the data of interest. Binary measurements can be represented with functions that take values in $\left\{0,1\right\}$, or, after a well-known and convenient trick of subtracting constant one measurements, with functions that take values in $\left\{-1,1\right\}$. For practical reasons it is sensible to consider functions for the sampling bases that provide fast transforms. Additionally, the function system should correspond well to the chosen representation system of the reconstruction space. For the reconstruction with wavelets, Walsh functions have proven to be a sensible choice, and are discussed in more detail in \S \ref{Ch:SamplingSpace}. Sampling from binary measurements has been analysed for linear reconstruction methods in \cite{WalshSSR,WalshHaar,Vegard} and for non-linear reconstruction methods in \cite{VegardUniform,moshtaghpour2020close}. We extend these results to the non-uniform recovery guarantees for non-linear methods. The combined work result in broad knowledge about linear and non-linear reconstruction for two of the three main measurement systems: Fourier and binary.

In the following we consider for the change of basis matrix 
\begin{equation}\label{eq:the_U}
U = \{u_{i,j}\}_{i,j \in \mathbb{N}} \in \mathcal B ( \ell^2(\N)), \quad u_{i,j} = \langle \omega_i, \varphi_j \rangle,
\end{equation}
where $\{\omega_j\}_{j \in \mathbb{N}}$ denotes the Walsh functions on $[0,1]$ as described in \S \ref{Ch:SamplingSpace}, and $\{\varphi_i\}_{i \in \mathbb{N}}$ demotes the Daubecies boundary wavelets on $[0,1]$ of order $p \geq 3$ described in \S \ref{sec:wavelets}. For the sake of readability, we always consider Daubechies boundary corrected wavelets in the following if we say wavelets.

We are now able to state the recovery guarantees for the Walsh-wavelet case. 

\begin{theorem}[Main theorem]\label{Theo:Main}
Given the notation above, let $\mathbf N = (N_0,\ldots,N_r)$ define the sampling levels as in \eqref{Eq:N} and $\mathbf M = (M_0, \ldots, M_r)$ represent the levels of the reconstruction space as in \eqref{Eq:M}. Consider $U$ as in \eqref{eq:the_U} , $\epsilon >0$ and let $\Omega = \Omega_{N,m}$ be a multilevel sampling scheme such that the following holds:
	\begin{enumerate}
		\item Let $N=N_r$, $K= \max_{k=1,\ldots,r} \left\{ \frac{N_{k}-N_{k-1}}{m_k} \right\}$, $M=M_r$, $s= s_1 + \ldots +s_r$ such that
		\begin{equation}\label{Eq:MainTheoAssumptionNM}
		N \gtrsim M^{2} \cdot \log(4MK\sqrt{s}).
		\end{equation}
		\item For each $k =1, \ldots, r$,
		\begin{align}\label{Eq:MainTheoAssumption}
		m_k \gtrsim \log(\epsilon^{-1}) \log \left( K^3s^{3/2} N \right) \cdot \frac {N_k - N_{k-1}} {N_{k-1}} \cdot \left( \sum_{l=1}^r 2^{-|k-l|/2}s_l \right)
		\end{align}
	\end{enumerate}
	Then with probability exceeding $1-s\epsilon$, any minimizer $\xi \in \ell^2(\N)$ of \eqref{Eq:ReconstructionProblem} satisfies
	\begin{equation}
	\|\xi - x \|_2 \leq c \cdot \left( \delta \sqrt{K}(1+L \sqrt{s}) + \sigma_{s,M}(f) \right),
	\end{equation}
	for some constant $c$, where $L = c \cdot \left( 1 + \frac{\sqrt{\log(6 \epsilon^{-1})}}{\log(4KM \sqrt{s})} \right)$. If $m_k = N_{k} - N_{k-1}$ for $1 \leq k \leq r$ then this holds with probability $1$.
\end{theorem}

This result allows one to exploit the asymptotic sparsity structure of most natural images under the wavelet transform. It was observed in Figure 3 in \cite{breaking} that the ratio of non-zero coefficients per level decreases very fast with increasing level and at the same time the level size increases. Hence, most images are not that sparse in the first levels and sampling patterns should adapt to that. However, they are very sparse in the higher levels. This difference in the sparsity is used in the theorem. The number of samples per level $m_k$ depends mainly on the sparsity in the related level $s_k$. The other sparsity terms $s_l$ with $l \neq k$ come in with a scaling of $2^{-|k-l|}$. This means that the impact of levels which are far away decays exponentially. 

The number of samples $m_k$ in relation to the level size $N_k - N_{k-1}$ also impacts the probability of an accurate reconstruction. Hence, in case the number of samples is very small the theorem still holds true but the probability that the algorithm succeeds becomes very small and the error large. Therefore, it is of high importance to choose the number of samples according to the desired accuracy and success probability. This is always a balancing act. The relationship also comes into play for the size of $K$. For large levels with very few measurements this value can become larger. Nevertheless, even if the largest levels are subsampled with only $1\%$ we get that $K = 100$. Additionally, the value $K$ only comes into play in a logarithmic term. Therefore, the impact of $K$ reduces to a reasonable size and the number of necessary samples or the relationship between $N$ and $M$ stays small.

\begin{remark}\label{Rem:ReductionNk}
	For awareness of potential extensions of this work to higher dimensions or other reconstruction and sampling spaces we kept the factor $(N_k - N_{k-1})/N_{k-1}$ in \eqref{Eq:MainTheoAssumption}. However, for the Walsh-wavelet case in one dimension this factor reduces to $1$.
	Hence, the Equation \eqref{Eq:MainTheoAssumption} can be further simplified to
	\begin{equation}
		m_k \gtrsim \log(\epsilon^{-1}) \log \left( K^3s^{3/2} N \right) \cdot \left( \sum_{l=1}^r 2^{-|k-l|/2}s_l \right),
	\end{equation}
	however, in general one needs the factor $(N_k - N_{k-1})/N_{k-1}$. 
\end{remark}

\begin{remark}\label{Rem:RelationFourier}
We would like to highlight the differences to the Fourier-wavelet case, i.e. to Theorem 6.2. in \cite{breaking}. The most striking difference is the squared factor in \eqref{Eq:MainTheoAssumptionNM}. In the Fourier-wavelet case this is dependent on the smoothness of the wavelet and shown to be 
$
	N \gtrsim M^{1 + 1/(2\alpha -1)} \cdot (\log(4MK\sqrt s ))^{1/(2\alpha -1)},$
where $\alpha$ denotes the decay rate under the Fourier transform, i.e. the smoothness of the wavelet. For very smooth wavelets this can be improved to 
\begin{equation}
	N \gtrsim M \cdot (\log(4MK\sqrt s ))^{1/(4\alpha -2)}.
\end{equation}
Hence, for very smooth wavelets we get the optimal linear relation, beside log terms. However, for non-smooth wavelets like the Haar wavelet, we get a squared relation instead of linear. The reason why we do not observe a similar dependence on the smoothness in terms of the sampling relation is that smoothness of a wavelet does not relate to a faster decay under the Walsh transform. 
This is also related to the fact that for Fourier measurements \eqref{Eq:MainTheoAssumption} become
\begin{equation}\label{eq:cond_wav1}
\begin{split}
m_k &\gtrsim \log(\epsilon^{-1})  \cdot  \log(\tilde N) \cdot \frac{N_k - N_{k-1}}{N_{k-1}}  \cdot
\Bigg(\hat s_k +\sum_{l=1}^{k-2}  s_l \cdot   2^{-(\alpha-1/2)A_{k,l}} +   \sum_{l=k+2}^r s_l \cdot 2^{-v B_{k,l}}\Bigg),
\end{split}
\end{equation}
where $A_{k,l}$ and $B_{k,l}$ are positive numbers, $\tilde N = (K\sqrt{s})^{1+ 1/v} 
N$, where $v$ denotes the number of vanishing moments, and $\hat s_k = \max\{s_{k-1},s_k,s_{k+1}\}.$ In particular, smoothness and vanishing moments of the wavelet does have an impact in the Fourier case, but not in the Walsh case. 
This is also confirmed in Figure \ref{fig:RecMatrix}, where we have plotted the absolute values of sections of $U$, where $U$ is the infinite matrix from \eqref{eq:the_U}. 
As can be seen in Figure \ref{fig:RecMatrix}, the matrix $U$ gets more block diagonal in the Fourier case with more vanishing moments confirming the dependence of $\alpha$ and $v$ in \eqref{eq:cond_wav1}. Note that for a completely block diagonal matrix $U$ the $m_k$ in \eqref{eq:cond_wav1} will only depend on $s_k$ and not any of the $s_l$ when $l \neq k$. In contrast this effect is not visible in the Walsh situation suggesting that the estimate in \eqref{Eq:MainTheoAssumption} captures the correct behaviour by not depending on $\alpha$ and $v$. The reason why is that a function needs to be smooth in the dyadic sense to have a faster decay rate under the Walsh transform. However, this is not related to classical smoothness but dyadic smoothness. So far it is not known which wavelets behave better under the Walsh transform. Therefore, it is possible that the estimate is not sharp and that we can get faster decay rates for specific wavelet orders. 

Finally, we want to highlight that this unknown relationship also leads to the squared factor in the relationship of $N$ and $M$ in \eqref{Eq:MainTheoAssumptionNM}. However, numerical examples in \S \ref{Ch:Numerics} suggest that this relation is not sharp. The experiments show that coefficients up to $N$ can be reconstructed, hence it is possible to reconstruct images with a reduced relation between the maximal sample and the maximal reconstructed coefficient. 
\end{remark}

\begin{figure}
	\centering
	\subfloat[Haar wavelets (Walsh)]{
		\includegraphics[width=0.32\textwidth]{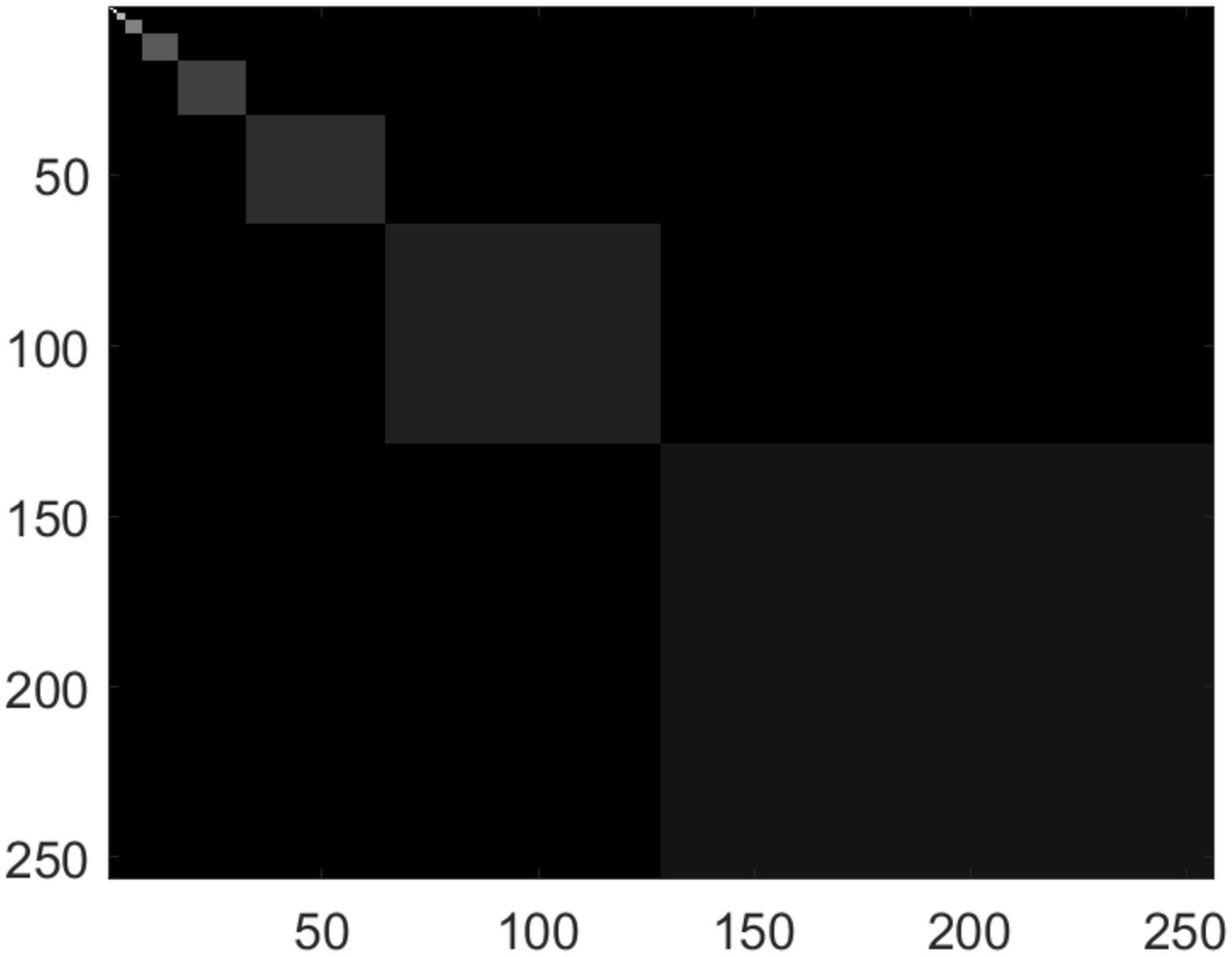}
		\label{fig:RecMatrixHaar}}
	\subfloat[$2$ vanish. mom. (Walsh)]{
		\includegraphics[width=0.32\textwidth]{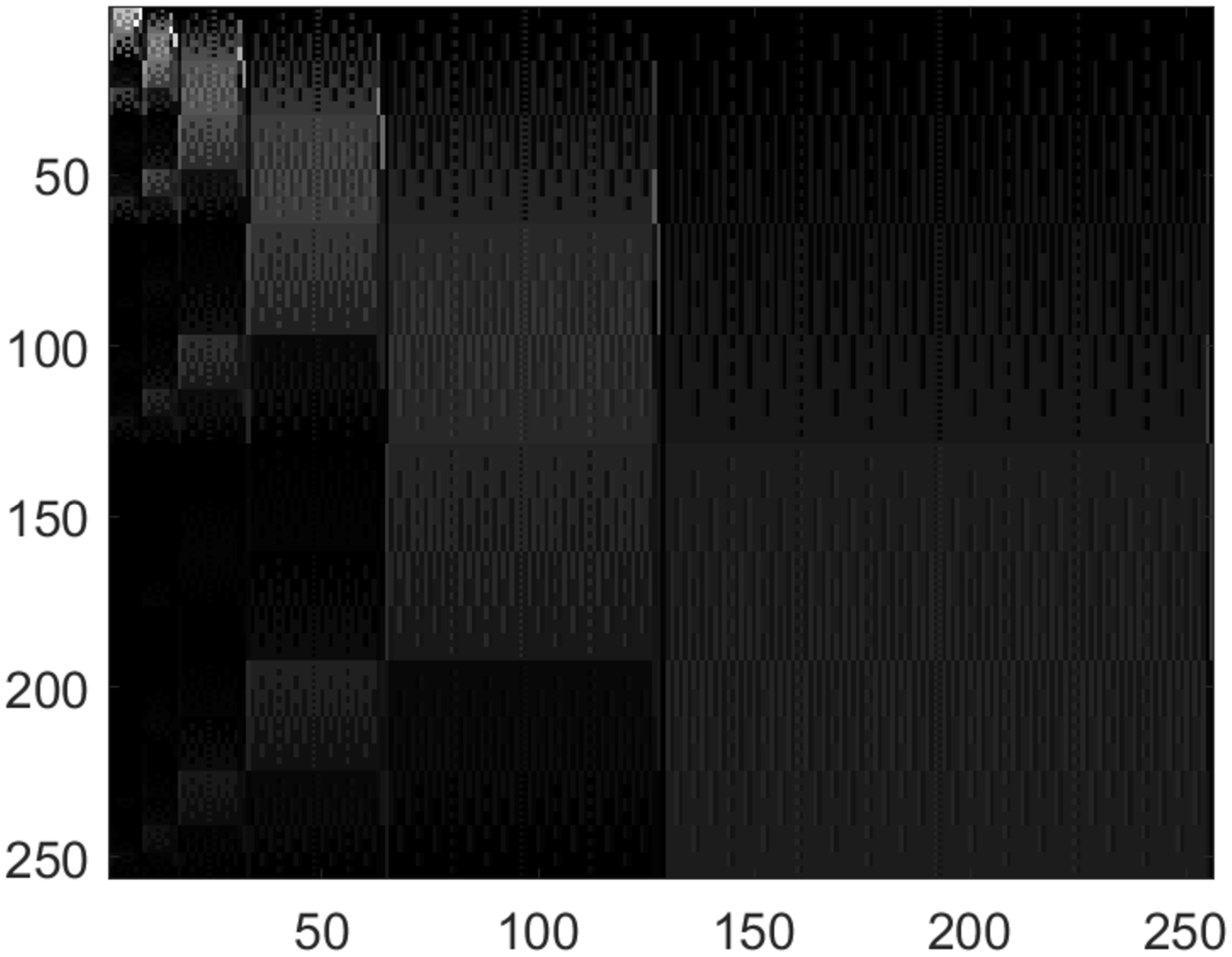}
		\label{fig:RecMatrix2}}
	\subfloat[$8$ vanish. mom. (Walsh)]{
		\includegraphics[width=0.32\textwidth]{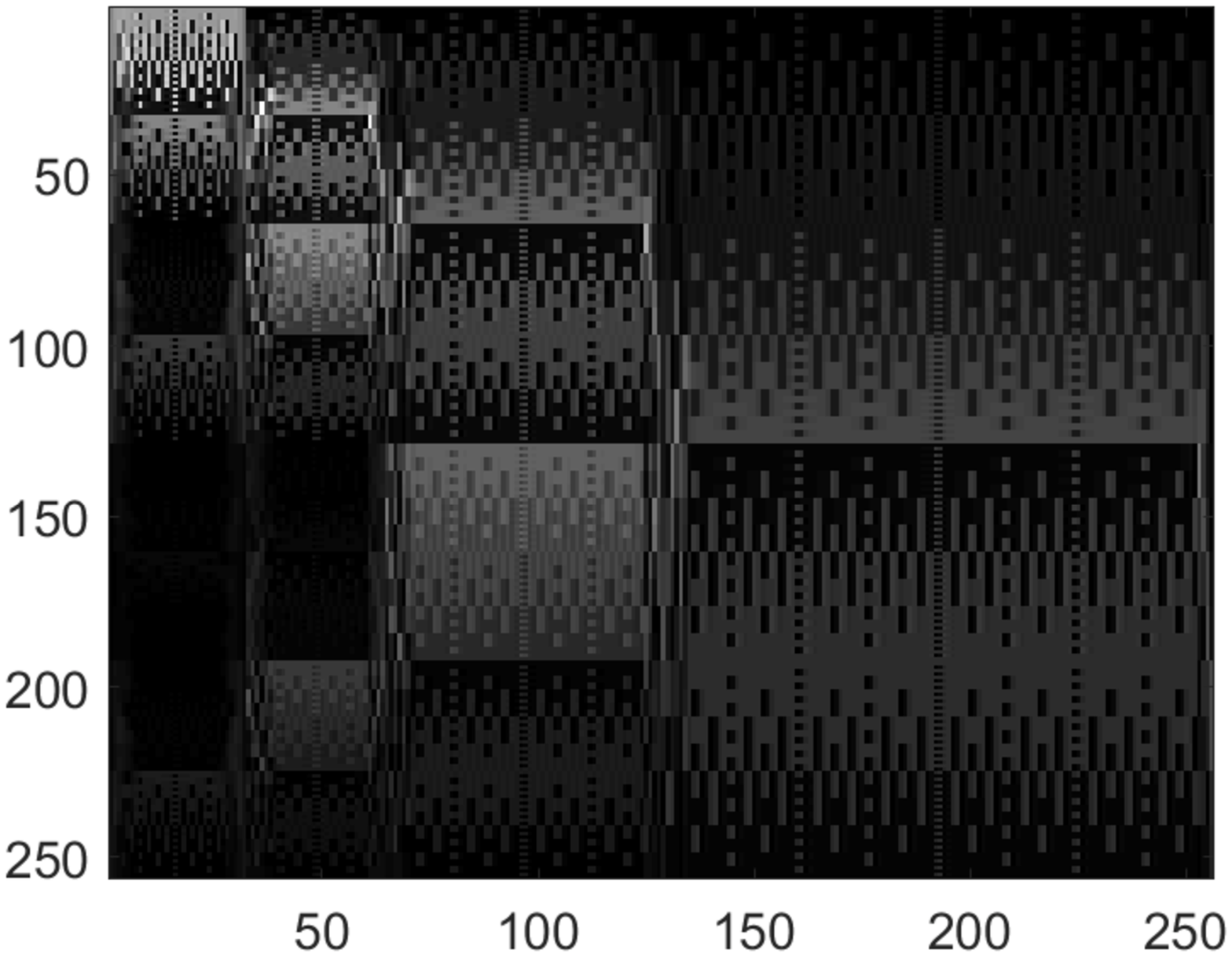}
		\label{fig:RecMatrix3}}
	
	\subfloat[Haar wavelets (Fourier)]{
		\includegraphics[width=0.32\textwidth]{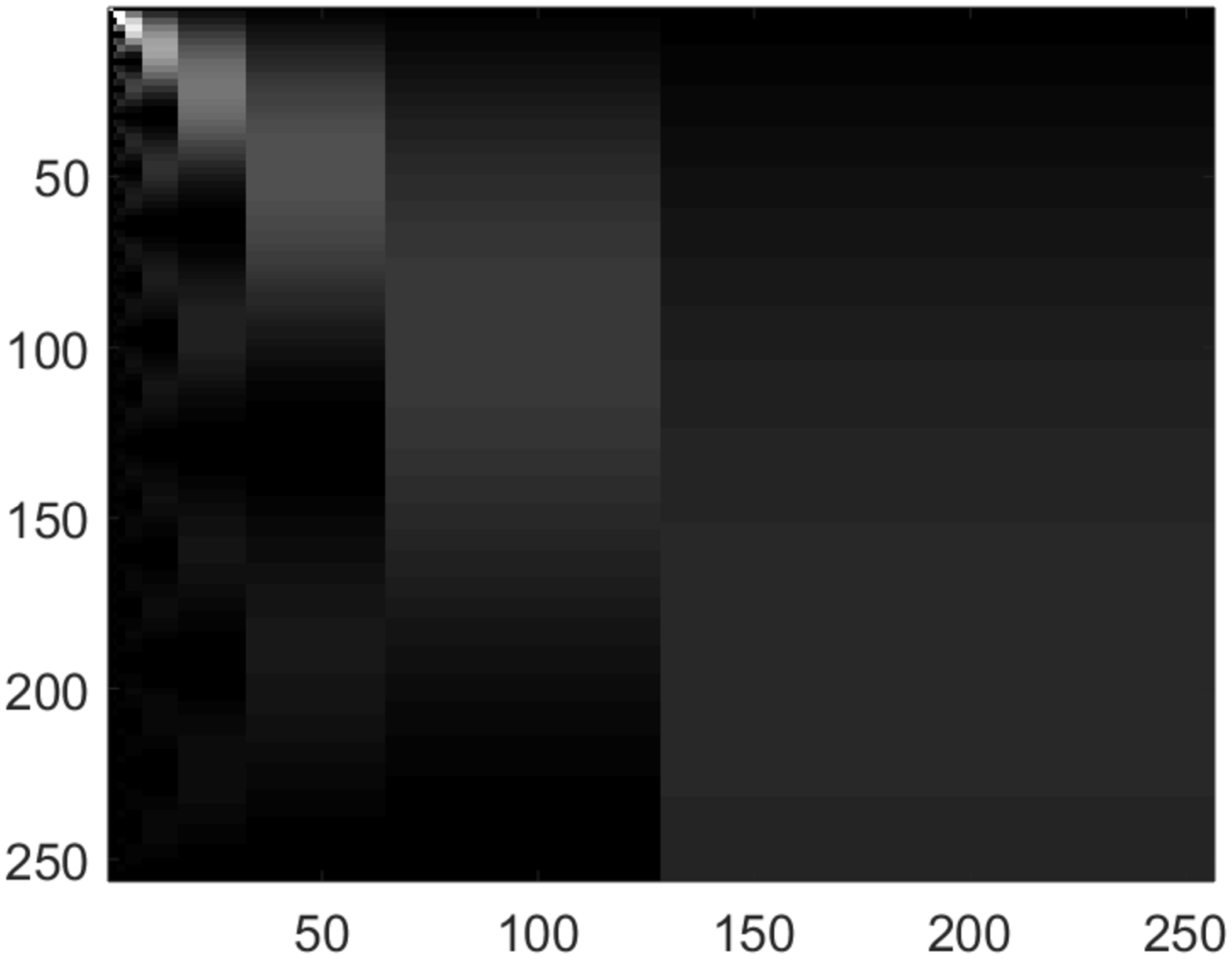}
		\label{fig:RecMatrixFourHaar}}
	\subfloat[$2$ vanish. mom. (Fourier)]{
		\includegraphics[width=0.32\textwidth]{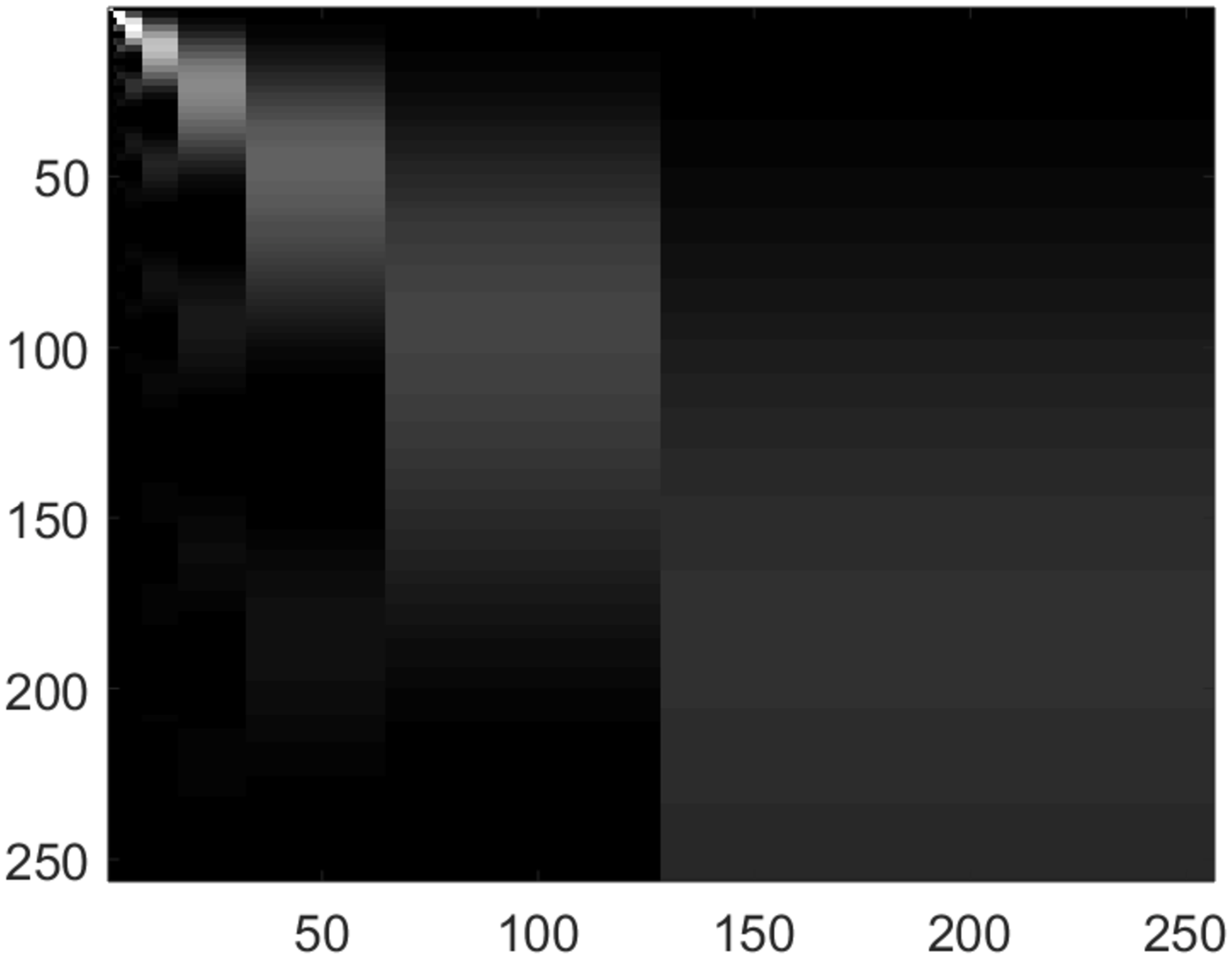}
		\label{fig:RecMatrixFour2}}
	\subfloat[$8$ vanish. mom. (Fourier)]{
		\includegraphics[width=0.32\textwidth]{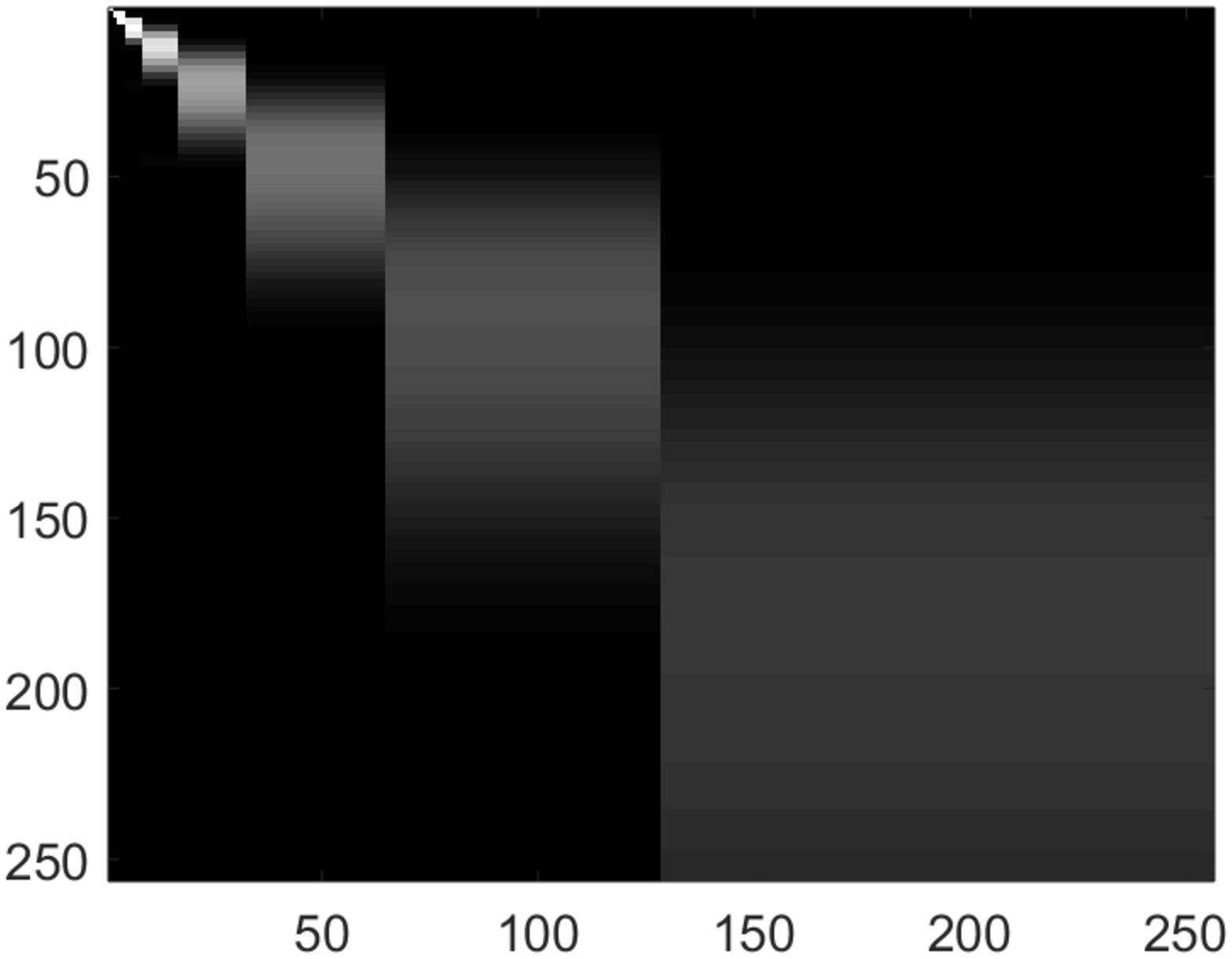}
		\label{fig:RecMatrixFour4}}
	\caption{Absolute values of $P_NUP_N$ with $N = 256$, where $U$ is the infinite matrix from \eqref{eq:the_U}, with boundary corrected Daubechies wavelets with different numbers of vanishing moments, and Walsh (upper row) and Fourier measurements (lower row). In the Fourier case, $U$ becomes more block diagonal as smoothness and the number of vanishing moments increase. This is not the case in the Walsh setting, suggesting that the non-dependence of smoothness and order (if higher than $p\geq 3$) in the estimate \eqref{Eq:MainTheoAssumption} is correct.}
	\label{fig:RecMatrix}
\end{figure}

\subsection{Connection to related works}

Reconstruction methods are mainly divided in two major classes of linear and non-linear methods. For the linear methods \emph{generalized sampling} \cite{GS} and the \emph{PBDW-method} \cite{PBDW} are prominent examples. 

Generalized sampling has evolved over time. The first closely related method is the \emph{finite section methods} introduced and analysed in \cite{FiniteSection4,FiniteSection1,FiniteSection3,FiniteSectionGroechnig}. This was further extended to \emph{consistent sampling} investigated by Aldroubi, Eldar, Unser and others \cite{SamplingTranslates, ConsistentSampling, RobustConsistentSampling, eldar2003FAA, RobustConsistentSampling3, RobustConsistentSampling4}. Finally generalized sampling has been studied by Adcock, Hansen, Hrycak, Gr\"ochenig, Kutyniok, Ma, Poon, Shadrin in  \cite{2DCase, GS, sharpBounds, PolySSR, Hansen_JAMS,AHPWavelet, Ma}. This works includes estimates for the general case of arbitrary sampling and reconstruction basis as well as more application focussed analysis. 

The PBDW-method evolved from the work of Maday, Patera, Penn and Yano in \cite{maday2} first under the name \textit{generalized empirical interpolation method}. This was then further analysed and extended to the PBDW-method by Binev, Cohen, Dahmen, DeVore, Petrova, and Wojtaszczyk \cite{deVore1, deVore2, deVore3,PBDW, maday3}. 

Both methods have in common that the relationship between the number of samples and reconstructed coefficients, the so called \emph{stable sampling rate} (SSR), controls the numerical stability and accuracy. The SSR has to be analysed for different application with their related sampling and reconstruction bases. It was shown that the SSR is linear for the Fourier-wavelet \cite{linearity}, Fourier-shearlet \cite{Ma} and Walsh-wavelet case \cite{WalshSSR}. However, this is not always the case as for the Fourier-polynomial situation \cite{PolySSR}.

In the non-linear setting the most prominent reconstruction technique is infinite-dimensional compressed sensing \cite{CS1} with its extension to structured CS as introduced in \cite{breaking}. Here, the analysis relies on the properties of the change of basis matrix and the sparsity of the signal. Early results have promoted random samples which have later been shown to be not as efficient for signals with a structured sparsity.

Similar to the linear methods the analysis for general sampling and reconstruction bases has been extended to the special applications. The Fourier case for different reconstruction systems has been analysed by Adcock, Hansen, Kutyniok, Lim, Poon and Roman \cite{breaking, KutyniokLimShearletFourier, GSinfCS, PoonFrames, Clarice}. Closely related to the recovery guarantees in this paper the following guarantees have been provided. For the Fourier wavelet case we know uniform recovery guarantees \cite{li2017uniformRecoveryWavFourier} and non-uniform recovery guarantees \cite{breaking,NoteOnHaarFourier}. For Walsh measurements we have coherence estimates by Antun \cite{Vegard}, uniform recovery guarantees from Adcock, Antun and Hansen \cite{VegardUniform} and an analysis for variable and multilevel density sampling strategies for the Walsh-Haar case and finite-dimensional signals by Moshtaghpour, Dias and Jacques in \cite{moshtaghpour2020close}. In this paper we present the non-uniform results for the Walsh-wavelet case as has been studied for the Fourier case in \cite{breaking,NoteOnHaarFourier}.

\subsection{Sampling and Reconstruction space}

\subsubsection{Sampling Space}\label{Ch:SamplingSpace}
We start with the sampling space. To model binary measurements Walsh functions have proven to be a good choice. They behave similar to Fourier measurements with the difference that they work in the dyadic rather than the decimal analysis. They also have an increasing number of zero crossing. This leads to the fact that the change of basis matrix $U$ gets a block diagonal structure, as can be seen in Figure \ref{fig:RecMatrix}. Moreover, it can be observed that $U$ is asymptotic incoherent. The incoherence of the matrix together with the asymptotic sparsity can be exploited in the reconstruction part. However, the fact that sampling functions are defined in the dyadic analysis leads to some difficulties and specialities in the proof.

Let us now define the Walsh functions, which form the kernel of the Hadamard matrix. Then we proceed with their properties and the definition of the Walsh transform.

\begin{definition}[\S 9.2 \cite{deutschWalsh}]
	Let $z =\sum_{i\in \mathbb{Z}} z_i 2^{i-1}$ with $n_i \in \left\{0,1 \right\}$ be the dyadic expansion of $z \in \mathbb{R}_+$. Analogously, let $x = \sum_{i \in \mathbb{Z}} x_i 2^{i-1}$ with the dyadic expansion $x_i \in \left\{ 0,1 \right\}$. The \emph{generalized Walsh functions} in $L^2([0,1])$ are given by 
	\begin{equation}
	\operatorname{Wal}(z,x) = (-1)^{\sum_{i \in \mathbb{Z}} (z_i + z_{i+1})x_{-i-1}}.
	\end{equation}
\end{definition}
This definition can also be extended to negative inputs by $\Wal(-z,x) = \Wal(z,-x) = -\Wal(z,x)$. Walsh functions are one-periodic in the second input if the first one is an integer. The first input $z$ is commonly denoted as \emph{parameter} or \emph{sequency} because of its relation to the number of zero crossings. For a fixed $z$ the function is then treated like a one-dimensional function with its input $x$ which could be interpreted as time as in the Fourier case.

The definition is extended to arbitrary inputs $z \in \R$ instead of the more classical definition for $z \in \N$. We would like to make the reader aware of different orderings of the Walsh functions. The one presented here is the \emph{Walsh-Kaczmarz} ordering in Figure \ref{fig:32Walsh}. It is ordered in terms of increasing number of zero crossings. This has the advantage that it relates nicely to the scaling of wavelets. Two other possible orderings are presented in \cite{deutschWalsh}. We have \emph{Walsh-Paley} in Figure \ref{fig:Walshdyad} with 
\begin{equation}
	\Wal_{\text{Pal}}(z,x) = (-1)^{\sum_{i \in \mathbb Z}z_ix_{-i}}
\end{equation}
and \emph{Walsh-Kronecker} in Figure \ref{fig:WalshHad} with
\begin{equation}
	\Wal_{\text{Kron}}(z,d,x) = (-1)^{\sum_{i=1}^{d}z_{d-i}x_{-i}}.
\end{equation}
Both have the drawback that the number of zero crossings is not increasing. This is the reason why we are not able to get the block diagonal structure in the change of basis matrix and hence do not get as much structure to exploit. The Walsh-Kronecker ordering has another disadvantage and is not often used in practice. The definition of the function includes knowledge about the length $d$ of the maximum sequence $z_{\max} \leq 2^d$ we are interested in. Depending on this value the function change also for smaller inputs $z \leq z_{\max}$, i.e. there is a third input $d$ which also leads to changes. Hence, in case one wants to change the number of samples and also increase the largest value for the samples all functions and measurements have to be recomputed, which is undesirable in practice, where measurements are expensive and time consuming. 

\begin{figure}
	\subfloat[Walsh-Kaczmarz 
	]{\includegraphics[width=0.31\textwidth]{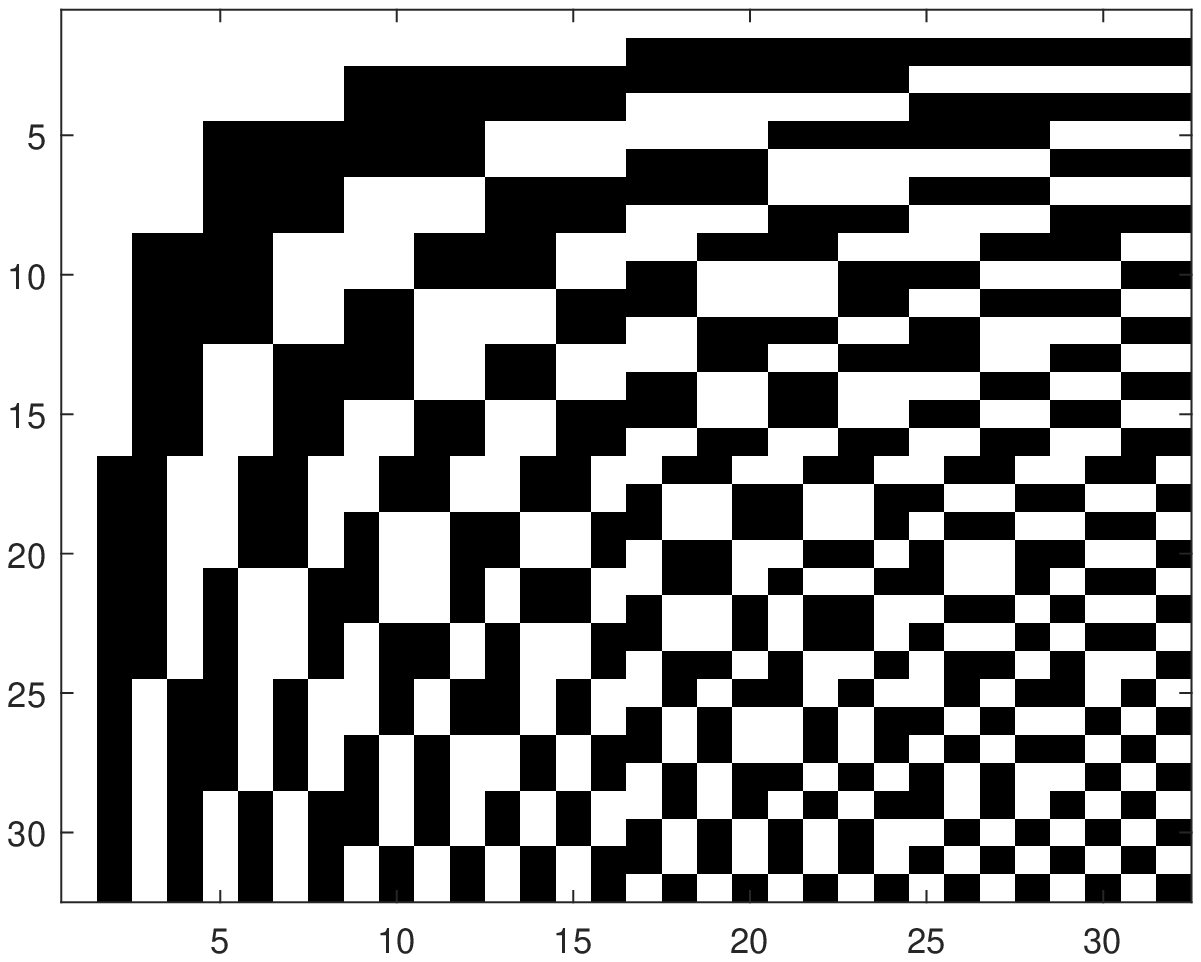}\label{fig:32Walsh}}
	\subfloat[Walsh-Paley 
	]{\includegraphics[width=0.31\textwidth]{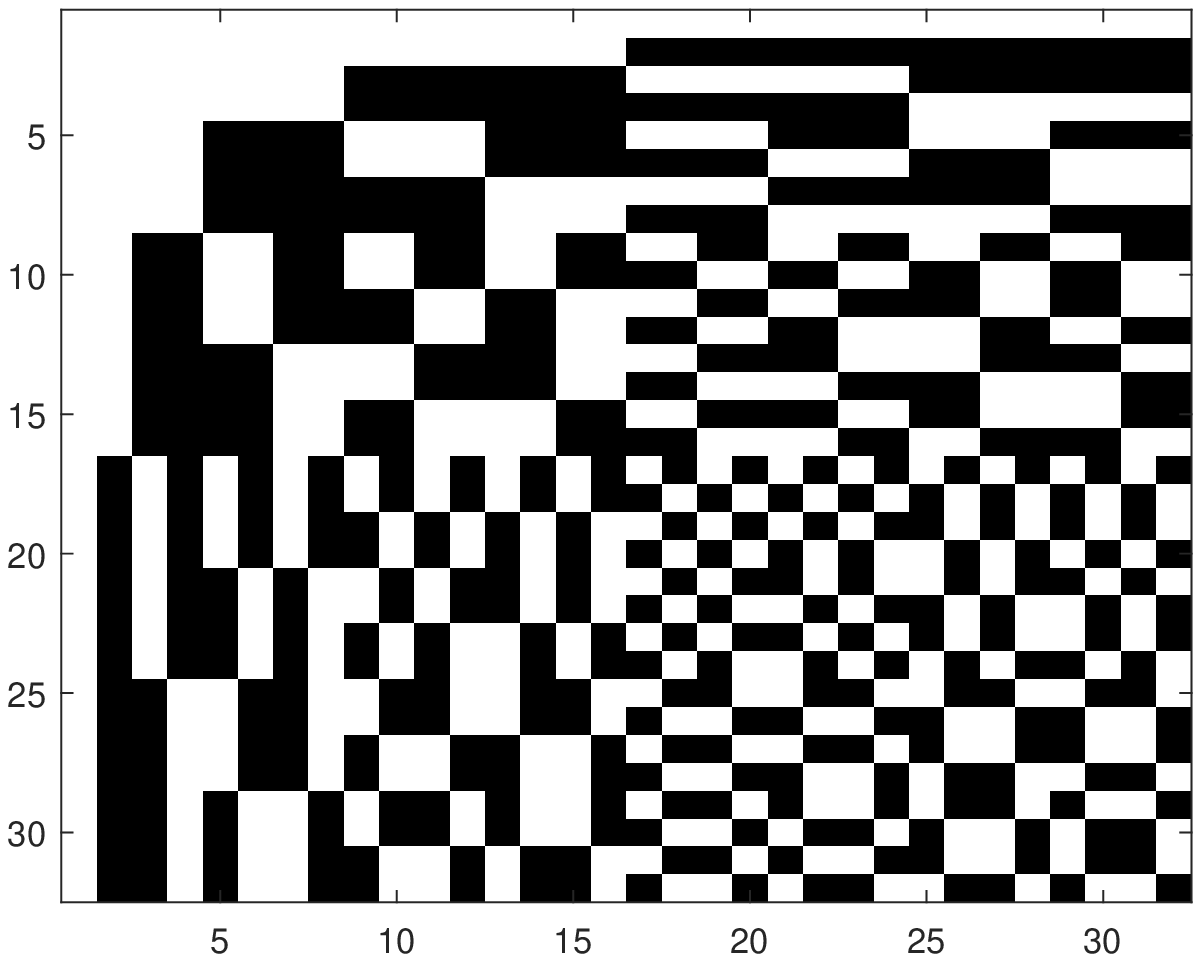}\label{fig:Walshdyad}}
	\subfloat[Walsh Kronecker
	]{\includegraphics[width=0.31\textwidth]{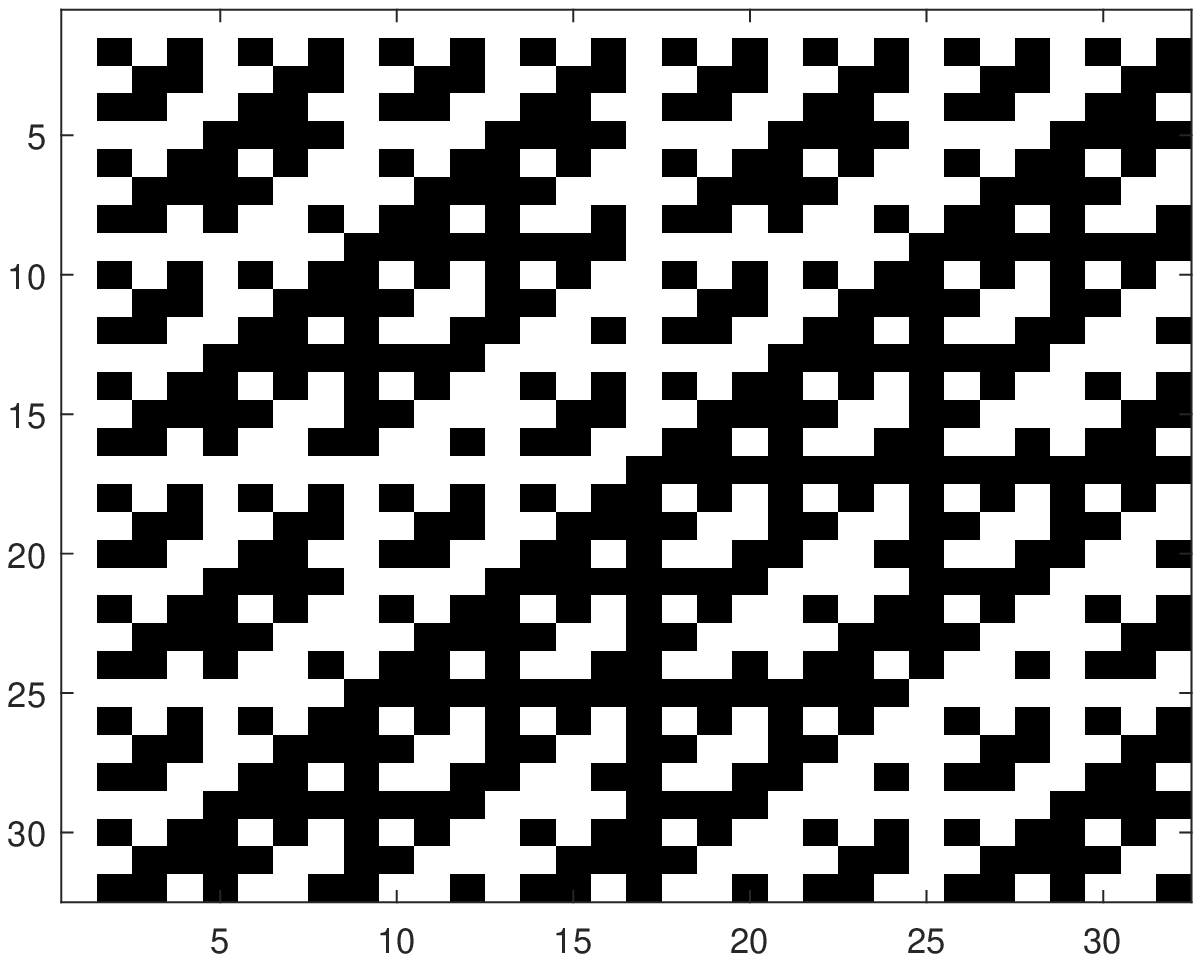}\label{fig:WalshHad}} 
	\caption{Different orderings of first $32$ Walsh functions}
	\label{fig:WalshOrderings}
\end{figure}

For the sampling pattern we divide the sequency parameter $z$ into blocks of doubling size. This is a natural division for two reasons. First, the size of the smallest constant interval decreases after every block. Second and more importantly, these blocks relate nicely into the level structure of the wavelets, discussed in the following chapter. We can see in Figure \ref{fig:RecMatrix} that the blocks relate nicely to the visible block structure of the change of basis matrix $U$. 

After the small excursion on orderings we now define the sampling space in one dimension by
\begin{equation}
	\mathcal{S} = \overline{\operatorname{span}} \left\{ \Wal(n,\cdot), n \in \mathbb N \right\},
\end{equation}
where $\overline{\operatorname{span}}$ denotes the closure of the set of linear combinations of the elements.
In general, it is not possible to acquire or save an infinite number of samples. Therefore, we restrict ourselves to the sampling space according to $\Omega_{N,m}$ or $\left\{1,\ldots,N\right\}$, i.e.
\begin{equation}
\mathcal{S}_{\Omega_{N,m}} = \operatorname{span} \left\{ \Wal(n,\cdot), n \in \Omega_{N,m} \right\} \quad \text{ or } \quad \mathcal{S}_{N} = \operatorname{span} \left\{ \Wal(n,\cdot), n \leq N \right\}.
\end{equation}

The Walsh functions obey some interesting properties which have been shown in \S 2.2. in \cite{WalshSSR}: the \emph{scaling property}, i.e.  $\Wal(2^jz,x) = \Wal(z,2^jx)$ for all $j \in \mathbb{N}$ and $n, x \in \mathbb{R}$ and the \emph{multiplicative identity}, i.e. $\Wal(z,x)\Wal(z,y) = \Wal(z,x \oplus y)$, where $\oplus$ is the dyadic addition.
With the Walsh functions we are able to define the continuous Walsh transform as a mapping from $L^1([0,1]) \mapsto L^1([0,1])$ by
\begin{equation}
\reallywidehat[W]{f}(z) = \langle f(\cdot), \Wal(z,\cdot) \rangle = \int_{[0,1]} f(x) \Wal(z,x)dx, ~ n \in \mathbb{R}.
\end{equation}
We will also use the notation $\mathcal W$ for the Walsh transform, similar to the Fourier operator $\mathcal F$, with 
\begin{equation}
	\mathcal W \left\{ f(x) \right\} (z) = \int_{[0,1]} f(x) \Wal(z,x)dx.
\end{equation} 
The properties from the Walsh functions are easily transferred to the Walsh transform.
We state now some more statements about the Walsh functions and the Walsh transform, which are necessary for the main proof.

\begin{lemma}[Cor. 4.3 \cite{WalshSSR}]\label{Cor:scalingProp}
	Let $t \in \mathbb{N}$ and $x \in [0,1)$, then the following holds:
	\begin{equation}
	\mathcal{W}\left\{f(x+ t) \right\} (s) = \mathcal{W}\left\{ f(x) \right\} (s) \Wal(t,s).
	\end{equation}
\end{lemma}

Remark that this only holds because $x$ and $t$ do not have non-zero elements in their dyadic representation at the same spot and therefore the dyadic addition equals the decimal addition. Next, we consider Walsh polynomials and see how we can relate the sum of squares of the polynomial to the sum of squares of its coefficients.

\begin{lemma}[Lem. 4.6 \cite{WalshSSR}]\label{Lemma2DphiSum}
	Let $A,B \in \mathbb{Z}$ such that $A \leq B$ and consider the Walsh polynomial $\Phi(z) = \sum_{j=A}^{B} \alpha_{j} \Wal(j,z)$ for $z \in \mathbb{R}_+$. If $L = 2^{n}$, $n \in \mathbb{N}$ such that $2L \geq B - A +1$, then
	\begin{equation}
	\sum\limits_{j=0}^{2L-1} \frac 1 { 2 L} \left| \Phi(\frac j {2 L})\right|^2 = \sum\limits_{j=A}^{B} |\alpha_{j}|^2 .
	\end{equation}
\end{lemma}

In the proof of Lemma \ref{Lem:EstimateSSR} we analyse the inner product of the wavelets with the Walsh function. For this we will combine the shifts in the wavelet into a Walsh polynomial. With this lemma at hand this is then easily bounded.

\subsubsection{Reconstruction Space}\label{sec:wavelets}

Next, we define the reconstruction space. As we are mainly interested in the reconstruction of natural signals in one dimension, we use Daubechies wavelets \cite{WaveletDef} and their boundary corrected versions \cite{boundaryWavelets}. They provide good smoothness and support properties. Moreover, they obey the Multi-resolution analysis (MRA). This results in the fact that the coefficients of natural signals obey a special sparsity structure with a lot of coefficients in the first part and fewer non-zero elements in the later coefficients. 

We start with the definition of classical Daubechies wavelet and then discuss the restriction to boundary corrected ones. The wavelet space is described by the wavelet $\psi$ at different levels and shifts $\psi_{j,m}(x) = 2^{j/2} \psi(2^j x -m )$ for $j, m \in \mathbb{N}$, i.e. we have the wavelet space at level $j$
\begin{equation}
	W_j := \operatorname{span}\left\{ \psi_{j,m} ,m \in \mathbb{N} \right\}.
\end{equation}
They build a representation system for $L^2(\mathbb{R})$, i.e. $\overline{\bigcup_{j\in \mathbb N}} W_j = L^2(\mathbb R)$. For the MRA we also define the scaling function $\phi$ and the according scaling space
\begin{equation}
	V_j = \operatorname{span}\left\{ \phi_{j,m}, m \in \mathbb N \right\},
\end{equation}
where $\phi_{j,m} (x) = 2^{j/2} \phi (2^j x - m)$. We have that $V_j = V_{j-1} \oplus W_{j-1}$ and $L^2(\mathbb R) = \operatorname{closure} \left\{ V_{J} \oplus \bigcup_{j \geq J} W_{j} \right\}$. The Daubechies scaling function and wavelet have the same smoothness properties. Therefore, they also have the same decay rate under the Walsh transform. The decay rate is of high importance for the analysis of the properties of the change of basis matrix.

The classical definition of Daubechies wavelets has a large drawback for our setting. Normally, they are defined on the whole line $\R$. Due to the fact that Walsh functions are defined on $[0,1]$ it is necessary to restrict the wavelets also to $[0,1]$. Otherwise there will be elements in the reconstruction space which are not in the sampling space and therefore the solution could not be unique. Hence, we are using boundary corrected wavelets (\S 4 \cite{boundaryWavelets}).  

For the definition of boundary wavelets, we have to correct those that intersect with the boundary. We start with the definition of the scaling space and continue with the wavelet space. For the discussion we consider for now the adaptation for the left boundary zero. If we remove all scaling functions which intersect with zero, the new function set does not even represent polynomial functions. Therefore, we have added the following functions.
\begin{equation}
\widetilde{\phi}^{\text{left}}_n(x) = \sum\limits_{l = 0}^{2p-2}  \binom{l}{n} \phi ( x + l - p +1) .
\end{equation}
These functions together with the translates of the scaling function with support on the positive line span all polynomials with degree smaller or equal to $p-1$ on $[0,\infty)$, where $p$ is the order of the scaling function. Next, we do the analogue for the right boundary. For this sake we first construct the functions for $(-\infty,0]$ simply by mirroring the $\widetilde{\phi}^{\text{left}}_n(x)$. To have the discussion on $[0,1]$ we shift the function to the right, such that we get
\begin{equation}
	\widetilde{\phi}_n^{\text{right}}(x) = \widetilde{\phi}_{-1-n}^{\text{left}}(-x) .
\end{equation}

Now, consider the lowest level $J_0$ such that the scaling functions do only intersect with one boundary $0$ or $1$, i.e. $2^{J_0} \geq 2p-1$. Then the interior scaling functions together with the the newly defined functions span $L^2([0,1])$. However, they are not orthogonal. Therefore, we apply the Gram-Schmidt procedure and obtain the function $\phi^{\text{right}}$ and $\phi^{\text{left}}$. The new functions have staggered support and a maximal support size of $2p-1$ and hence still have the desired property of a small support. The new function system contains at every level $2^j+2$ functions. However, in most applications it is desirable to only have $2^j$ many. Therefore, we remove the two outermost inner scaling functions. The new scaling space at level $j$ is given by:

\begin{equation}
V_j^b = \operatorname{span} \left\{ \phi_{j,n}^b : n=0, \ldots 2^j-1 \right\},
\end{equation}
where
\begin{align} \label{EqBoundScaling}
\phi_{j,n}^b(x) = \begin{cases}
2^{j/2} \phi^{\text{left}}_n (2^j x)  &  n=0,\ldots p-1 \\
2^{j/2} \phi_n (2^j x) & n = p,\ldots 2^j - p -1 \\
2^{j/2} \phi_{2^j -n-1} ^{\text{right}} (2^j (x-1)) & n = 2^j -p, \ldots 2^j -1 .
\end{cases}
\end{align}

The new system still obeys the MRA. Hence, the boundary wavelets can be deduced from the boundary corrected scaling functions as 
\begin{equation}
W_j^b = V_{j+1}^b \cap (V_j^b)^\perp .
\end{equation}
Fortunately, we only need the smoothness properties of the wavelet, especially only the knowledge if the function is Lipschitz continuous, for the decay rate under the Walsh transform in Lemma \ref{Lemma:Decayrate}. This is preserved also after the modification to the boundary corrected wavelets. The boundary wavelet will be denoted by $\psi^b$ and $\psi_{j,m}^b(x) = 2^{j/2} \psi^b (2^j x -m)$. Because we will use in the analysis the MRA property and hence replace the elements from the reconstruction space by the scaling function, as presented in the next paragraph, we do not need the details about the construction of the wavelet. The interested reader should seek out for \cite{boundaryWavelets} for a detailed analysis.

We now want to consider the reconstruction space. Let $R \in \N$ and $M=2^R$ than we have that the reconstruction space of size $M$ is given by
\begin{equation}\label{Eq:RecSpace1}
	\RM = V_{J_0}^b \oplus W_{J_0}^b \oplus \ldots \oplus W_{R-1}^b = V_R^b
\end{equation} 
This representation with the scaling function and wavelet suggests the ordering in different level according to the index $j$ in the next section \ref{Ch:Ordering}.

It was proved in \cite{boundaryWavelets} that $V_R$ can be spanned by the scaling function, its translates and the reflected version $\phi^\#(x) = \phi(-x+1)$, i.e.
\begin{equation}\label{Eq:RecSpace2}
V_R = \operatorname{span} \left\{\phi_{R,m}, m = 0, \ldots , 2^j - p -1, \phi_{R,m}^\#, m = 2^j-p, \ldots, 2^j -1 \right\}.
\end{equation}
With this representation of the reconstruction space we are able to present $\varphi \in \RM$ with $||\varphi||_2 =1$ as
\begin{equation}\label{Eq:repvarphi1d}
\varphi = \sum\limits_{n=0}^{2^R-p-1} \alpha_k \phi_{R,n} + \sum\limits_{n = 2^R -p}^{2^R-1} \beta_k \phi_{R,n}^\# \text{ with } \sum\limits_{n=0}^{2^R-p-1} |\alpha_n|^2 + \sum\limits_{n = 2^R -p}^{2^R-1} |\beta_n|^2= 1 .
\end{equation}

\begin{remark}
	We consider here only the case of Daubechies wavelets of order $p \geq 3$ and $p=1$, i.e. the Haar wavelet. The theory also holds for the case for order $p=2$. Nevertheless, we get unpleasant exponents $\alpha$ depending on the wavelet and different cases to consider. The results do not improve with more smoothness for the higher order wavelets. In contrast for the Haar wavelet, we can get even better estimates due to the perfect block structure of the change of basis matrix in that case. A detailed analysis of the relation between Haar wavelets and Walsh functions can be found in \cite{WalshHaar} and we discuss the recovery guarantees for this special case in \S \ref{Ch:WalshHaar}.
\end{remark}

For the evaluation of the change of basis matrix we investigate its elements, i.e. the inner product between the Walsh function and wavelet or scaling function, respectively. To ease this analysis we use the reconstruction space representation as in \eqref{Eq:RecSpace2}. This reduces the analysis to the scaling function. However, to avoid a lot of different cases we aim to take the shifts out of the inner product. To do this we will introduce Corollary \ref{Cor:scalingProp}.
However, in the assumptions we have that $t \in \N$ and $x \in (0,1]$. And because we also transfer the scaling factor $2^R$ out of the scaling function. The scaling function in level $0$ has that its support is larger than $[0,1]$. Therefore, Lemma \ref{Cor:scalingProp} could not be used, i.e.

\begin{align}
\int\limits_{2^{-R}(n-p+1)}^{2^{-R}(n+p)} 2^{R/2} \phi (2^R x - n)\Wal( k, x)dx 
& = 2^{-R/2} \int\limits_{-p+1}^{p} \phi(x)\Wal( k, 2^{-R}(x +n))dx \\
& \neq 2^{-R/2} \int\limits_{-p+1}^{p} \phi(x)\Wal( k, 2^{-R}(x \oplus n))dx.
\end{align}

Therefore, we have to separate the scaling function into parts which have support in $[0,1]$. Remark that this is not a contradiction to the construction of the boundary wavelets. They are indeed supported in $[0,1]$. However, only from the beginning of the scaling $J_0$ and not the scaling function at level $0$. Therefore, we use the representation of the scaling function at level $0$ from \cite{WalshSSR} as
\begin{equation}\label{phii}
\phi(x) = \sum\limits_{i=-p+2}^p \phi_i (x-i+1) \text{ with } \phi_i(x) = \phi(x+i-1) \mathcal{X}_{[0,1]}(x)
\end{equation}
and
\begin{equation}
\phi_{R,n} = 2^{R/2}\sum\limits_{i=-p+2}^p \phi_i(2^Rx-i+1-n) .
\end{equation}
This can also be done accordingly for the reflected function $\phi^\#$. 
More detailed information about this problem can be found in \cite{WalshSSR}.

\subsubsection{Ordering}\label{Ch:Ordering}

\begin{figure}
	\includegraphics[width=0.5\textwidth]{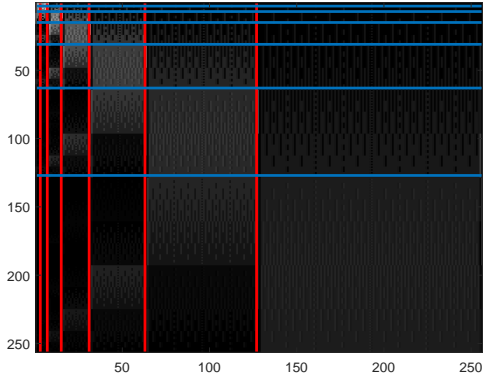}
	\caption{Blocks for the ordering of Walsh functions in blue and wavelets in red}
	\label{fig:OrderingBlocks}
\end{figure}

We are now discussing the ordering of the sampling and reconstruction basis functions. We order the reconstruction functions according to the levels, as in \eqref{Eq:RecSpace1}. With this we get the multilevel subsampling scheme with the level structure. For this sake, we bring the scaling function at level $J_0$ and the wavelet at level $J_0$ together into one block of size $2^{J_0+1}$. The next level constitutes of the wavelets of order $J_0+1$ of size $2^{J_0+1}$ and so forth. Therefore, we define 
\begin{equation}\label{Eq:M}
\mathbf{M} = (M_0, M_1, \ldots, M_r) = (0,2^{J_0+1}, 2^{J_0+2}, \ldots, 2^{J_0+r})
\end{equation}
to represent the level structure of the reconstruction space.
For the sampling space we use the same partition. We only allow by the choice of $q \geq 0$ oversampling. Let
\begin{equation}\label{Eq:N}
\mathbf{N} = (N_0, N_1, \ldots, N_{r-1}, N_r) = (0, 2^{J_0+1}, 2^{J_0 +2}, \ldots, 2^{J_0+r-1}, 2^{J_0+r+q}).
\end{equation}
We then get for the reconstruction matrix $U$ in \eqref{eq:the_U} with $u_{i,j} = \langle \omega_i, \varphi_j \rangle$ that $\omega_j(x) = \Wal(j,x)$ and for the first block we have 
\begin{align}
	\varphi_i &= \phi_{J_0,i} \text{ for } i=0, \ldots, 2^{J_0}-p-1 \text{ and } \\
	\varphi_i &= \phi_{J_0,i}^\# \text{ for } i = 2^{J_0}-p, \ldots, 2^{J_0}-1.
\end{align}
For the next levels, i.e. for $i \geq 2^{J_0}$ we get for $l$ with $2^l \leq i <2^{l+1}$ and $m = i -2^l$ that $\varphi_i = \psi_{l,m}^b$.

The proof of the main theorem relies mainly on the analysis of the change of basis matrix. Numerical examples and rigour mathematics \cite{WalshHaar} show that it is perfectly block diagonal for the Walsh-Haar case. And it is also close to block diagonality for other Daubechies wavelets, which can be seen in Figure \ref{fig:RecMatrix}. We highlight the different parts of the change of basis matrix with respect to the ordering in Figure \ref{fig:OrderingBlocks}. 

An intuition about this phenomena is given in Figure \ref{fig:cancelation}. We plotted Haar wavelets at different scales with Walsh functions at different sequencies. In \ref{fig:cancelation1} the scaling of the Haar wavelet is higher than the sequency of the Walsh function. Therefore, the Walsh function does not change the wavelet on its support and hence it integrates to zero. The next one \ref{fig:cancelation2} shows a wavelet and Walsh function at similar scale and sequency which relates to parts of the change of basis function in the inner block. Here, the two functions multiply nicely to get a non-zero inner product. Last, we have in \ref{fig:cancelation3} that the Walsh functions oscillate faster than the wavelet and hence the Walsh function is not disturbed by the wavelet and can integrate to zero.

\begin{remark}
	The main theorem only holds for the case of one dimensional signals. One reason why it is difficult to extend the results to higher dimensions is the choice of the ordering. It is possible to use tensor products of the one dimensional basis functions to get a basis in higher dimensions. However, this includes that we have to tensor faster oscillating functions with slower oscillating ones. As a consequence one has an exponentially increasing number of non-zero or slow decaying coefficients to consider. This makes the analysis very cumbersome and probably also results in the curse of dimensionality in terms of the relationship of samples and reconstructed coefficients.
\end{remark}

\begin{figure}
	\centering
	\subfloat[Faster oscillating wavelet with $n=5, j=3, m=1$]{\includegraphics[width=0.32\textwidth]{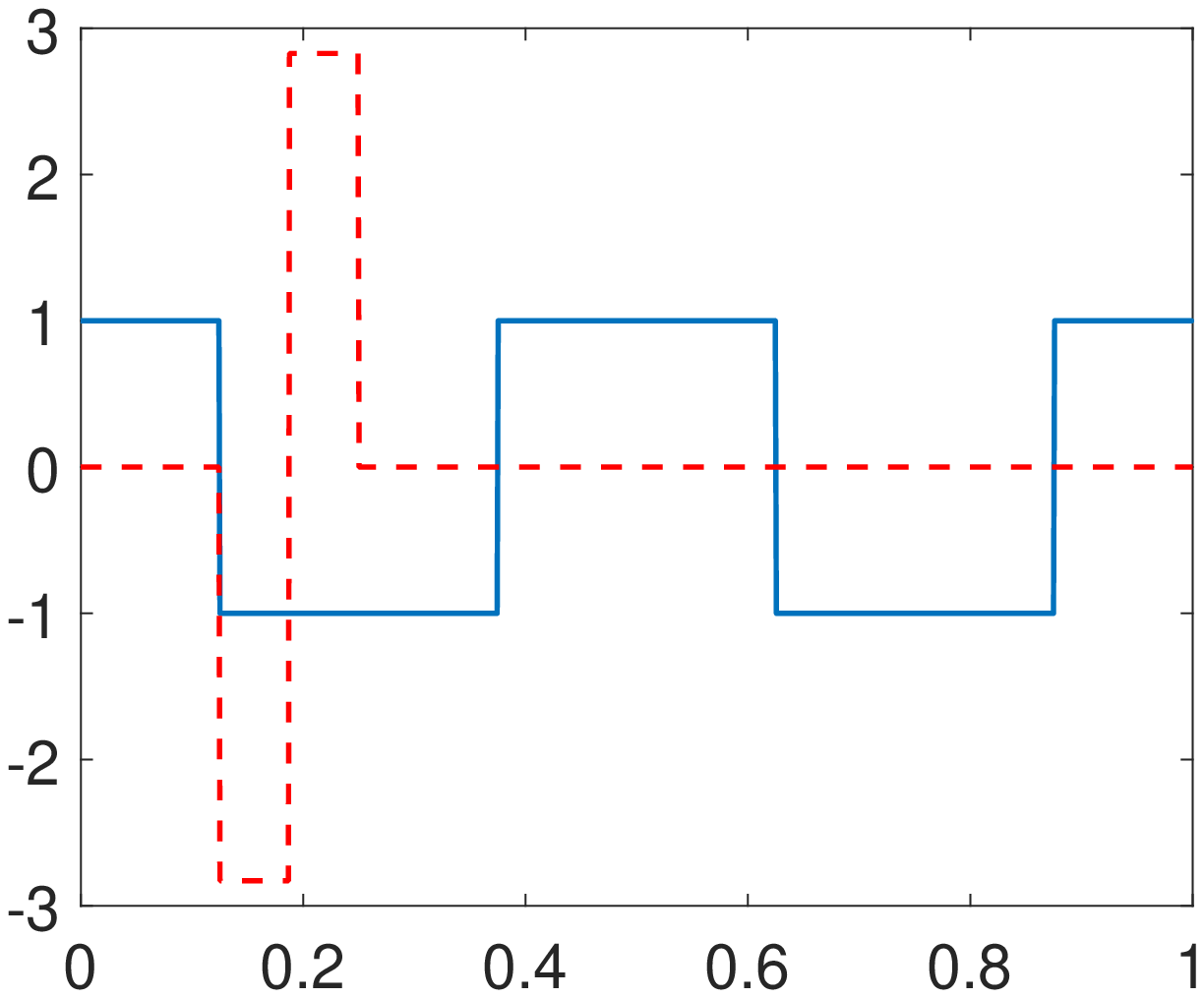}\label{fig:cancelation1}}
	\subfloat[Same oscillation with $n=9, j=3, m=1$]{\includegraphics[width=0.32\textwidth]{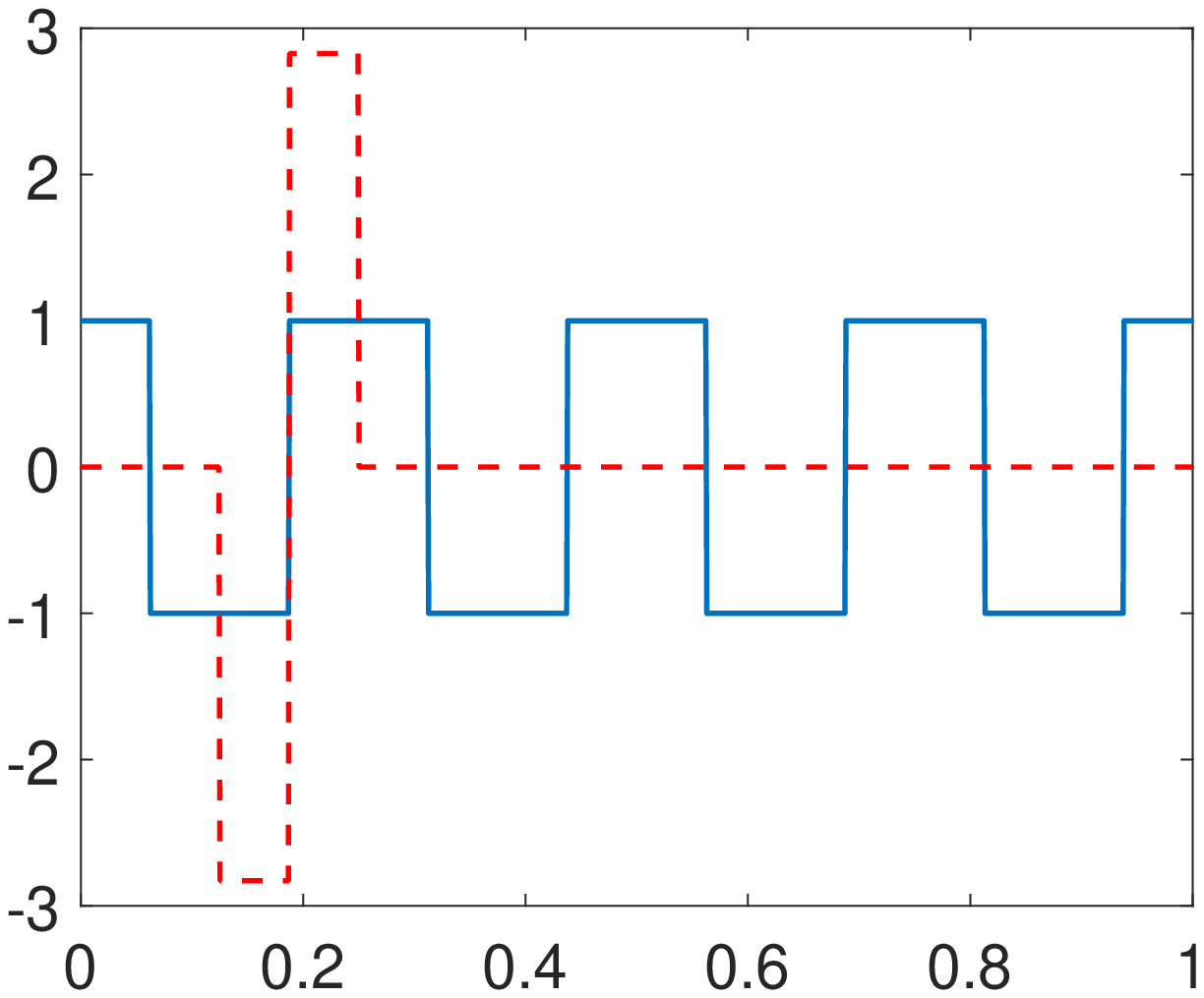}\label{fig:cancelation2}}
	\subfloat[Faster oscillating Walsh function with $n=9, j=2, m=1$ ]{\includegraphics[width=0.32\textwidth]{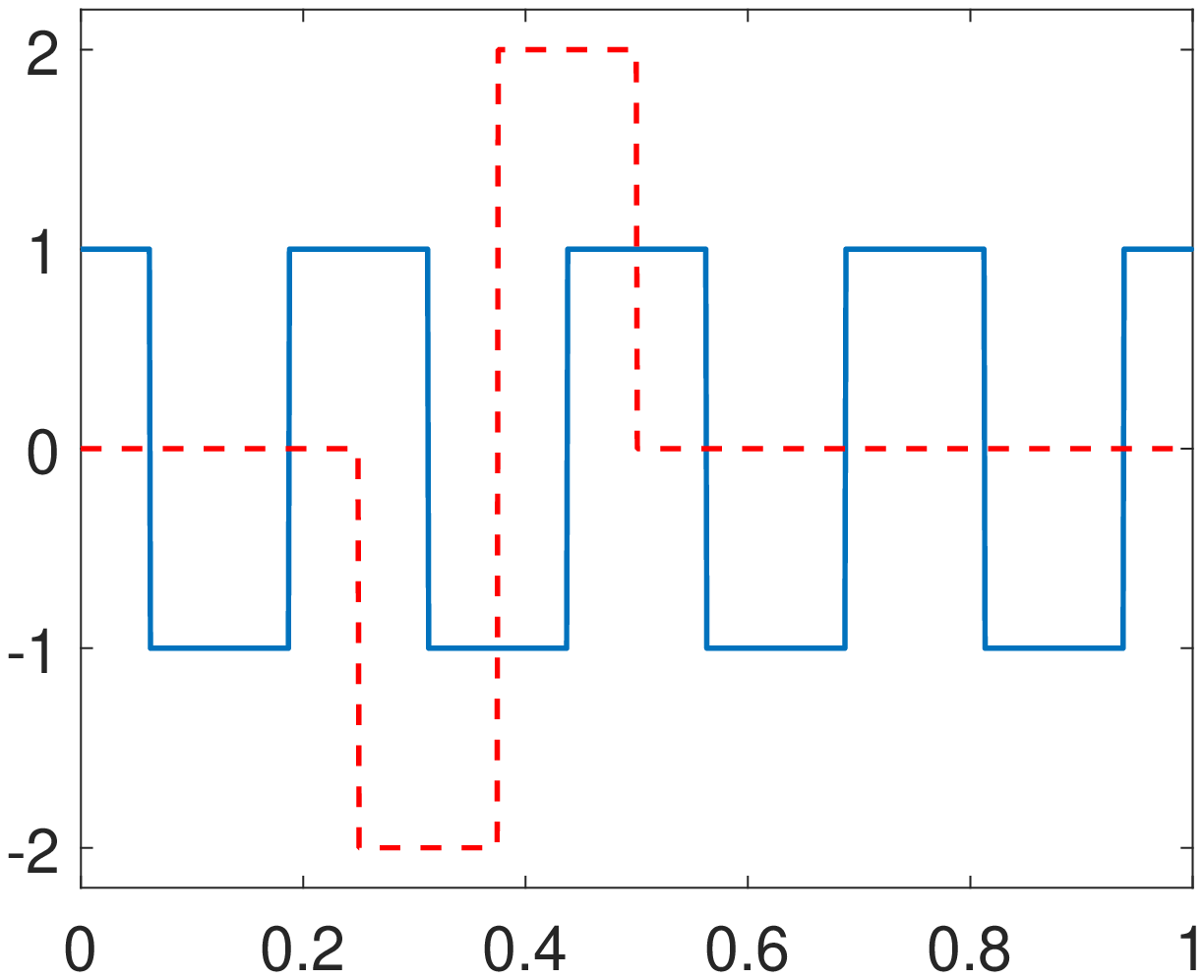}\label{fig:cancelation3}} 
	\caption{Intuition for block diagonal structure of the change of basis matrix with Haar wavelet $\psi_{j,m}$ (red, dashed line) and Walsh functions of order $n$ (blue).}
	\label{fig:cancelation}
\end{figure}

\section{Proof of the Main Result}

The proof of the main theorem relies on the investigation of the change of basis matrix as well as the relative sparsity of signals. With this analysis it is possible to use the results from \cite{breaking} to prove the non-uniform recovery guarantees.

The section is structured as follows: We start with the definition of the analysis tools for the change of basis matrix and the signal. Then, we are able to present Theorem 5.3 from \cite{breaking}. In section "key estimates" we evaluate the necessary analysis values for the Walsh-wavelet case, i.e. the local coherence, relative sparsity and strong balancing property. For the local coherence we use estimates from \cite{VegardUniform}. The analysis of the relative sparsity is related to the Fourier-Haar case in \cite{NoteOnHaarFourier} and the proof techniques of Lemma \ref{Lem:EstimateSSR} are similar to the ones in the main proof of \cite{WalshSSR}. For the final analysis of the relative sparsity also the previous estimate on the local coherence come into play. This is also the case for the estimate of $\tilde M$. Finally, the proof of the strong balancing property follows fast with the results from \S \ref{Ch:RelSparsity}. In \S \ref{Ch:Mainproof} these results are combined to get the main result.

Haar wavelets play a special role in the setting of Walsh functions, as they are structurally very similar. This is the reason why the main theorem can be shortened in this application. This is presented in the final subsection of this section.

\subsection{Preliminaries}
In this section we introduce Theorem 5.3 from \cite{breaking}. To do so we first introduce the definitions of the different elements therein.

We start with the \emph{balancing property}. To be able to solve \eqref{Eq:ReconstructionProblem} it is important that the  uneven finite section of the change of basis matrix is close to an isometry. The balancing property controls the relation between the number of samples and reconstructed coefficients, such that the matrix $P_N U P_M$ is close to an isometry.

\begin{definition}[Def. 5.1\cite{GSinfCS}]
	Let $U \in \mathcal B ( \ell^2(\N))$ be an isometry. Then $N \in \N$ and $K \geq 1$ satisfy the \emph{strong balancing property} with respect to $U,M \in \N$ and $s \in \N$ if 
	\begin{align}\label{Eq:balancingProperty}
	||P_M U^* P_N U P_M - P_M ||_{\ell^\infty \rightarrow \ell^\infty} &\leq \frac 1 8 \left( \log^{1/2}(4\sqrt s KM) \right)^{-1}, \\ ||P_M^\perp U^* P_N U P_M|| _{\ell^\infty \rightarrow \ell^\infty} &\leq \frac 1 8 ,
	\end{align}
	where $|| \cdot ||_{\ell^\infty \rightarrow \ell^\infty}$ is the norm on $\mathcal B (\ell^\infty (\N))$.
\end{definition}

This topic not only arises for the non-linear reconstruction but also for the linear reconstruction.  In the finite-dimensional setting this is assured by the SSR. The SSR has been analysed for different applications, like Walsh-wavelet \cite{WalshSSR}, Fourier-wavelet \cite{linearity, MilanaClarice} and Fourier-shearlet \cite{Ma}.

Next, we use the notation as in \cite{breaking}. Let
\begin{equation}
\tilde{M} = \min \left\{ i \in \N: \max_{k \geq i} ||P_N U e_k ||_2 \leq \frac 1 {32K\sqrt{s}} \right\}. 
\end{equation}

In the rest of the analysis we are interested in the number of samples needed for stable and accurate recovery. This value depends besides known values on the local coherence and the relative sparsity which are defined next. We start with the (global) coherence.

\begin{definition}[Def. 2.1 \cite{breaking}]\label{Def:Coherence}
	Let $U = (u_{i,j})_{i,j=1}^N \in \mathbb C ^{N \times N}$ be an isometry. The \emph{coherence} of $U$ is
	\begin{equation}\label{Eq:DefCoherence}
		\mu(U) = \max_{i,j =1 , \ldots, N} |u_{i,j}|^2
	\end{equation} 
\end{definition}

With this it is possible to define the local coherence for every level band.

\begin{definition}[Def. 4.2 \cite{breaking}]\label{Def:localCoherence}
		Let $U \in \mathcal B ( \ell^2(\N))$ be an isometry. The $(k,l)^{th}$ \emph{local coherence} of $U$ with respect to $\mathbf{M}, \mathbf{N}$ is given by
	\begin{equation}\label{Eq:localCoherencefin}
	\mu_{\mathbf{N},\mathbf{M}}(k,l) = \sqrt{\mu(P_{N_k}^{N_{k-1}}U P_{M_l}^{M_{l-1}}) \cdot \mu(P_{N_k}^{N_{k-1}}U)}, \quad k,l =1, \ldots, r.
	\end{equation}
	We also define
	\begin{equation}\label{Eq:localCoherenceInfin}
	\mu_{\mathbf{N},\mathbf{M}}(k, \infty) = \sqrt{\mu(P_{N_k}^{N_{k-1}}U P_{M_{r-1}}^\perp) \cdot \mu(P_{N_k}^{N_{k-1}}U)}.
	\end{equation}
\end{definition}

The local coherence will be evaluated for the Walsh-wavelet case in Corollary \ref{Cor:localCoherencefin} and Corollary \ref{Cor:LocalCoherenceInf}.

\begin{definition}[Def. 4.3 \cite{breaking}]\label{Def:relativeSparsity}
		Let $U \in \mathcal B ( \ell^2(\N))$ be an isometry and $\mathbf{s} = (s_1, \ldots, s_r) \in \N^r$ and $1 \leq k \leq r$ the $k^{th}$ \emph{relative sparsity} is given by
	\begin{equation}
	S_k = S_k(\mathbf{N},\mathbf{M}, \mathbf{s}) = \max_{\eta \in \Theta} ||P_{N_k}^{N_{k-1}}U \eta ||^2,
	\end{equation}
	with
	\begin{equation}
	\Theta = \left\{ \eta : ||\eta||_\infty\leq 1, | \operatorname{supp}(P_{M_l}^{M_{l-1}}\eta) |= s_l, l =1, \ldots r \right\}.
	\end{equation}
\end{definition}

We devote \S \ref{Ch:RelSparsity} to the analysis of the relative sparsity for our application type. The estimate can be found in Corollary \ref{Cor:relativeSparsity}.

After clarifying the notation and settings we are now able to state the main theorem from \cite{breaking}.

\begin{theorem}[Theo. 5.3 \cite{breaking}]\label{Theo:MainRef}
	Let $U \in \mathcal{B}(\ell^2(\N))$ be an isometry and $x \in \ell^1(\N)$. Suppose that $\Omega = \Omega_{N,m}$ is a multilevel sampling scheme, where $\mathbf{N} = (N_1, \ldots, N_r) \in \N^r$ and $\mathbf m = (m_1, \ldots, m_r) \in \N^r$. Let $(\mathbf s, \mathbf M)$, where $\mathbf{M} = (M_1, \ldots, M_r) \in \N^r$, $M_1 < \ldots < M_r$ and $\mathbf s = (s_1, \ldots, s_r)\in \N^r$, be any pair such that the following holds:
	\begin{enumerate}
		\item \label{Item1:TheoMainRef} The parameters
		\begin{equation}
			N = N_r , \quad K = \max_{k=1, \ldots, r} \left\{ \frac {N_{k}- N_{k-1}} {m_k} \right\},
		\end{equation}
		satisfy the strong balancing property with respect to $U, M :=M_r$ and $s:= s_1 + \ldots + s_r$;
		\item For $\epsilon \in (0, e^{-1}]$ and $1 \leq k \leq r$,
		\begin{equation}\label{Eq:TheoRef1}
			1 \gtrsim \frac{N_k - N_{k-1}}{m_k} \cdot \log (\epsilon^{-1}) \left( \sum_{l=1}^r \mu_{\mathbf{N},\mathbf{M}}(k,l) s_l \right) \cdot \log(K \tilde{M} \sqrt s),
		\end{equation}
		(with $\mu_{\mathbf{N},\mathbf{M}}(k,r)$ replaced by $\mu_{\mathbf{N},\mathbf{M}}(k, \infty)$) and $m_k \gtrsim \hat m _k\log(\epsilon^{-1})\log(K\tilde{M}\sqrt s)$, where $\hat m_k $ is such that
		\begin{equation}\label{Eq:TheoRef2}
			1 \gtrsim \sum_{k=1}^r \left( \frac{N_k - N_{k-1}}{\hat m _k} -1 \right) \cdot \mu_{\mathbf{N},\mathbf{M}} (k,l) \tilde s _k,
		\end{equation}
		for all $l=1, \ldots, r$ and all $\tilde s _1, \ldots, \tilde s _r \in (0, \infty)$ satisfying
		\begin{equation}
			\tilde s _1 + \ldots + \tilde s _r = s_1 + \ldots + s_r, \quad \tilde s _k \leq S_k(\mathbf{N}, \mathbf N, \mathbf s).
		\end{equation}
	\end{enumerate}
	Suppose that $\xi \in \ell^1(\N)$ is a minimizer of \eqref{Eq:ReconstructionProblem}. Then, with probability exceeding $1 - s \epsilon$ ,
	\begin{equation}
		|| \xi - x ||_2 \leq c \cdot \left( \delta \cdot \sqrt K \cdot (1+ L \cdot \sqrt s) + \sigma_{s,M}(f) \right)
	\end{equation}
	for some constant $c$ and $L = c \cdot \left( 1+ \frac{\sqrt{\log(6 \epsilon^{-1})}}{\log (4KM\sqrt s)} \right)$. If $m_k = N_k - N_{k-1}$ for $1 \leq k \leq r$ then this holds with probability $1$.
\end{theorem}

It is a mathematical justification to use structured sampling schemes as discussed in \cite{breaking} in contrast to the first compressed sensing results which promoted the use of random sampling masks \cite{CS1, Donoho}. In the next section we will transfer this result to the Walsh-wavelet case. For this sake we investigate the previously defined values for this case.

\subsection{Key estimates}

In this chapter we discuss the important estimates that are needed for the proof of Theorem \ref{Theo:Main}. They are also interesting for themselves and allow a deeper understanding of the relation between Walsh functions and wavelets.

\subsubsection{Local coherence estimate}

For the local coherence we are interested in the largest value of section of $U$. For this investigation we need to gain insight into the value of $|\langle \varphi, \Wal(k,\cdot) \rangle|$ for $\varphi$ being a Daubechies wavelet and $k \in \N$. Therefore, we start with restating the results about the decay rate of functions under the Walsh transform.

\begin{lemma}[Lem. 4.7 \cite{WalshSSR}]\label{Lemma:Decayrate}
	Let $\phi$ be the mother scaling function of order $p \geq 3$ and $\phi^\#$ be its reflected version. Moreover, let $\psi$ be the corresponding mother wavelet. Then we have that $C_\varphi = \sup_{t \in \R} |\varphi'(t)|$ exist and 
	\begin{equation}
	|\reallywidehat[W]{\phi_i}(z)| \leq  \frac {C_{\phi}} {z} \quad ,\quad |\reallywidehat[W]{\phi_i^\#}(z)| \leq  \frac {C_{\phi^\#}} {z} \quad \text{ and } \quad |\reallywidehat[W]{\psi}(z)| \leq  \frac {C_{\psi}} {z}  .
	\end{equation}
	We denote by $C_{\phi,\psi}$ the maximum of $\left\{C_\phi, C_{\phi^\#},C_\psi, \right\}$.
\end{lemma}

Remark that Lemma 4.7 is stated only for inputs of the type $\frac j L + m$, where $L= 2^R$ for some $R \in \N$ and $j =0, \ldots, L-1$. However, the proof lines follow equivalently for general inputs $z \in \R$. Moreover, the constant $C_\varphi$ can be easily reduced from the proof lines. 

This Lemma is not explicitly used in the analysis of the local coherence. However, the next Theorem also relies on this knowledge. Moreover, we will use Lemma \ref{Lemma:Decayrate} for the relative sparsity.

Now, we recall Proposition 4.5 from \cite{VegardUniform} about the local coherence. Note that the local coherence has a different definition in \cite{breaking} and \cite{VegardUniform}. We will continue to use the one from \cite{breaking} and adapt the results from \cite{VegardUniform} accordingly. 

\begin{theorem}[Prop. 4.5 \cite{VegardUniform}]\label{Th:VegardLocalCoherence}
	Let $U$ be the change of basis matrix for Walsh functions and boundary wavelets of order $p \geq 3$ and minimal wavelet decomposition $J_0$. Moreover, let $\mathbf{M}$ and $\mathbf{N}$ as in \eqref{Eq:M} and \eqref{Eq:N}. Then we have that
	\begin{equation}
		\mu(\PNk U\PMl) \leq C_{\mu}2^{-J_0 -k}2^{-|l-k|}
	\end{equation}
	for the constant $C_{\mu}>0$ independent of $k,l$.
\end{theorem}

\begin{remark}
	The definition of the constant $C_\mu$ is not given in detail in \cite{VegardUniform}. However, following the discussion in Lemma 6.5. therein together with the estimate from Lemma \ref{Lemma:Decayrate} we get that
	\begin{equation}
		C_{\mu}=\left(2p~C_{\phi,\psi}\right)^2.
	\end{equation}
\end{remark}

Beside of bounding the first factor in the local coherenc estimate. This theorem can also be used to bound the second factor of the local coherence. Afterwards we combine both results to the estimate of the local coherences in Corollary \ref{Cor:localCoherencefin} and \ref{Cor:LocalCoherenceInf}.

\begin{corollary}\label{Lem:boundOnmuPNU}
	Let $U$ be the change of basis matrix for the boundary Daubechies wavelets and Walsh functions. Moreover, let $\mathbf M$ and $\mathbf N$ be defined by \eqref{Eq:M} and \eqref{Eq:N}. Then we have that
	\begin{equation}
	\mu(P_{N_k}^{N_{k-1}}U) \leq C_{\mu}2^{-(J_0 + k-1)} 
	\end{equation}
\end{corollary}

\begin{proof}
	Recall that 
	\begin{equation}
	\mu(P_{N_k}^{N_{k-1}}U\PMl) = \max_{i=N_{k-1}+1, \ldots,N_k} \max_{j=M_{l-1}+1, \ldots,M_l} |u_{i,j}|^2.
	\end{equation}
	Hence, we get that
	\begin{align}
		\mu(P_{N_k}^{N_{k-1}}U) = \max_{l=1, \ldots,r} \mu(P_{N_k}^{N_{k-1}}U\PMl).
	\end{align}
	Next, we use Theorem \ref{Th:VegardLocalCoherence} to obtain
	\begin{align}
		\max_{l=1, \ldots,r} \mu(P_{N_k}^{N_{k-1}}U\PMl) 
		& \leq \max_{l=1, \ldots,r} \left\{ C_{\mu}2^{-J_0 -k}2^{-|l-k|} \right\} \\
		& =  C_{\mu}2^{-(J_0 + k-1)}.
	\end{align}
\end{proof}

With these two theorems at hand we can now give an estimate for the local coherence.

\begin{corollary}\label{Cor:localCoherencefin}
	Let $\mu_{\mathbf{N},\mathbf{M}}(k,l)$ be as in \eqref{Eq:localCoherencefin}. Then,
	\begin{equation}
		\mu_{\mathbf{N},\mathbf{M}}(k,l) \leq C_{\mu} 2^{-1/2} 2^{-(J_0 + k-1)} 2^{-|k-l|/2}.	
	\end{equation}
\end{corollary}

\begin{proof}
	Combining the estimate from Theorem \ref{Th:VegardLocalCoherence}, the result in Corollary \ref{Lem:boundOnmuPNU} and the definition of the local coherence in \ref{Def:localCoherence} we get 
		\begin{align}
			\mu_{\mathbf{N},\mathbf{M}}(k,l)  
			& = \sqrt{\mu(P_{N_k}^{N_{k-1}}U P_{M_l}^{M_{l-1}}) \cdot \mu(P_{N_k}^{N_{k-1}}U)}\\
			& \leq \left( C_{\mu}2^{-J_0 -k}2^{-|l-k|}\right) ^{1/2} \left(  C_{\mu}2^{-(J_0 + k-1)} \right)^{1/2} \\
			& = C_{\mu} 2^{-1/2} 2^{-(J_0 + k-1)} 2^{-|k-l|/2}.
		\end{align}
\end{proof}

Because of the infinite-dimensional setting we also have to estimate $\mu_{\mathbf{N},\mathbf{M}}(k,\infty)$ from \eqref{Eq:localCoherenceInfin}. This is done in the following Corollary.

\begin{corollary}\label{Cor:LocalCoherenceInf}
	Let $\mu_{\mathbf{N},\mathbf{M}}(k, \infty)$ as in \eqref{Eq:localCoherenceInfin}. Then,
	\begin{equation}
	\mu_{\mathbf{N},\mathbf{M}}(k, \infty) \leq C_{\mu} 2^{-(J_0 + k-1)}2^{-(r-k)/2}.
	\end{equation}
\end{corollary}

\begin{proof}
	We have that
	\begin{equation}
	\mu_{\mathbf{N},\mathbf{M}}(k,\infty) = \sqrt{\mu(P_{N_k}^{N_{k-1}}U P_{M_{r-1}}^\perp) \cdot \mu(P_{N_k}^{N_{k-1}}U)}.
	\end{equation}
	We know from Corollary \ref{Lem:boundOnmuPNU} that $\mu(P_{N_k}^{N_{k-1}}U) \leq C_{\mu}2^{-(J_0 + k-1)}$. Remark that $k<r$. Hence, we have with Theorem \ref{Th:VegardLocalCoherence} and the independence of the oversampling parameter $q$ that 
	\begin{align}
	\mu(P_{N_k}^{N_{k-1}}UP_{M_{r-1}}^\perp) 
	& = \sup_{q \in \R} \mu(P_{N_k}^{N_{k-1}}U P_{M_r}^{M_{r-1}})\\
	&\leq C_{\mu}2^{-J_0 -k}2^{-|r-k|}
	 = C_{\mu} \cdot 2^{-(J_0+r-1)}.	
	\end{align}
	We get 
	\begin{align}
	\mu_{\mathbf{N},\mathbf{M}}(k,\infty)
	& = \sqrt{\mu(P_{N_k}^{N_{k-1}}U P_{M_{r-1}}^\perp) \cdot \mu(P_{N_k}^{N_{k-1}}U)} \\
	& = \sqrt{\mu(P_{N_k}^{N_{k-1}}U P_{M_{r-1}}^\perp)} \cdot\sqrt{ \mu(P_{N_k}^{N_{k-1}}U)} \\
	& \leq C_{\mu}^{1/2} \cdot 2^{-(J_0+r-1)/2} C_{\mu}^{1/2} 2^{-(J_0+k-1)/2} \\
	&=  C_{\mu} 2^{-(J_0 + k-1)}2^{-(r-k)/2}.
	\end{align}
\end{proof}

Note that the same local coherence estimate was found for the Fourier-Haar case in \cite{NoteOnHaarFourier}.

\subsubsection{Relative sparsity estimate}\label{Ch:RelSparsity}

Now, we want to estimate the relative sparsity of the change of basis matrix $U$ in the Walsh-wavelet case. This is important to bound the local sparsity terms $\tilde s_k$ in Equation \eqref{Eq:TheoRef2} in Theorem \ref{Theo:MainRef}. To do so we recall from Definition \ref{Def:relativeSparsity} that
\begin{equation}
	\sqrt{S_k}= \max_{\eta \in \Theta} ||\PNk U \eta||_2.
\end{equation}
With the estimate from \cite{NoteOnHaarFourier} in \S 4.3. we get
\begin{equation}\label{Eq:relSparsityEstimates}
	\max_{\eta \in \Theta} ||\PNk U \eta||_2 \leq \max_{\eta \in \Theta} \sum_{l=1}^r ||\PNk U \PMl||_2||\PMl \eta||_2 \leq \sum_{l=1}^r ||\PNk U \PMl||_2 \sqrt{s_l},
\end{equation}
where we use that $$||\PMl \eta||_2 \leq \sqrt{||\PMl \eta||_0} = \sqrt{s_l}.$$
Hence, we need to bound $||\PNk U \PMl||_2$. For this sake we first bound $||P_N^\perp U P_M ||_2$ and $||\PNk U \PMl||_2$ in a second step in Lemma \ref{Lem:relativeSparsityHelp}. That is then finally used to bound the relative sparsity in Corollary \ref{Cor:relativeSparsity}.
\begin{lemma}\label{Lem:EstimateSSR}
	Let $U$ be the change of basis matrix for the Walsh measurements and boundary wavelets of order $p \geq 3$. Let the number of samples $N$ be larger than the number of reconstructed coefficients $M$. Then we have that
	\begin{equation}
		||P_N^\perp U P_M ||^2_2 \leq C_{rs} \cdot \left(\frac M N\right),
	\end{equation}
	where $C_{rs} = (16p-8)^2\max \left\{ C_\phi^2,C_{\phi^\#}^2\right\}$ is dependent on the wavelet.
\end{lemma}

\begin{proof}
	We start with bounding $||P_N^\perp U P_M ||_2 $. We rewrite it as follows
	\begin{equation}
		||P_N^\perp U P_M ||_2  =  \sup _{\varphi \in \RM} ||P_{\SN}^\perp \varphi||_2.
	\end{equation}
	This value gets smaller if $N$ grows in relation to $M$. However, from a practical perspective it is desirable to take as few samples $N$ with in contrast a large number $M$. For the further analysis we define the fraction of these two by $S = \frac M N$.
	
	We include for completeness the intermediate steps, which are similar to the proof of the main theorem in \cite{WalshSSR}. However, we believe that this allows us to give a deeper understanding. Especially, the constant $C_{rs}$ is interesting to understand and to see what impacts its size.
	
	We first use the MRA property to rewrite $\varphi \in \RM$ as the sum of the elements in the related scaling space. Take in mind at this point that we only consider values of $M = 2^R$. Hence, we only jump from level to level. We have from \eqref{Eq:proof1varphi} for $\varphi \in \RM$ with $||\varphi||_2=1$
	\begin{equation}\label{Eq:proof1varphi}
	\varphi = \sum\limits_{l=0}^{2^R-p-1} \alpha_l \phi_{R,n} + \sum\limits_{l=2^R -p}^{2^R-1} \beta_l \phi_{R,n}^\# \text{ with } \sum\limits_{l=0}^{2^R-p-1} |\alpha_n|^2 + \sum\limits_{l = 2^R -p}^{2^R-1} |\beta_n|^2 = 1 .
	\end{equation}
	This reduces the problem of the sum of the inner products for the orthogonal projection from a lot of different wavelets to shifted scaling function at the same level. We have
	\begin{align}
		P_{\SN}^\perp \varphi 
		& = \sum_{k > N} \langle \Wal(k,\cdot),\varphi \rangle \\
		& = \sum_{k > N} \langle \Wal(k,\cdot), \sum\limits_{l=0}^{2^R-p-1} \alpha_l \phi_{R,n} + \sum\limits_{l=2^R -p}^{2^R-1} \beta_l \phi_{R,n}^\# \rangle \\
		& = \sum_{k > N} \sum\limits_{l=0}^{2^R-p-1} \alpha_l \langle \Wal(k,\cdot),   \phi_{R,n} \rangle + \sum_{k > N} \sum\limits_{l=2^R -p}^{2^R-1} \beta_l \langle \Wal(k,\cdot) , \phi_{R,n}^\# \rangle.
	\end{align}
	Hence, we start with controlling the inner product $\langle \Wal(k,\cdot), \phi_{R,n} \rangle$ and analogously $\langle \Wal(k,\cdot),\phi_{R,n}^\# \rangle$. Our aim is to remove the scaling and the shift from the wavelet and get instead the product between the Walsh transform of the original mother wavelet and a Walsh polynomial. For this we follow the ideas in \cite{WalshSSR}. Remember first, from the discussion in \S \ref{sec:wavelets} that the mother scaling function is divided into the sum of functions that are supported in $[0,1]$, i.e.	$\phi(x) = \sum\limits_{i=-p+2}^p \phi_i (x-i+1)$ with $\phi_i(x) = \phi(x+i-1) \mathcal{X}_{[0,1]}(x)$
	and hence 
	\begin{equation}
		\langle \Wal(k,\cdot), \phi_{R,n} \rangle = \sum_{i=-p+2}^p \langle \Wal(k,\cdot), \phi_{i,R,n} \rangle.
	\end{equation}
	This allows us to only deal with $\langle \Wal(k,\cdot), \phi_{i,R,n} \rangle$. We get with a variable change
	\begin{align}
		\langle \Wal(k,\cdot), \phi_{i,R,n} \rangle 
		& = 2^{-R/2}\int\limits_{2^{-R}(n+i-1)}^{2^{-R}(n+i)} \phi_i (2^R x - n -i +1) \Wal( k, x)dx \\
		& = 2^{-R/2} \int\limits_0^1 \phi_i (x)\Wal( k, 2^{-R}(x+n+i-1))dx.
	\end{align}
	Next, we use Lemma \ref{Cor:scalingProp} to get the shift out of the integral. We then only deal with the integral between the mother wavelet and the Walsh function independent on the summation factor $n$. For this we have to make sure that the second input is positive. We define $p_R : \mathbb Z \rightarrow \N$ to map $z$ to the the smallest integer with $p_R(z)2^R + z > 0$. This allows us to use Lemma \ref{Cor:scalingProp} because $x \in [0,1]$ and $n + i -1 +2^R p_R(n+i-1) \in \N$. We get
	\begin{align}
		& 2^{-R/2} \int\limits_0^1 \phi_i (x)\Wal( k, 2^{-R}(x+n+i-1))dx \\ 
		& = 2^{-R/2} \Wal(k, 2^{-R}(n+i-1 + 2^R p_R(i-1))) \int\limits_0^1 \phi_i(x)\Wal( k, 2^{-R}x)dx \\
		& = 2^{-R/2} \Wal(n+i-1 + 2^R p_R(i-1), \frac{k}{2^{R}})\reallywidehat[W]{\phi_i}(\frac{ k}{2^R}).
	\end{align}
	With this we are able to represent the inner product of every shifted version of $\phi_{i,R,n}$ with the Walsh function as product of the Walsh transform of $\phi_i$ and a Walsh function which contains the shift information. In the following we want to rewrite the inner products such that we are left with a Walsh polynomial and the Walsh transform of the mother scaling function. For this define
	\begin{align}
			\Phi_i(z) & = \sum\limits_{n=0}^{2^R-p-1} \alpha_n \Wal(n+i-1 + 2^R p_R(i-1),z) \text{ and }\\
			\Phi_i^\# (z) &= \sum\limits_{n = 2^R -p}^{2^R-1} \beta_n \Wal(n + i -1+ 2^R p_R(i -1),z).
	\end{align}
	We get
	\begin{align}
		\sum\limits_{n=0}^{2^R-p-1}  \alpha_n  \langle \phi_{i,R,n} , \Wal( k, \cdot) \rangle  =  2^{-R/2} \reallywidehat[W]{\phi_i}(\frac{ k}{2^R}) \Phi_i(\frac{ k}{2^R})
	\end{align}
	and
	\begin{equation}
		\sum\limits_{n = 2^R -p}^{2^R-1}  \beta_n  \langle \phi_{i,R,n}^\# , \Wal( k, x) \rangle  
		=  2^{-R/2} \reallywidehat[W]{\phi_i^\#}(\frac{ k}{2^R}) \Phi_i^\#(\frac{ k}{2^R}).
	\end{equation}
	After this evaluation we can go back to estimate the norm of $||P_N^\perp U P_M||_2$. We have  
	\begin{align}
		||P_N^\perp U P_M||_2
		& = \sup_{\varphi \in \RM} || P_{\SN}^\perp \varphi|| \\
		&= || P_{\mathcal{S}_N}^\perp
		(\sum\limits_{n=0}^{2^R-p-1} \alpha_n \sum\limits_{i=-p+2}^p \phi_{i,R,n} + \sum\limits_{n=2^R-p}^{2^R-1} \beta_n \sum\limits_{i=-p+2}^p \phi_{i,R,n}^\#) ||_2\\
		& \leq \sum\limits_{i=-p+2}^p || P_{\mathcal{S}_N}^\perp
		(\sum\limits_{n=0}^{2^R-p-1} \alpha_n \phi_{i,R,n} + \sum\limits_{n=2^R-p}^{2^R-1} \beta_n \phi_{i,R,n}^\#) ||_2 \\
		& = \sum\limits_{i=-p+2}^p \sqrt{\sum\limits_{k > N} 2^{-R} \left| \reallywidehat[W]{\phi_i}(\frac k {2^R}) \Phi_i (\frac k {2^R}) + \reallywidehat[W]{\phi^\#_i}(\frac k {2^R}) \Phi_i^\# (\frac k {2^R}) \right|^2},
	\end{align}
	where we used the presentation of $\varphi$ from \eqref{Eq:proof1varphi} and the Cauchy-Schwarz inequality in the third line.
	After multiplying out the brackets we are left with 
	\begin{align}\label{Eq:PartsOfNorm}
		\sum\limits_{k > N} 2^{-R} &\left| \reallywidehat[W]{\phi_i}(\frac k {2^R}) \Phi_i (\frac k {2^R}) \right|^2 \text{, } \sum\limits_{k > N} 2^{-R} \left| \reallywidehat[W]{\phi_i^\#}(\frac k {2^R}) \Phi_i^\# (\frac k {2^R}) \right|^2 \text{ and } \\
		& \sum\limits_{k > N} 2^{-R+1} \left| \reallywidehat[W]{\phi_i}(\frac k {2^R}) \Phi_i (\frac k {2^R}) \reallywidehat[W]{\phi_i^\#}(\frac k {2^R}) \Phi_i^\# (\frac k {2^R}) \right|
	\end{align}
	Because $\phi$ and $\phi^\#$ share the same decay rate, it is sufficient to only deal with \eqref{Eq:PartsOfNorm} and deduce the rest from it.
	To estimate these values, we use the one-periodicity of the Walsh functions. For this sake let $M=2^R$. We always reconstruct a full level as we do not know in which part of the level the information is located. 
	
	It can be observed that we divide $k$ in both functions by $2^R$. We want to separate the input $k/2^R$ into its integer and rational part. This allows us to use the one-periodicity of the Walsh function and to work for the decay rate of the wavelet under the Walsh function only with the major integer part. For this we replace $k = mM +j$, where $j = 0, \ldots, M-1$ and $m \geq S = N/M$.  
	This leads to
	\begin{align}\label{Eq:k=mL+j}
	\sum\limits_{k \geq N} 2^{-R} \left| \reallywidehat[W]{\phi_i}(\frac k {2^R}) \Phi_i (\frac k {2^R}) \right|^2 
	& \leq   \sum\limits_{j=0}^{M-1} \frac{1}{M} \left| \Phi_i(\frac{ j}{M}) \right|^2 \sum\limits_{m \geq S} \left|\reallywidehat[W]{\phi_i}\left(\frac j M + m\right)\right|^2 .
	\end{align}
	We estimate with Lemma \ref{Lemma:Decayrate} and the integral bound for series
	\begin{align}\label{Eq:EstimateSum2}
	\sum\limits_{m \geq S} \left|\reallywidehat[W]{\phi_i}\left(\frac j M + m\right)\right|^2 
	& \leq \sum\limits_{m \geq S} \frac {C_\phi^2}{ m^{2}} 
	\leq C_\phi^2 \left(\frac {1}{ S^{2}} +  \int\limits_{S}^\infty \frac {1}{x^{2}}dx \right) 
	\leq C_\phi^2\left(\frac 1 {S^{2}}+ \frac{1}{ S}\right) 
	\leq \frac{2C_\phi^2}{ S}.
	\end{align}
	Here $C_\phi$ depends on the choice of the wavelet. In contrast to the Fourier case there is no known relationship between the smoothness of the wavelet and the decay rate or the behaviour of $C_\phi$, as discussed in Remark \ref{Rem:RelationFourier}. 
	
	To estimate the sum over the Walsh polynomials we use Lemma \ref{Lemma2DphiSum} and \eqref{Eq:proof1varphi}. We get similar to computations in \cite{WalshSSR}, that
	\begin{align}\label{Eq:PhiRep}
	\Phi_i(z) & = \sum\limits_{n=0}^{2^R-p} \alpha_n \Wal(n + i -1 +  2^R p_R(i-1),z) \\ 
	& = \sum\limits_{n = i -1 +  2^R p_R(i-1)}^{2^R - p + i - 1 + 2^R p_R(i-1)} \alpha_{n -i +1 -2^R p_R(i-1)} \Wal(n,z) 
	\end{align}
	and hence
		\begin{equation}\label{Eq:BoundPhi}
	\sum\limits_{j=0}^{M-1} \frac 1 M \left| \Phi_i (\frac j M) \right|^2 = \sum\limits_{n =  i -1 +  2^R p_R(i-1)}^{ 2^R -p + i - 1 +  2^R p_R(i-1)} |\alpha_{n -i +1 -2^R p_R(i-1)}|^2 = \sum\limits_{n =  0}^{ 2^R -p } |\alpha_{n }|^2 \leq 1.
	\end{equation} 
	The analogue holds true for the $\phi^\#$ part. We again replace $k = mM +j$ and obtain
	\begin{align}\label{Eq:k=mL+j2}
	\sum\limits_{k \geq N} 2^{-R} \left| \reallywidehat[W]{\phi_i^\#}(\frac k {2^R}) \Phi_i (\frac k {2^R}) \right|^2 
	& \leq   \sum\limits_{j=0}^{M-1} \frac{1}{M} \left| \Phi_i^\#(\frac{ j}{M}) \right|^2 \sum\limits_{m \geq S} \left|\reallywidehat[W]{\phi_i^\#}\left(\frac j M + m\right)\right|^2 .
	\end{align}
	With the same reasoning as in \eqref{Eq:EstimateSum2} we get
	\begin{align}\label{Eq:EstimateSum22}
	\sum\limits_{m \geq S} \left|\reallywidehat[W]{\phi_i^\#}\left(\frac j M + m\right)\right|^2 
	\leq \frac{2C_\phi^2}{ S}.
	\end{align}
	Finally, for the Walsh polynomial we have
	\begin{equation}\label{Eq:BoundPhi2}
	\sum\limits_{j=0}^{M-1} \frac 1 M \left| \Phi_i^\# (\frac j M) \right|^2 = \sum\limits_{n = 2^R -p + i -1 +  2^R p_R(i-1)}^{ 2^R -1 + i - 1 +  2^R p_R(i-1)} |\beta_{n -i +1 -2^R p_R(i-1)}|^2 = \sum\limits_{n =  2^R - p}^{ 2^R -1 } |\beta_{n }|^2 \leq 1.
	\end{equation} 
	We get together
	\begin{align}
		||P_N^\perp U P_M ||_2 
		& \leq   \sum\limits_{i=-p+2}^p \left( \sum\limits_{k \geq M} 2^{-R} \left| \reallywidehat[W]{\phi_i}(\frac k {2^R}) \Phi_i (\frac k {2^R}) \right|^2 + \sum\limits_{k \geq M} 2^{-R} \left|\reallywidehat[W]{\phi^\#_i}(\frac k {2^R}) \Phi_i^\# (\frac k {2^R}) \right|^2 \right. \\
		& \left. + 2 \left( \sum\limits_{k \geq M} 2^{-R} \left| \reallywidehat[W]{\phi_i}(\frac k {2^R}) \Phi_i (\frac k {2^R}) \right|^2 \right)^{1/2} \left( \sum\limits_{k \geq M} 2^{-R} \left| \reallywidehat[W]{\phi^\#_i}(\frac k {2^R}) \Phi_i^\# (\frac k {2^R}) \right|^2 \right)^{1/2} \right)^{1/2}. \\
		& \leq \sum\limits_{i=-p+2}^{p} \left( \frac {2C_\phi^2} {S} + \frac{{2C_{\phi^\#}}^2}{S} + 4 \frac {C_\phi C_{\phi^\#} }{S} \right)^{1/2} 
		 \leq \frac 8 {S^{1/2}} (2p-1)  \max \left\{ C_\phi^2,{C_{\phi^\#}}^2\right\}^{1/2}.
	\end{align}
	When we now replace $S = N/M$ and set $C_{rs} = (16p-8)^2\max \left\{ C_\phi^2,{C_{\phi^\#}}^2\right\}$ we get
	\begin{equation}
		||P_N^\perp U P_M ||_2^2 \leq C_{rs} ~ \frac M N.
	\end{equation}
\end{proof}

In the next Lemma we estimate $||\PNk U \PMl||_2^2$ using the previous result.

\begin{lemma}\label{Lem:relativeSparsityHelp}
	Let $U$ be the change of basis matrix given by the Walsh measurements and boundary wavelets of order $p \geq 3$. Then we have that 
	\begin{equation}
		||\PNk U \PMl||_2^2 \leq C_{\max} \cdot 2^{-|l-k|+1},
	\end{equation}
	where $C_{\max} = \max\left\{ C_{\mu}, C_{rs} \right\}$. 
\end{lemma}

\begin{proof}
	We use similar estimates as in Corollary \ref{Cor:localCoherencefin}. We know from Lemma \ref{Lem:EstimateSSR} that $||P_N^\perp U P_M ||_2^2 \leq C_{rs} \cdot \left(\frac M N\right)$, whenever $N \geq M$. With this we get for $k > l$:
		\begin{align}
			||P_{N_k}^{N_{k-1}} U P_{M_l}^{M_{l-1}}||_2^2 
			&\leq ||P_{N_{k-1}}^\perp U P_{M_l}||_2^2 
			\leq C_{rs} \left(\frac {M_l} {N_{k-1}} \right) \\
			&= C_{rs} \left(2^{(J_0 + l)} \cdot 2^{-(J_0+k-1)} \right) 
			= C_{rs} \cdot 2^{(l-k)+1} = C_{rs} \cdot 2^{-|l-k|+1},
		\end{align}
	where the first inequality enlarges the section of the matrix and with it the norm. This allows to use Lemma \ref{Lem:EstimateSSR}. The rest follows from the definition of $M_l$ and $N_{k-1}$ in \S \ref{Ch:Ordering}.
	For $l \geq k$ we get from Theorem \ref{Th:VegardLocalCoherence}
		\begin{equation}
			\max_{i=N_{k-1}+1, \ldots,N_k} \max_{j=M_{l-1}+1, \ldots,M_l} |u_{i,j}|^2  = 
			\mu(P_{N_k}^{N_{k-1}} U P_{M_l}^{M_{l-1}})
						\leq  C_{\mu} 2^{-(J_0+l-1)}.
		\end{equation}
	With this we conclude 
		\begin{align}
			||\PNk U \PMl||_2^2 
			& = \sup_{z \in \mathbb{C}^{2^{(J_0+l-1)}}, ||z||_2 =1} \sum_{i=N_{k-1}+1}^{N_k} \sum_{j =M_{l-1}+1}^{M_l} |u_{i,j} z_j|^2 \\
			& \leq \sup_{z \in \mathbb{C}^{2^{(J_0+l-1)}}, ||z||_2 =1} \max_{i=N_{k-1}+1, \ldots,N_k} \max_{j=M_{l-1}+1, \ldots,M_l} |u_{i,j}|^2 \sum_{i=N_{k-1}+1}^{N_k}  \sum_{j =M_{l-1}+1}^{M_l} |z_j|^2 \\
			& \leq\sup_{z \in \mathbb{C}^{2^{(J_0+l-1)}}, ||z||_2 =1} C_{\mu} \cdot 2^{-(J_0+l-1)}~2^{J_0 + k}  \sum_{j =M_{l-1}+1}^{M_l} |z_j|^2 \\
			& \leq C_{\mu} 2^{(J_0 + k) -(J_0 +l-1)} =C_{\mu} 2^{(k-l)+1} = C_{\mu}  2^{-|k-l|+1}.
		\end{align}
	Both equations combined lead to the desired estimate with $C_{\max} = \max \left\{ C_{\mu}, C_{rs} \right\}$.
\end{proof}

With this Lemma at hand we can now bound the relative sparsity $S_k(\mathbf N, \mathbf M, \mathbf s)$.

\begin{corollary}\label{Cor:relativeSparsity}
	For the setting as before we have
	\begin{equation}
		S_k (\mathbf N, \mathbf M, \mathbf s) \leq 2C_{\text{geo}}C_{\max} \sum_{l=0}^{r-1} 2^{-|k-l|/2} s_l.
	\end{equation}
\end{corollary}

\begin{proof}
	With the estimates from and Equation \eqref{Eq:relSparsityEstimates} we get
	\begin{align}
		S_k 
		& = \left( \sum_{l=1}^r ||\PNk U \PMl||_2 \sqrt{s_l} \right)^2
		= 2 C_{\max}  \left(\sum_{l=1}^r 2^{-|k-l|/2} \sqrt{s_l}\right)^2 \\
		&\leq 2 C_{\max} \sum_{l=1}^r 2^{-|k-l|/2} \sum_{l=1}^r 2^{-|k-l|/2} s_l 
		\leq 2 C_{\text{geo}} C_{\max} \sum_{l=1}^r 2^{-|k-l|/2} s_l.
	\end{align}
	The third inequality follows from the Cauchy-Schwarz inequality and the last line from the fact the notation of $\sum_{l=1}^r 2^{-|k-l|/2}= C_k \leq C_{\text{geo}}$ for all $k$.
\end{proof}

\subsubsection{Bounding $\tilde M$}

In the estimate of Theorem \ref{Theo:MainRef} we have the value $\tilde{M}$. We aim to bound this value to reduce the number of free parameters. Hence, we estimate 
\begin{equation}
\tilde{M} = \min \left\{ i \in \N: \max_{m \geq i} ||P_N U e_m ||_2 \leq \frac 1 {32K\sqrt{s}} \right\}.
\end{equation}
We start with the following calculation with $m = 2^{(J_0+n)} = M_{n} \geq N$
\begin{align}
||P_N U e_m||_2 = \left( \sum_{i=1}^N |u_{i,m}|^2 \right)^{1/2} \leq \left( C_{\mu}  N \cdot 2^{-(J_0+n)}\right)^{1/2}.
\end{align}
Hence, for $2^{(J_0 +n)} \geq C_{\mu} \cdot N \cdot \left( 32 K \sqrt{s} \right)^2$ we have
\begin{equation}
||P_N U e_m ||_2 \leq \frac 1 {32K \sqrt{s}}.
\end{equation}
Therefore, 
\begin{equation}\label{Eq:Mtildebound}
	\tilde{M} \leq C_{\mu} \cdot \lceil N 32^2 K^2 s \rceil.
\end{equation}

\subsubsection{Balancing Property}

In this chapter we show that the first assumption of Theorem \ref{Theo:MainRef} is fulfilled.

For this sake, we use the results from the previous chapter, especially Lemma \ref{Lem:EstimateSSR}. We start with relating a bound on $||P_N^\perp U P_M||_2$ to the relationship between $N$ and $M$.  

\begin{corollary}\label{Lem:SSRWalshTheo}
	Let $\mathcal{S}$ and $\mathcal{R}$ be the sampling and reconstruction space spanned by the Walsh functions and separable boundary wavelets of order $p \geq 3$ respectively. Moreover, let $M = 2^{R}$ with $R \in \mathbb{N}$. Then, we get for all $\theta \in (1, \infty)$ 
	\begin{equation}
	||P_N^\perp U P_M||_2  \leq \theta,
	\end{equation}
	whenever
	\begin{equation}
	N \geq C_{rs} \cdot M \cdot \theta ^{-2}.
	\end{equation}
\end{corollary}

\begin{proof}
	Rewriting $N \geq C_{rs} \cdot M \cdot \theta ^{-2}$ gives us
	\begin{equation}
		\theta^2 \geq C_{rs} \frac M N.
	\end{equation}
	And hence with Lemma \ref{Lem:EstimateSSR} and $\theta >1$ we get
	\begin{equation}
		||P_N^\perp U P_M||_2 \leq \left( C_{rs} \cdot \frac M N \right)^{1/2} \leq \theta.
	\end{equation}
\end{proof}

With this at hand we can now proof the relation between $N,M$ such that the strong balancing property is satisfied. 

\begin{lemma}\label{Lem:BalancingProperty}
	For the setting as before $N,K$ satisfy the strong balancing property with respect to $U,M$ and $s$ whenever $N \gtrsim M^2 (\log(4MK\sqrt s ))$.
\end{lemma}

\begin{proof}
		From Lemma \ref{Lem:SSRWalshTheo} we have that $||P_N^\perp U P_M ||_2 \leq \frac 1 {8 \sqrt{M}} \left(\log^{1/2}(4 KM \sqrt s)\right)^{-1}$ whenever it holds that $N \gtrsim M^{2} \left(\log(4 KM \sqrt s)\right)$. Using additionally that $U$ is an isometry we get
	\begin{align}\label{Eq:HelpInverse}
	& ||P_M U^* P_N U P_M - P_M ||_{\infty} 
	= || P_M U^* P_N U P_M - P_M U^* I U P_M||_{\infty} \\
	& = ||P_M U^* P_N^\perp U P_M ||_{\infty} 
	\leq \sqrt{M} || P_N^\perp U P_M ||_{2} \leq \frac 1 8 \left(\log^{1/2}(4 KM \sqrt s)\right)^{-1}.
	\end{align}
	For the second inequality we have that 
		\begin{align}
	&||P_M^\perp U^* P_N U P_M ||_{\infty}
	= ||P_M^\perp U^* P_N U P_M + P_M^\perp U^* I U P_M||_{\infty} \\
	&= ||P_M^\perp U^* P_N^\perp U P_M||_{\infty} 
	\leq \sqrt{M} || P_N^\perp U P_M ||_{2} \leq  \frac 1 8 \left(\log^{1/2}(4 KM \sqrt s)\right)^{-1} \leq \frac 1 8.
	\end{align}
	The last inequality follows from the fact that $K,M,s$ are integers and therefore $\log(4KM\sqrt s) \geq 1$. Hence, the strong balancing property is fulfilled.
\end{proof}

\subsection{Proof of the main theorem}\label{Ch:Mainproof}

In this chapter we bring the previous results together to proof Theorem \ref{Theo:Main}.

\begin{proof}[Proof of Theorem \ref{Theo:Main}]
	We show that the assumptions of Theorem $5.3.$ in \cite{breaking} are fulfilled. Moreover, we follow the lines of \cite{NoteOnHaarFourier}.
	
	With Lemma \ref{Lem:BalancingProperty} we have that $N,K$ satisfy the strong balancing property with respect to $U,M$ and $s$. Hence, point \eqref{Item1:TheoMainRef} in Theorem \ref{Theo:MainRef} is fulfilled.

	The last two steps show that \eqref{Eq:TheoRef1} and \eqref{Eq:TheoRef2} are fulfilled and Theorem \ref{Theo:MainRef} can be applied.
	We have
	\begin{align}
		& \frac {N_{k}-N_{k-1}}{m_k} \log (\epsilon^{-1}) \left( \sum_{l=1}^r \mu_{\mathbf{N},\mathbf{M}}(k,l) s_l \right) \log (K \tilde M \sqrt s) \\
		& \leq \frac {N_{k}-N_{k-1}}{m_k} \log (\epsilon^{-1}) \log(32^2 C_{\mu} NK^3s^{3/2})\\
		& \left( \sum_{l=1}^{r-1} C_{\mu} 2^{-1/2} 2^{-(J_0 + k-1)} 2^{-|k-l|/2} s_l + C_{\mu} 2^{-(J_0 + k-1)}2^{-(r-k)/2} s_r \right)  \\
		& =  C_{\mu} \frac{\log(\epsilon^{-1})}{m_k} \frac{N_{k}-N_{k-1}}{N_{k-1}} \left( \sum_{l=1}^r 2^{-|k-l|/2} s_l \right) \log(32 C_{\mu} NK^3s^{3/2}),
	\end{align}
	where we used the estimate of $\mu_{\mathbf{N},\mathbf{M}}$ from Corollary \ref{Cor:localCoherencefin} and \ref{Cor:LocalCoherenceInf}, \eqref{Eq:MainTheoAssumption} and \eqref{Eq:Mtildebound}. Moreover, $ C_{\mu}$ is independent of $k,l,\bfM,\bfN,\bfs$. Therefore,
	\begin{equation}
		C_{\mu} \frac{\log(\epsilon^{-1})}{m_k} \frac{N_{k}-N_{k-1}}{N_{k-1}} \left( \sum_{l=1}^r 2^{-|k-l|/2} s_l \right) \log(32 C_{\mu} NK^3s^{3/2}) \lesssim 1
	\end{equation}
	and Equation \eqref{Eq:TheoRef1} is fulfilled. Now, we consider Equation \eqref{Eq:TheoRef2} we directly estimate $\mu_{\bfN,\bfM}(k,l)$ for $k<r$ directly without the $2^{-1/2}$ term to keep the sum notation.
	\begin{align}
		& \sum_{k=1}^r \left( \frac{N_k - N_{k-1}}{\hat m _k} -1 \right) \mu_{\mathbf{N},\mathbf{M}}(k,l) \tilde s _k \\
		& \leq \sum_{k=1}^r \left( \frac{N_k - N_{k-1}}{\hat m _k} \right) C_{\mu} 2^{-(J_0+k-1)}  2^{-|k-l|/2} \tilde s _k \\
		& = C_{\mu} \frac{N_k - N_{k-1}}{N_{k-1}} \sum_{k=1}^r \frac {\tilde s _k} {\hat m _k} 2^{-|l-k|/2} \\
		& \leq C_{\mu} \sum_{k=1}^r \frac {\tilde s _k} {\hat m _k} 2^{-|l-k|/2}
	\end{align}
	 Due to the fact that the geometric series is bounded we have
	 \begin{equation}
	 	\sum_{k=1}^r 2^{-|l-k|/2} \leq C_{\text{geo}}, \quad \text{ for all }l=1,\ldots,r.
	 \end{equation}
	 We are left with bounding $\tilde{s_k}/\hat{m}_k$ for all $k=1,\ldots,r$. Denote the constant from $\lesssim$ in Theorem \ref{Theo:MainRef} by $C$. With the estimate in Corollary \ref{Cor:relativeSparsity} we can then bound $\tilde{s_k}$ with \eqref{Eq:MainTheoAssumption} by
	 \begin{align}\label{Eq:tildesk}
	 	\tilde{s_k} 
	 	& \leq S_k(\bfN,\bfM,\bfs) \leq C_\text{geo} C_{\mu} \cdot 2 \sum_{l=0}^{r-1} 2^{-|k-l|/2} s_l \\
	 	&\leq 2C_\text{geo} C_{\mu} C m_k \cdot \frac{N_{k-1}}{N_{k}-N_{k-1}} \left( \log(\epsilon^{-1}) \log (K^2 s N )\right)^{-1}  \\
	 	& \leq 2 C_\text{geo} C_{\mu} C  \frac{N_{k-1}}{N_{k}-N_{k-1}} \hat{m}_k = 2C_\text{geo} C_{\mu} C \hat{m}_k.
	 \end{align}
	 All together yields	 
	 \begin{align}
	  \sum_{k=1}^r \left( \frac{N_k - N_{k-1}}{\hat m _k} -1 \right) \mu_{\mathbf{N},\mathbf{M}}(k,l) \tilde s _k 
	 \leq 2 C_{\mu}^2 C_{\text{geo}}^2 C  \lesssim 1.
	 \end{align}
	 
\end{proof}

\subsection{Recovery guarantees for the Walsh-Haar case}\label{Ch:WalshHaar}

In this section we pay attention to the Walsh-Haar case. This relationship is of high interest because of the very similar behaviour of Walsh functions and Haar wavelets. As seen earlier this results in perfect block diagonality of the change of basis matrix, see Figure \ref{fig:RecMatrixHaar}. This block diagonal structure has been analysed in \cite{WalshHaar} with a focus on its impact on linear reconstruction methods. In contrast to general Daubechies wavelets it is possible to evaluate the inner product between the Walsh function and Haar wavelet and scaling function exactly.
\begin{lemma}[Lem. 1 \cite{WalshHaar}]\label{Th:HaarWavDec}
	Let $\psi = \mathcal{X}_{[0,1/2]} - \mathcal{X}_{(1/2,1]}$ be the Haar wavelet. Then, we have that 
	\begin{equation}
	|\langle \psi_{R,j}, \Wal(n,\cdot) \rangle| = \begin{cases}
	2^{-R/2} &\quad 2^R \leq n < 2^{R+1}, 0 \leq j \leq 2^R -1 \\
	0 &\quad  \text{otherwise.}
	\end{cases}
	\end{equation}
\end{lemma}
For the scaling function we have
\begin{lemma}[Lem. 2 \cite{WalshHaar}]\label{Lem:HaarScalDec}
	Let $\phi = \mathcal{X}_{[0,1]}$ be the Haar scaling function. Then, we have that the Walsh transform obeys the following block and decay structure
	\begin{equation}\label{Eq:scalingDecay}
	|\langle \phi_{R,j}, \Wal(n,\cdot) \rangle| = \begin{cases}
	2^{-R/2} \quad & n < 2^R, 0 \leq j \leq 2^R -1 \\
	0 \quad &\text{otherwise.}
	\end{cases} 
	\end{equation}
\end{lemma}
The order does not change from the one for general Daubechies wavelets and each level $j$ contains $2^j$ many elements. Due to the structure of the change of basis matrix $U$ in the Walsh-Haar case, the off diagonal blocks do not impact the coherence and sparsity structure at one level. Therefore, the number of samples per level only depends on the incoherence in this given level and the relative sparsity within. With this the main theorem simplifies for the Walsh-Haar case to the next Corollary. 
\begin{corollary}
	Let the notation be as before, but let the wavelet be the Haar wavelet. Moreover, let $\epsilon >0$ and $\Omega = \Omega_{N,m}$ be a multilevel sampling scheme such that: 
	\begin{enumerate}
		\item The number of samples is larger or equal the number of reconstructed coefficients, i.e. $N \geq M$.
		\item Let $K = \max_{k=1, \ldots, r} \left\{ \frac{N_k - N_{k-1}}{m_k} \right\}$, $M=M_r$, $N=N_r$ and $s= s_1 + \ldots + s_r$ and for each $k=1,\ldots,r$:
		\begin{equation}
			m_k \gtrsim \log(\epsilon^{-1})\log(K \sqrt{s} N) \cdot s_k.
		\end{equation}
	\end{enumerate}
	Then, with probability exceeding $1-s\epsilon$, any minimizer $\xi \in \ell^1(\N)$ satisfies
	\begin{equation}
	||\xi - x ||_2 \leq c \cdot \left( \delta \sqrt{K}(1+L \sqrt{s}) + \sigma_{s,M}(f) \right),
	\end{equation}
	for some constant $c$, where $L = c \cdot \left( 1 + \frac{\sqrt{\log(6 \epsilon^{-1})}}{\log(4KM \sqrt{s})} \right)$. If $m_k = N_{k} - N_{k-1}$ for $1 \leq k \leq r$ then this holds with probability $1$.
\end{corollary}

\begin{proof}
	The proof is separated in two parts. We first evaluate the analysis tools for the Walsh-Haar case. Second, we conclude that under our assumptions the requirements of theorem \ref{Theo:MainRef} are fulfilled, i.e. the balancing propery, equation \eqref{Eq:TheoRef1} and \eqref{Eq:TheoRef2}.
	
	For the analysis of the balancing property, let $M=2^R$ and $\varphi \in \RM$ be represented as $\varphi = \sum_{j =0}^{2^R-1} \alpha_j \phi_{R,j}$ with $\sum_{j =0}^{2^R-1} |\alpha_j|^2 = 1$. Then we have for $N \geq M$ that
	\begin{equation}
		||P_N^\perp U P_M||_2^2 =  \max_{\varphi \in \RM} \sum_{k>N} |\langle \Wal(k,\cdot), \varphi \rangle|^2 =  \sum_{k>N} |\langle \Wal(k,\cdot), \sum_{j =0}^{2^R-1} \alpha_j \phi_{R,j} \rangle|^2 = 0
	\end{equation}
	because $k > N \geq M = 2^R$ such that Lemma \ref{Lem:HaarScalDec} applies. Hence, the balancing property is satisfied for any $K$ and with it assumption \eqref{Item1:TheoMainRef} in \ref{Theo:MainRef} is fulfilled. 
	
	Next, we analyse $\tilde M$. 
	\begin{equation}
	\tilde{M} = \min \left\{ i \in \N: \max_{m \geq i} ||P_N U e_m ||_2 \leq \frac 1 {32K\sqrt{s}} \right\}
	\end{equation}
	we have for $m = 2^R + j$ with $j \leq 2^R-1$ that
	\begin{equation}
		||P_N U e_m ||_2^2 = \sum_{k \leq N} | \langle \Wal(k,\cdot), \phi_{R,j} \rangle|^2 = \begin{cases}
			N2^{R/2} \quad & 2^R < N \\
			0 	\quad &	2^R > N.
		\end{cases}
	\end{equation}
	We obtain that the minimal $i$ is achieved for $m = 2^{R+1}$ and therefore $\tilde{M} =2N$. Similar to the above discussion we get that $\mu_{\bfN,\bfM} (k,l) = 0$ for $k \neq l$. This removes the sum in \eqref{Eq:TheoRef1} in theorem \ref{Theo:MainRef}.
	\begin{align}
	& \frac {N_{k}-N_{k-1}}{m_k} \log (\epsilon^{-1}) \left( \sum_{l=1}^r \mu_{\mathbf{N},\mathbf{M}}(k,l) s_l \right) \log (K \tilde M \sqrt s) \\
	& \leq \frac {N_{k}-N_{k-1}}{m_k} \log (\epsilon^{-1}) \left( 2^{-(J_0 + k-1)} s_k \right) \log(2K\sqrt{s}N) \\
	& = \frac {N_k - N_{k-1}}{N_{k-1}} \frac{\log(\epsilon^{-1})}{m_k} s_k \log(2K\sqrt{s}N) \\
	& = \frac{\log(\epsilon^{-1})}{m_k} s_k \log(2K\sqrt{s}N) \lesssim 1.
	\end{align}
	For the second equation \eqref{Eq:TheoRef2} we get with the general estimate of the relative sparsity in Equation \eqref{Eq:tildesk}
	\begin{align}
		& \sum_{k=1}^r \left( \frac{N_k - N_{k-1}}{\hat m _k} -1 \right) \mu_{\mathbf{N},\mathbf{M}}(k,l) \tilde s _k \\
		& \leq \left( \frac{N_l - N_{l-1}}{\hat m_l} \right) 2^{-(J_0 +l-1)} \tilde s _l \\
		& \leq \frac{N_l - N_{l-1}}{N_{l-1}}\frac {\tilde s _l} {\hat m _l} \leq C \lesssim 1.
	\end{align}
	With this we have seen that all assumptions of theorem \ref{Theo:MainRef} are fulfilled and the recovery is guaranteed.
\end{proof}

\section{Numerical experiments}\label{Ch:Numerics}

\begin{figure}
	\subfloat[Continuous original function $f$ from Equation \eqref{Eq:OrigCont}]{
	\includegraphics[width=0.48\textwidth]{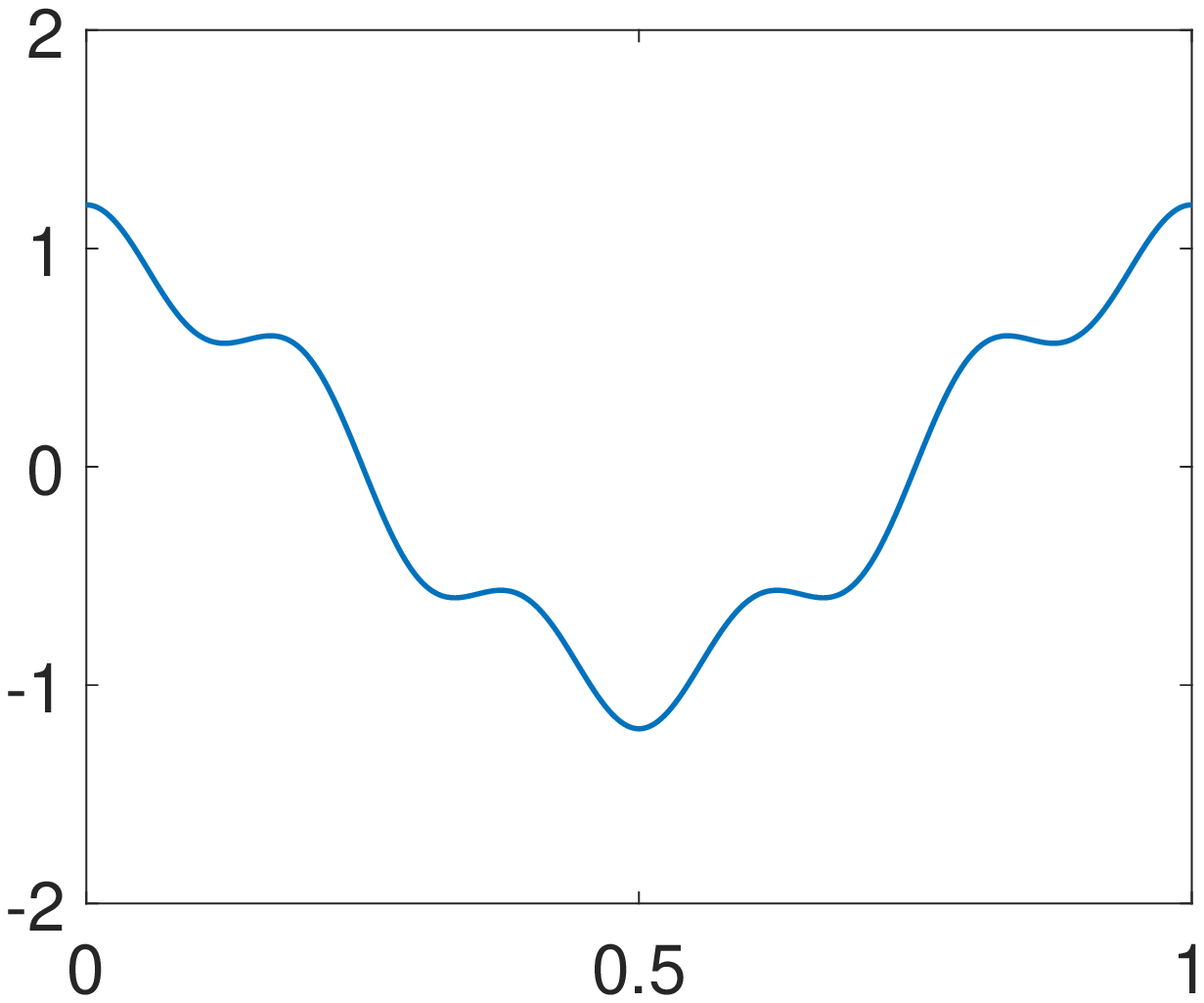}\label{fig:OrigCont}}
	\subfloat[Discontinuous original function $g$ from Equation \eqref{Eq:OrigDisc}]{
		\includegraphics[width=0.48\textwidth]{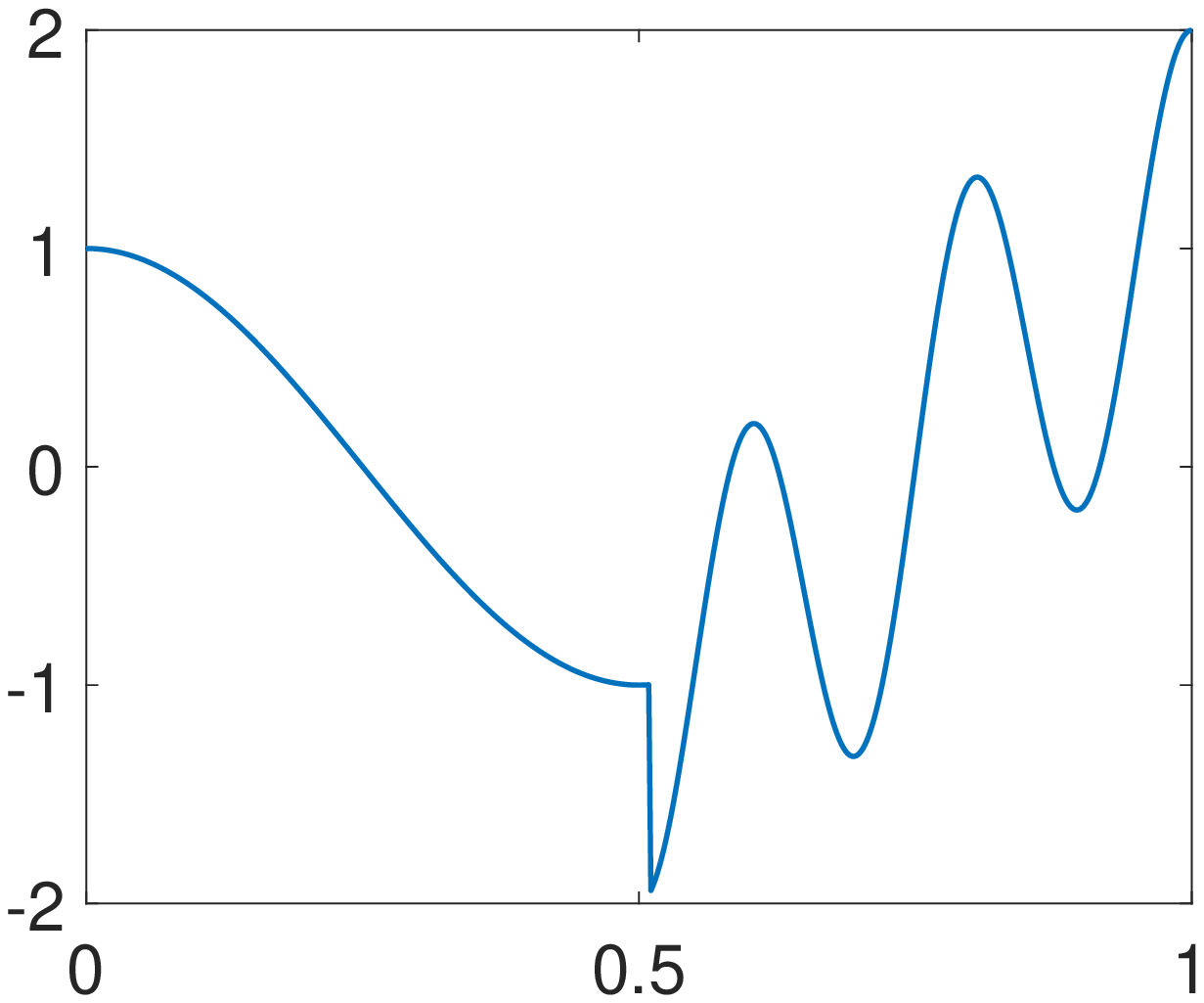}\label{fig:OrigDisc}}
	\caption{Original functions}
	\label{fig:Orig}
\end{figure}

\begin{figure}
	\subfloat[Sampling pattern with $N=2^8$ and $|m|=256$]{
		\includegraphics[width=0.48\textwidth]{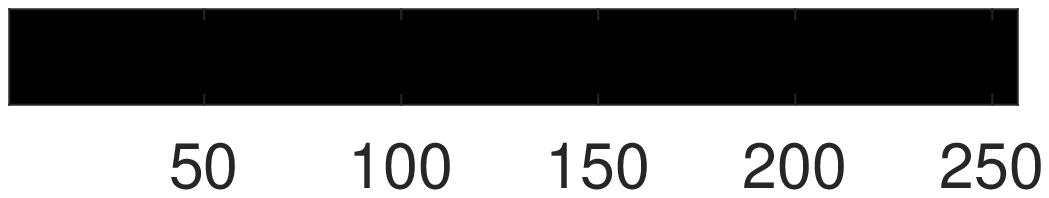}
		\label{fig:samplingN8}}
	\subfloat[Sampling pattern with $N=2^9$ and $|m|=256$]{
		\includegraphics[width=0.48\textwidth]{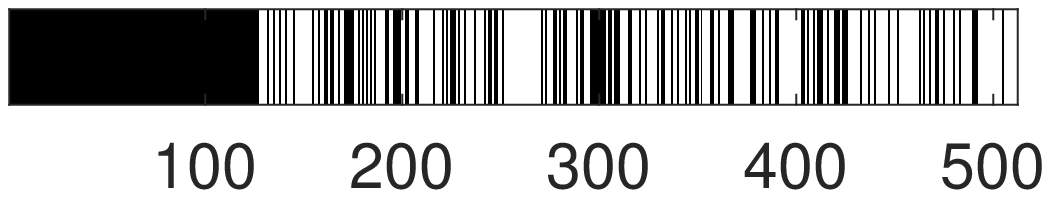}
		\label{fig:samplingN9}}
	\quad
		\subfloat[Sampling pattern with $N=2^{10}$ and $|m|=256$]{
		\includegraphics[width=0.48\textwidth]{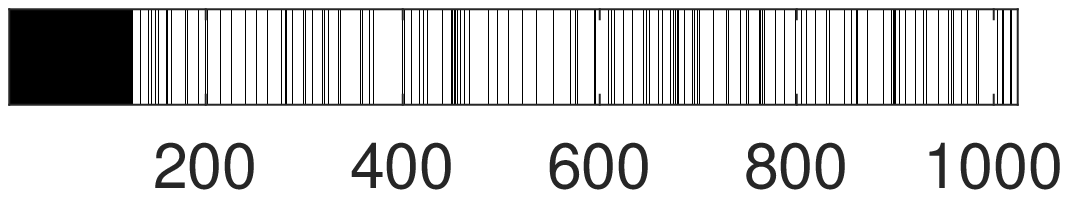}
		\label{fig:samplingN10}}
	\subfloat[Sampling pattern with $N=2^6$ and $|m|=64$]{
		\includegraphics[width=0.48\textwidth]{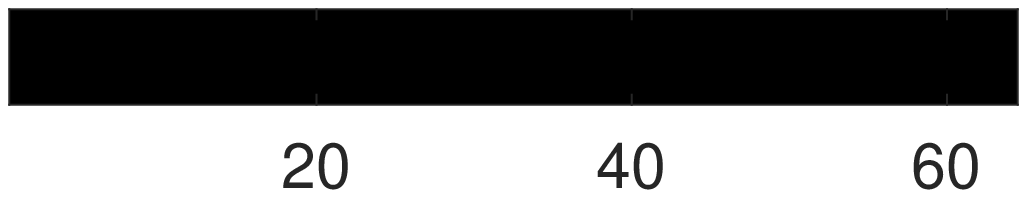}
		\label{fig:samplingN6Cont}}
	\quad
	\subfloat[Sampling pattern with $N=2^7$ and $|m|=64$]{
		\includegraphics[width=0.48\textwidth]{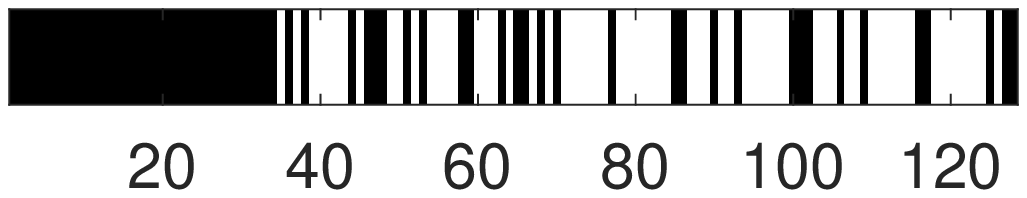}
		\label{fig:samplingN7Cont}}
	\subfloat[Sampling pattern with $N=2^8$ and $|m|=64$]{
		\includegraphics[width=0.48\textwidth]{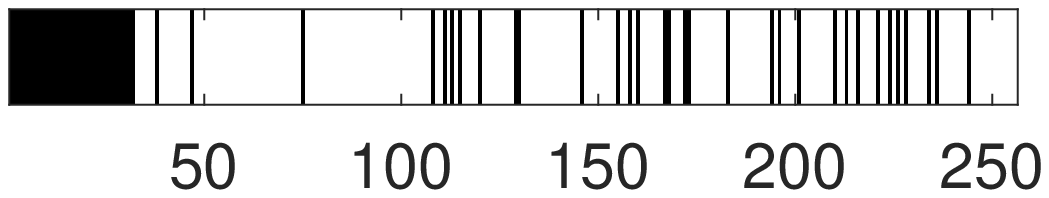}
		\label{fig:samplingN8Cont}}
	\caption{Sampling pattern used in the experiments in Figure \ref{fig:RecDiscont} and \ref{fig:NMTWCScont}, the samples are taken in the black area.}
	\label{fig:samplingPattern}
\end{figure}

\begin{figure}
	\centering
	\subfloat[CS reconstruction with $N = 2^8$]{
		\includegraphics[width=0.48\textwidth]{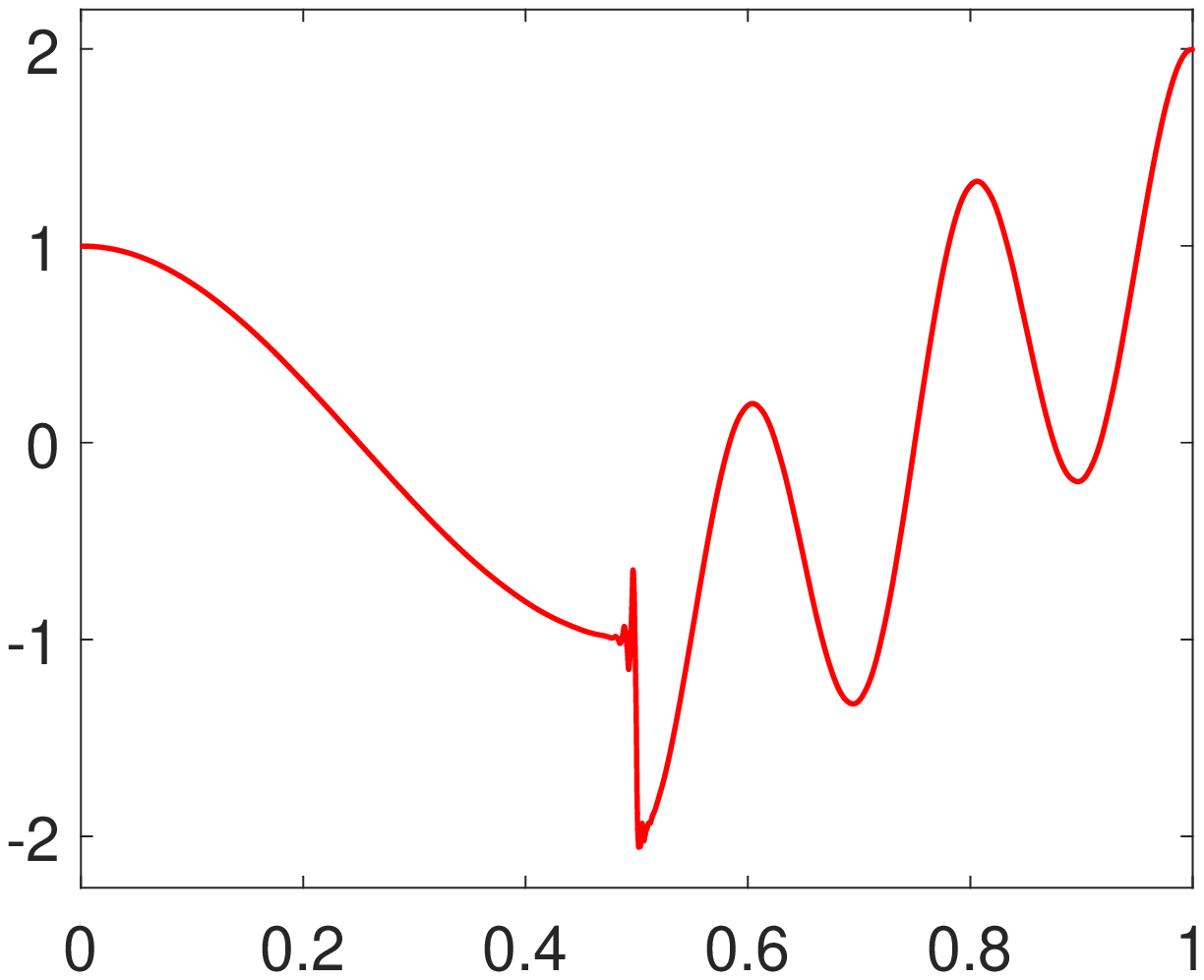}
		\label{fig:CSDisc8}}
	\subfloat[Reconstructed wavelet coefficients for $N=2^8$]{
		\includegraphics[width=0.48\textwidth]{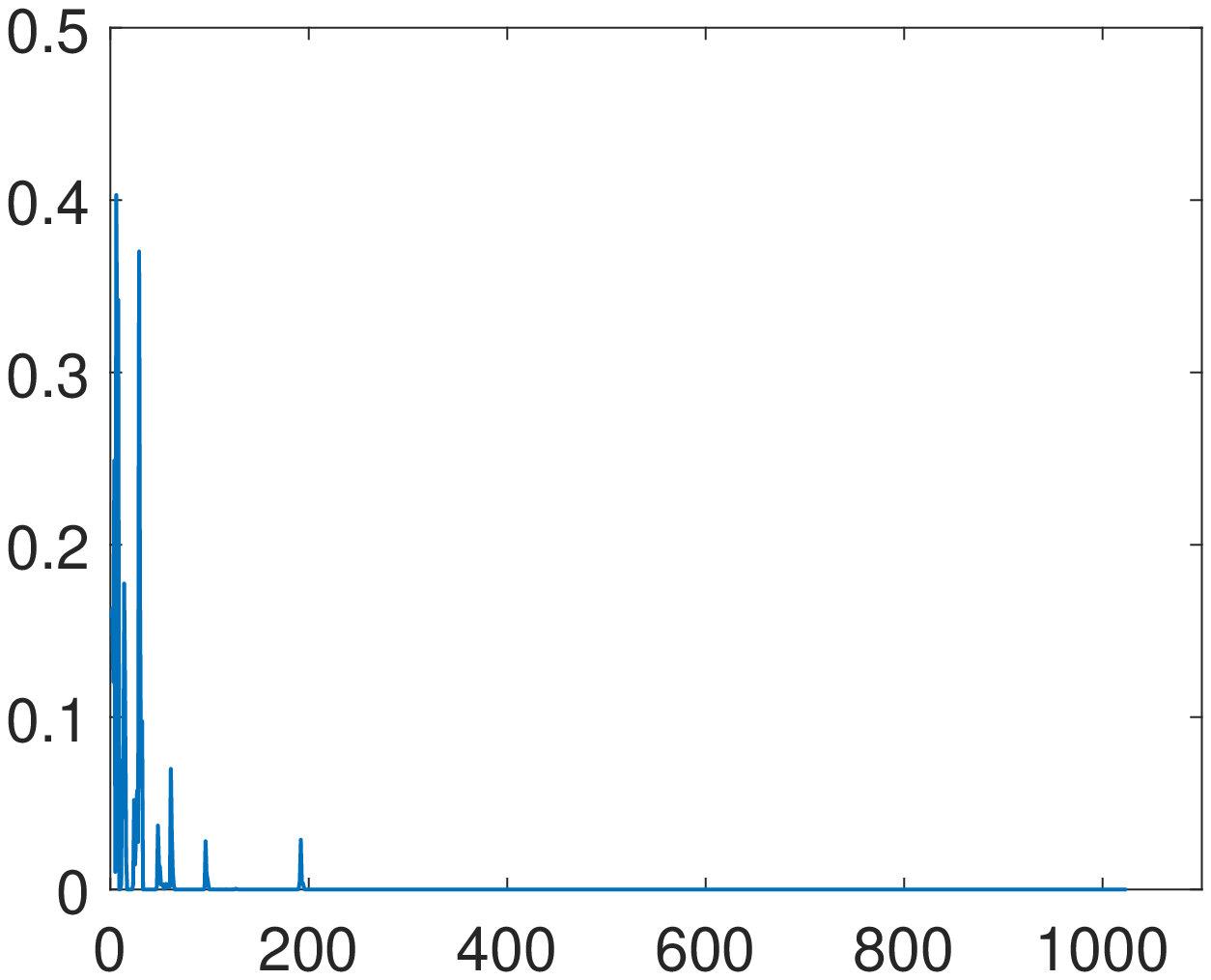}
		\label{fig:RecWaveCoeff8}}
	\quad
	\subfloat[CS reconstruction with $N = 2^9$]{
		\includegraphics[width=0.48\textwidth]{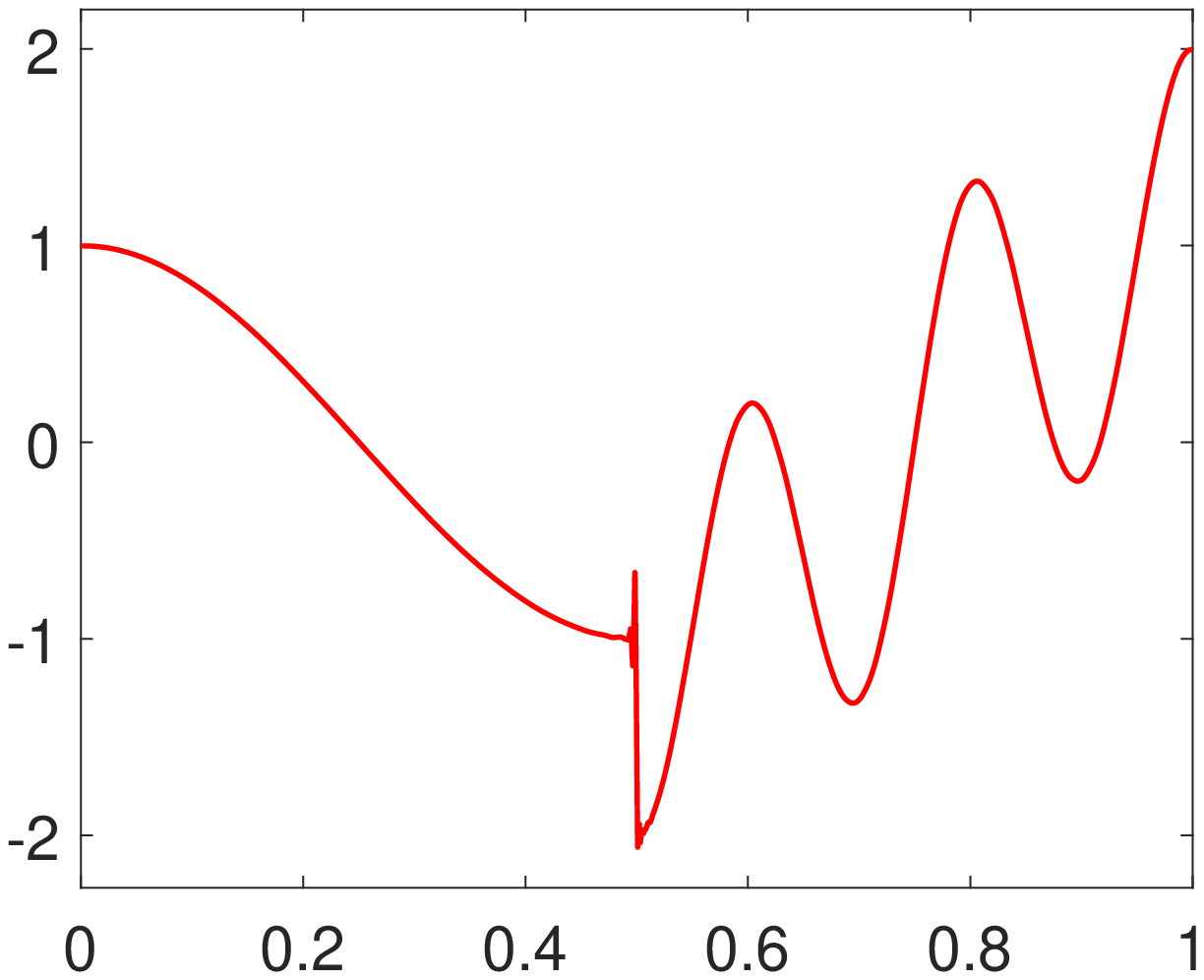}
		\label{fig:CSDisc9}}
	\subfloat[Reconstructed wavelet coefficients for $N=2^9$]{
		\includegraphics[width=0.48\textwidth]{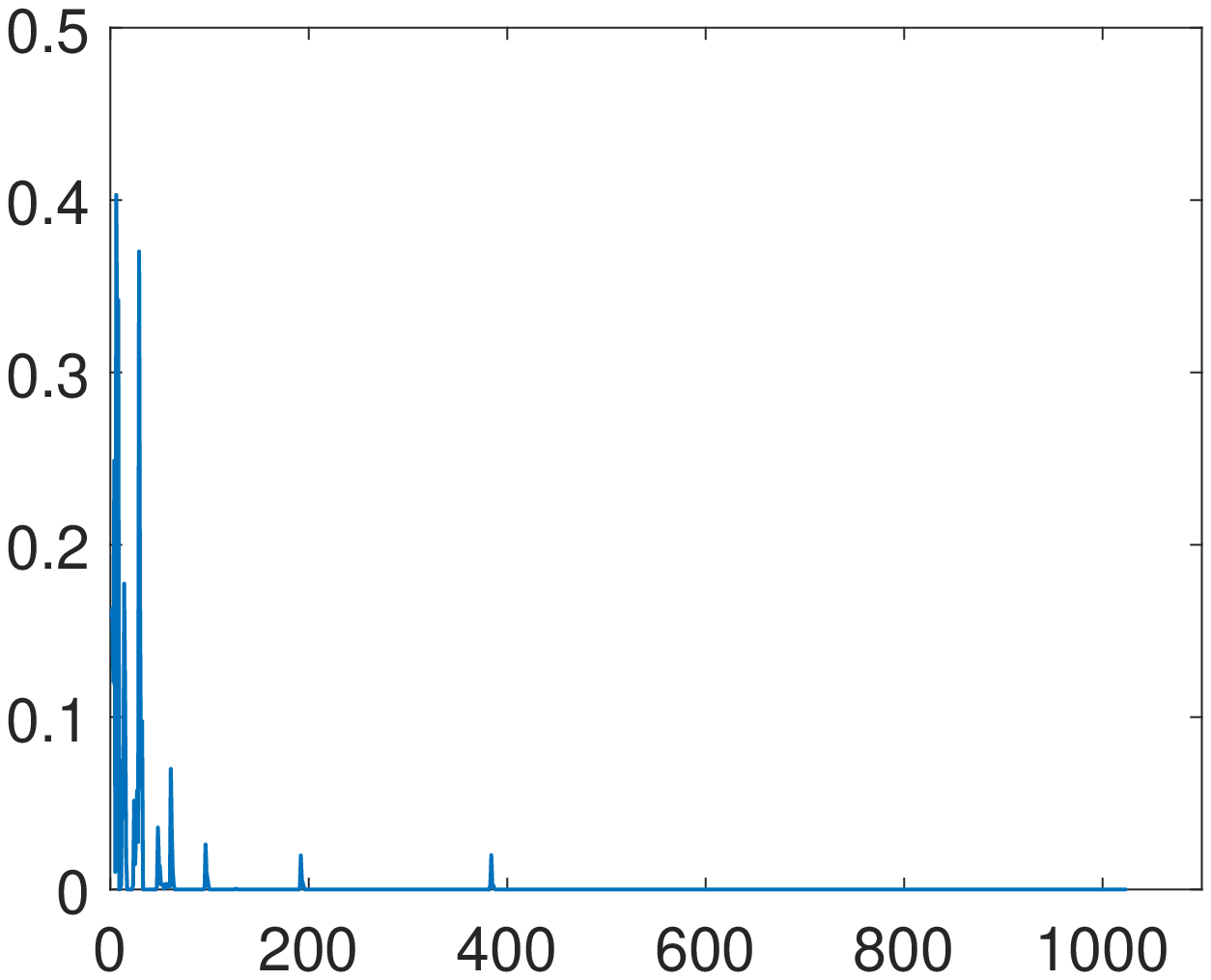}
		\label{fig:RecWaveCoeff9}}
	\quad
	\subfloat[CS reconstruction with $N = 2^{10}$]{
		\includegraphics[width=0.48\textwidth]{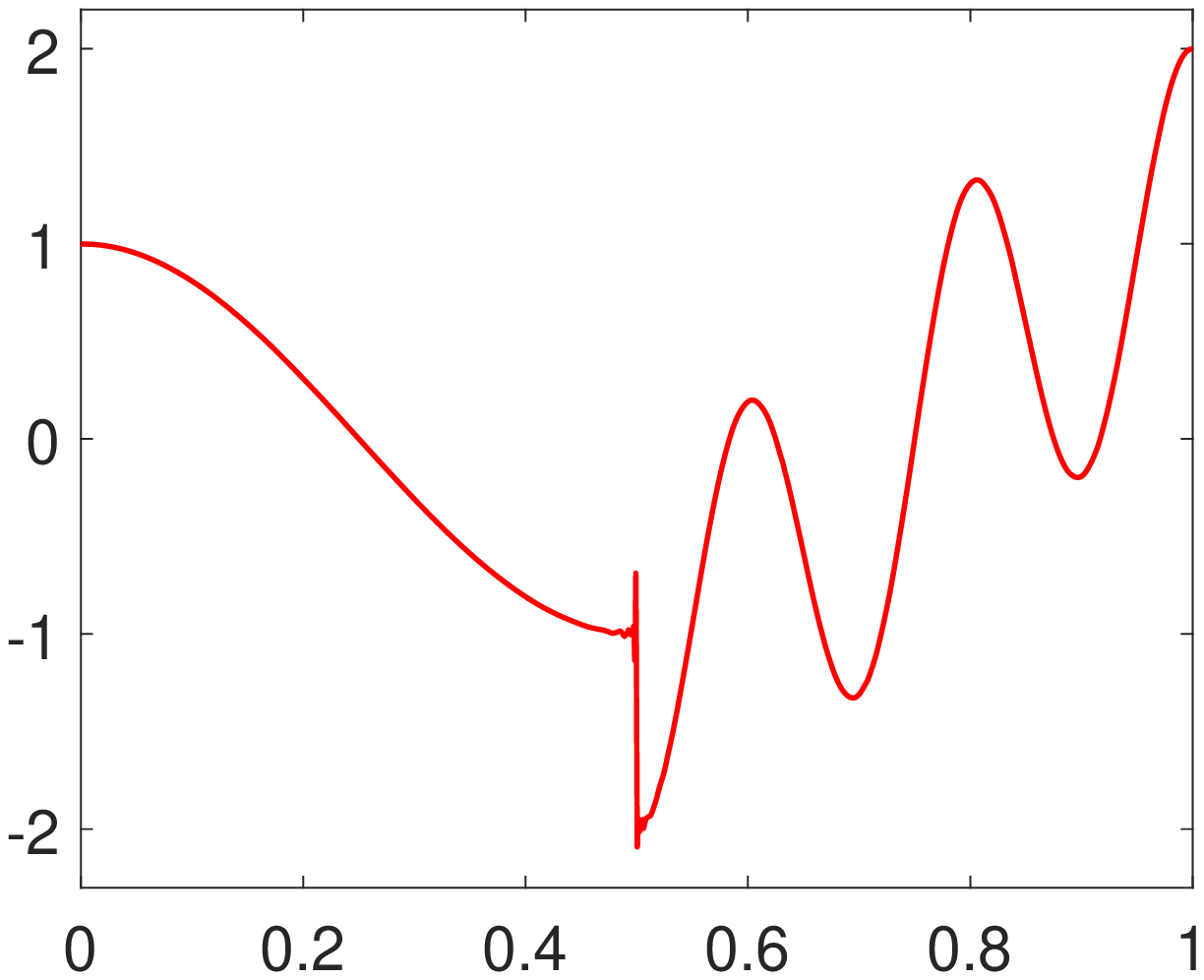}
		\label{fig:CSDisc10}}
	\subfloat[Reconstructed wavelet coefficients for $N=2^{10}$]{
		\includegraphics[width=0.48\textwidth]{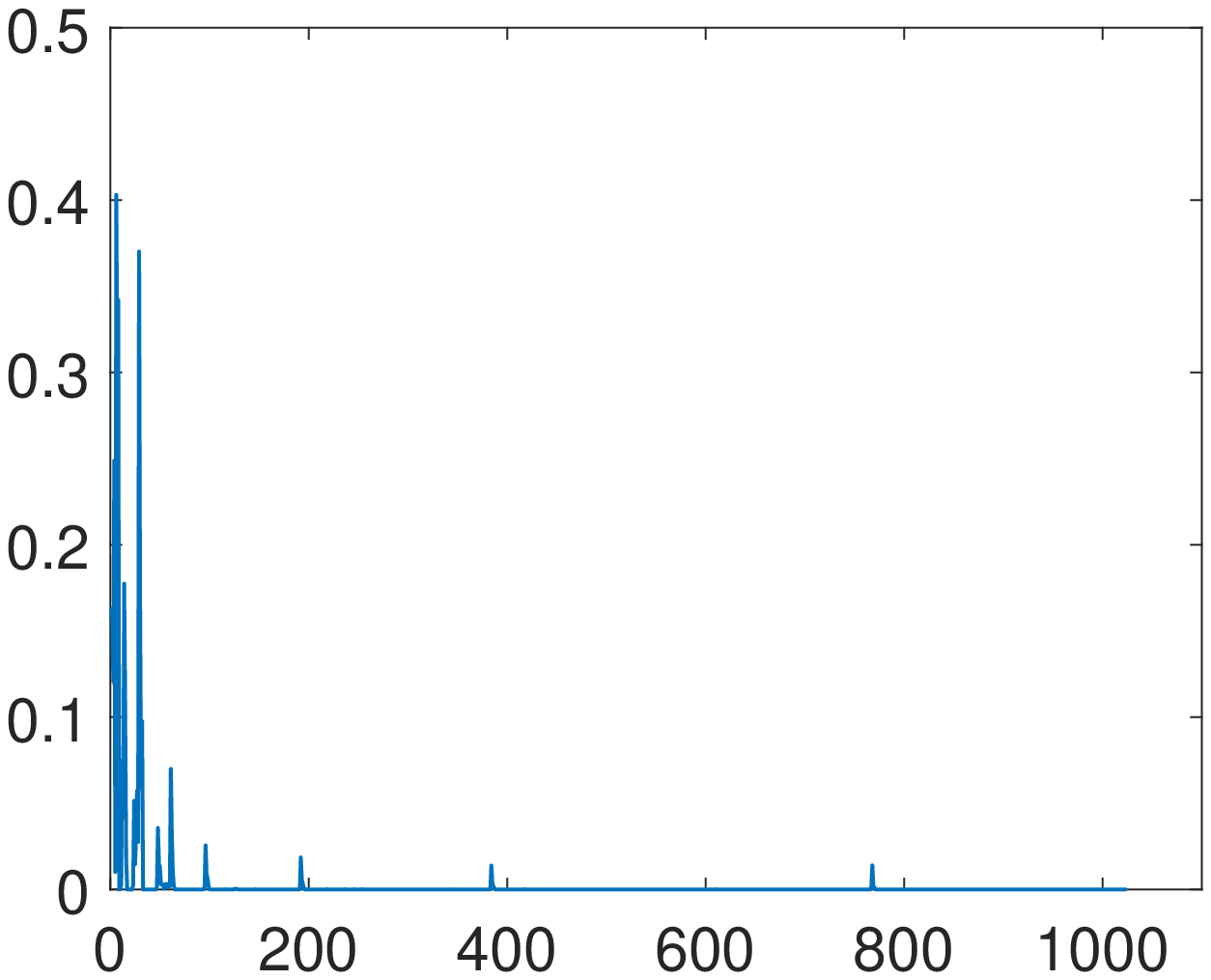}
		\label{fig:RecWaveCoeff10}}
	\caption{Reconstruction of the signal $g$ with $|m| = 256$ and varying choices of $N$ to visualize the impact of $N$ on the maximal wavelet coefficients.}	
	\label{fig:RecDiscont}
\end{figure}

\clearpage

\begin{figure}
	\subfloat[CS reconstruction of $f$ with $N = 2^{6}$ and $|m| = 64$]{
		\includegraphics[width=0.48\textwidth]{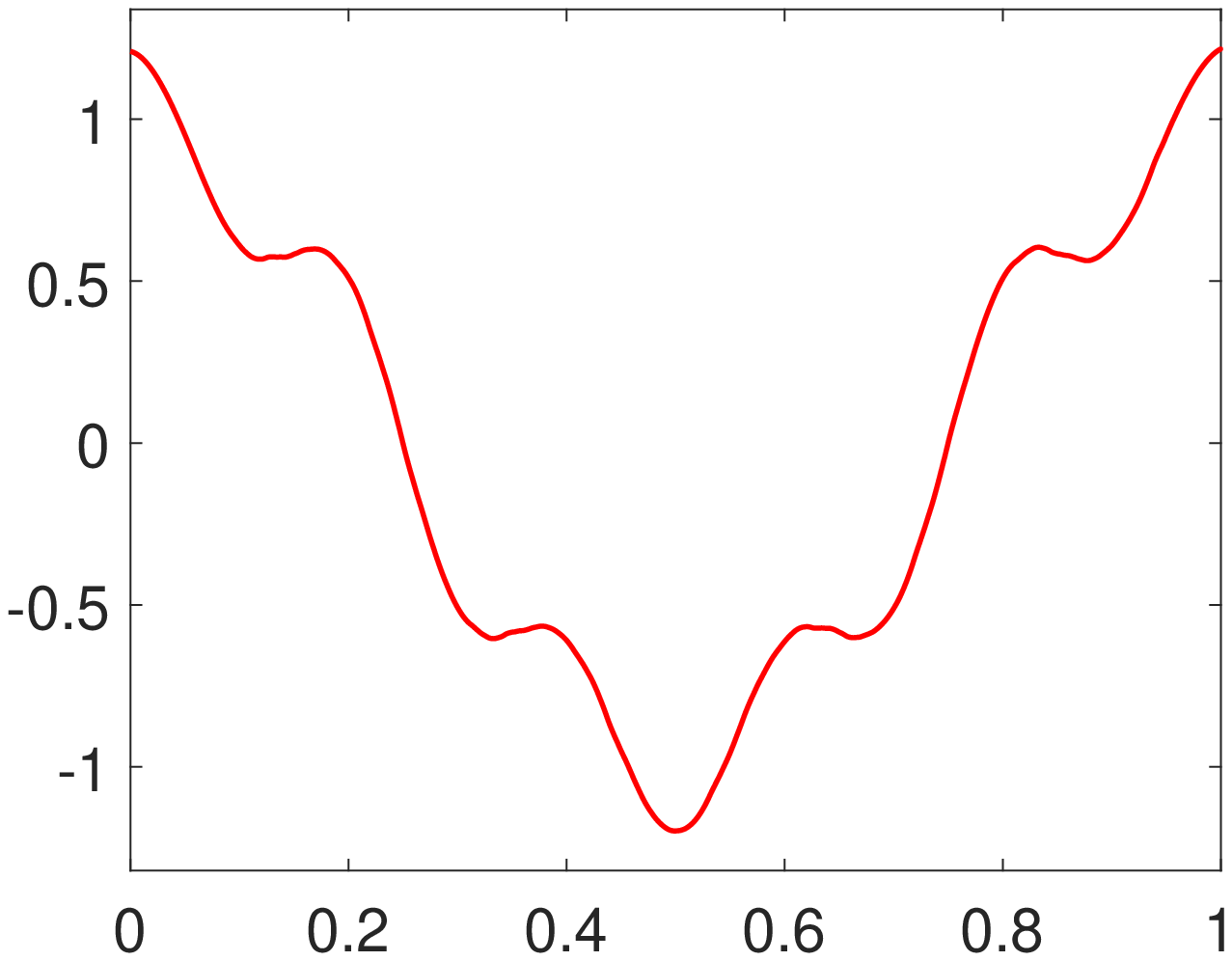}
		\label{fig:CSCont8}}
	\subfloat[Reconstructed wavelet coefficients of $f$ for $N=2^{6}$ and $|m|=64$]{
		\includegraphics[width=0.48\textwidth]{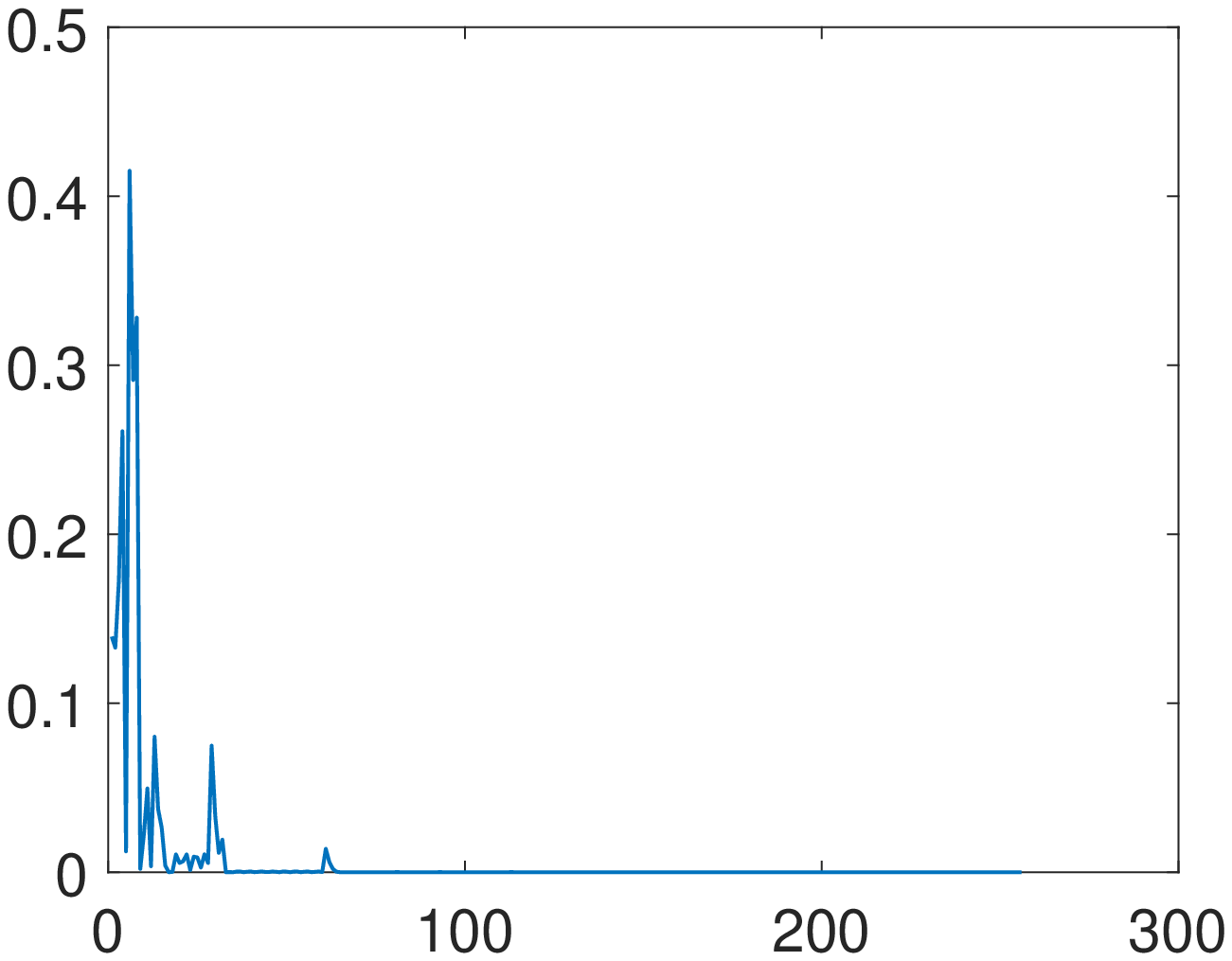}
		\label{fig:RecWaveCoeffCont}}
	\quad
	\subfloat[TW reconstruction of $f$ from $|m|=64$ samples]{
		\includegraphics[width=0.48\textwidth]{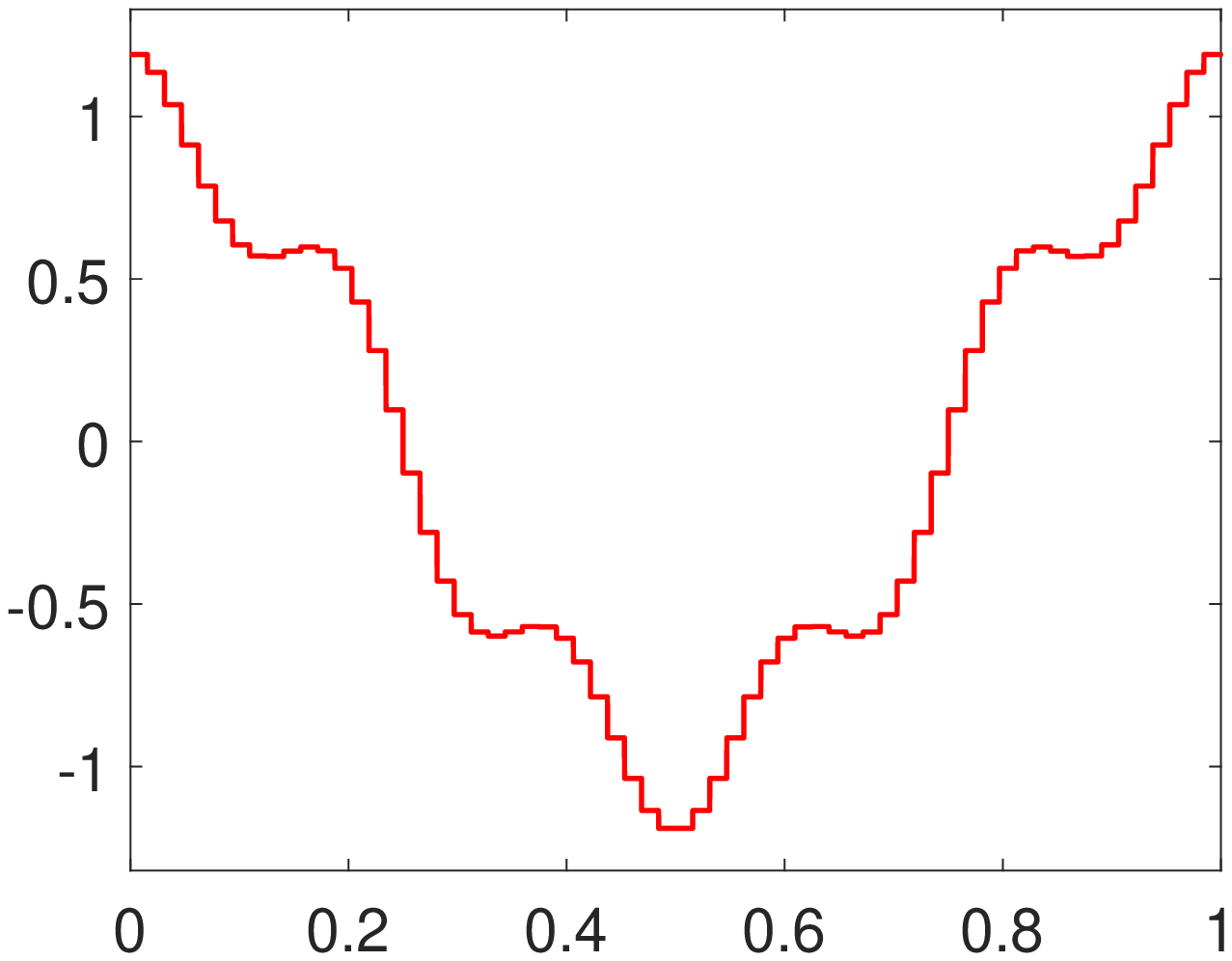}\label{fig:TWCont}}
	\subfloat[TW reconstruction of $g$ from $|m| = 256$ samples]{
		\includegraphics[width=0.48\textwidth]{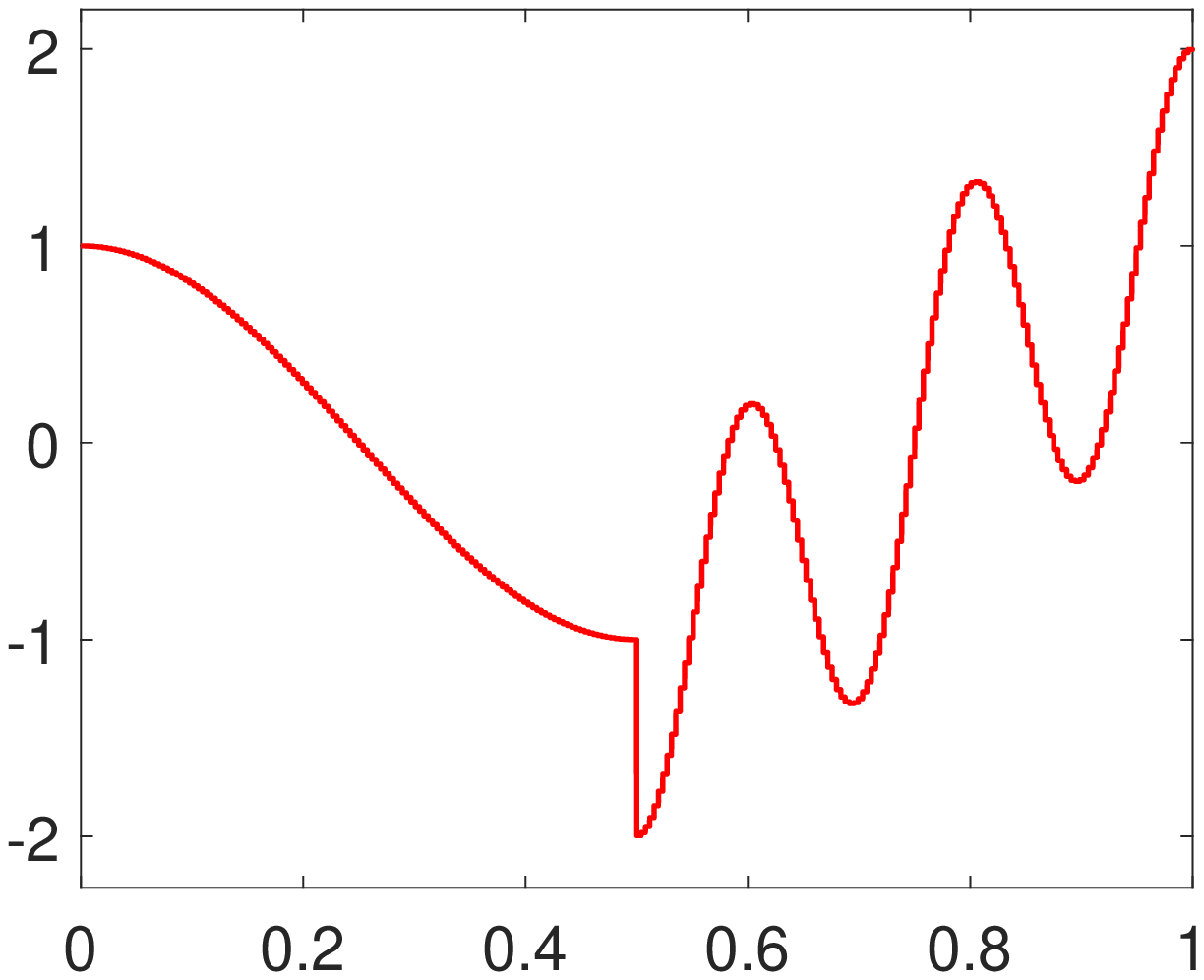}\label{fig:TWDisc}}
	\quad
	\subfloat[Error plot for reconstruction of $f$]{
		\includegraphics[width=0.48\textwidth]{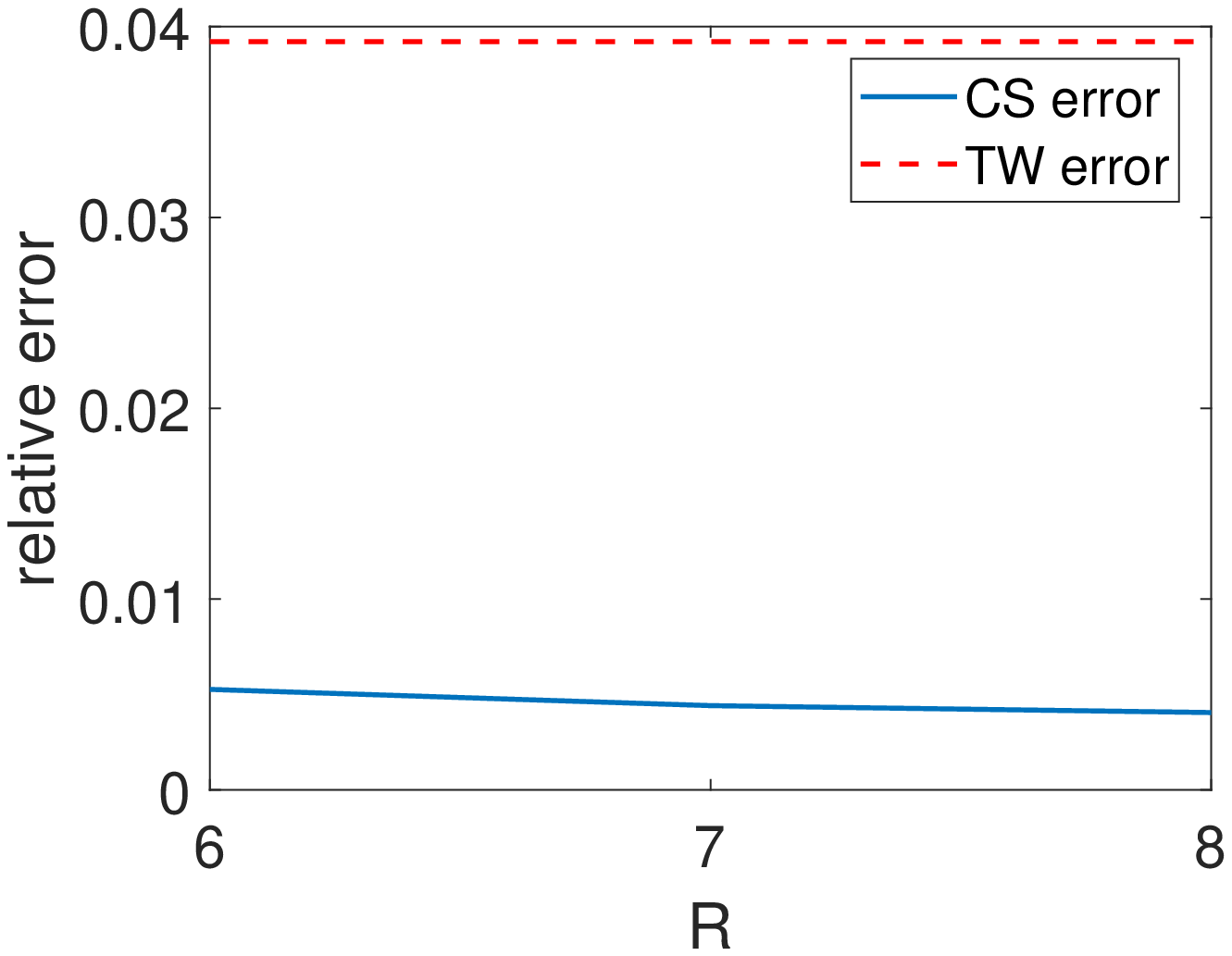}
		\label{fig:NMrelationCont}}
	\subfloat[Error plot for reconstruction of $g$]{
		\includegraphics[width=0.48\textwidth]{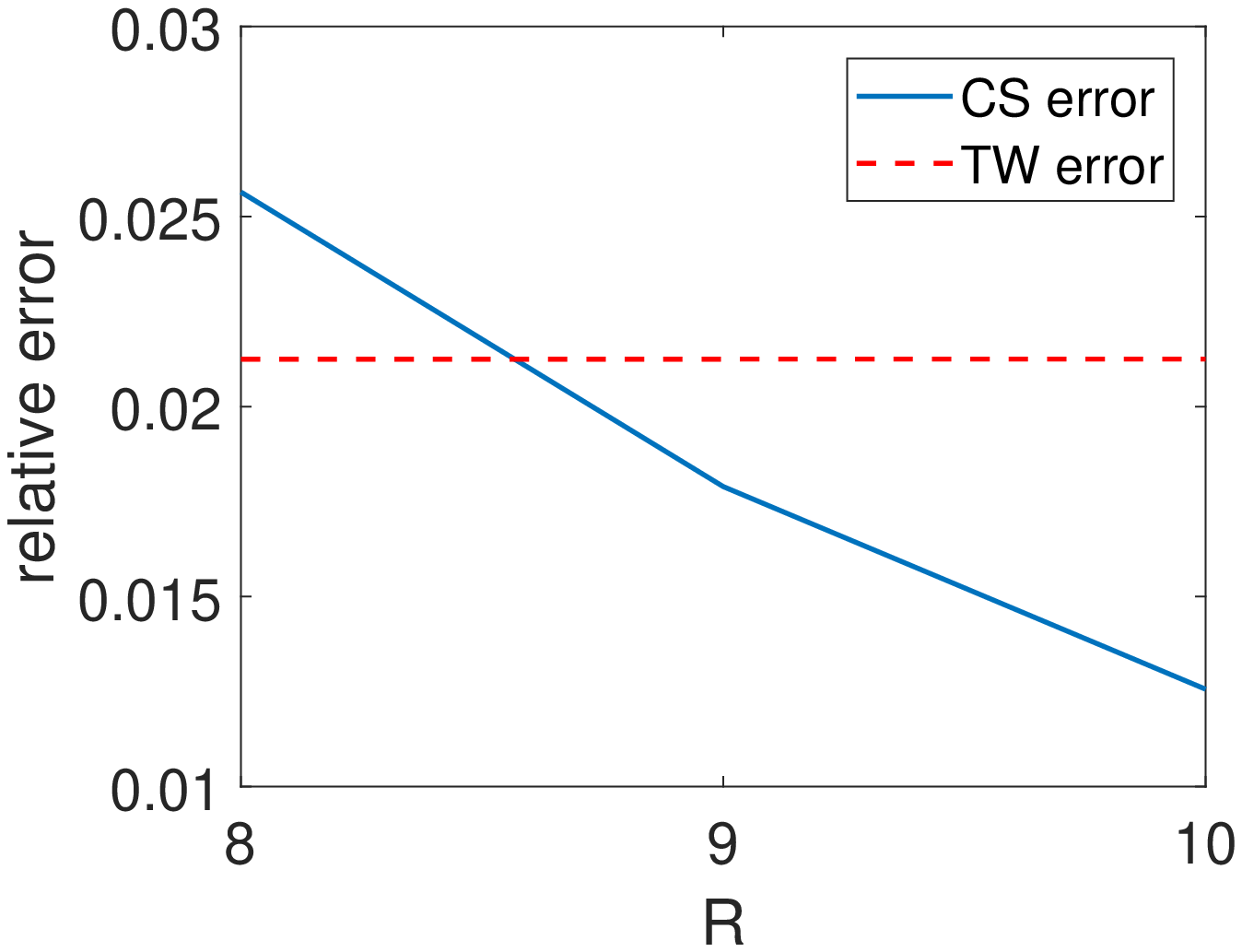}
		\label{fig:NMrelationDisc}}
	\caption{CS reconstruction of $f$ with its wavelet coefficients. TW reconstruction for $f$ and $g$ from the same number of samples as used in CS. CS and truncated Walsh series error values with $|m| = 64$ for $f$ and $|m|=256$ for $g$. The $x$-axis represents the sampling bandwidth $N=2^R$ and the $y$-axis the relative error term in the $\ell_2$ norm.}
	\label{fig:NMTWCScont}
\end{figure}

In this section we demonstrate how the theoretical results can be used in practice. We investigate the reconstruction of two signals with different smoothness properties. We start the analysis with the impact of the sampling bandwidth $N$ if we keep the number of samples constant. Further we compare the reconstruction with CS to the application of the inverse Walsh transform to the subsampled data. This highlights the subsampling possibilities of the method. Related to the discussion in Remark \ref{Rem:RelationFourier} we investigate the experimental differences between Fourier and Walsh measurements. Finally, we illustrate the importance of the structure of the sampling pattern with the \emph{flip test} as introduced in \cite{breaking}.

To see the impact of the sparsity of the signal we consider two functions. We have
\begin{equation}\label{Eq:OrigCont}
	f(x) = \cos(2\pi x)  + 0.2 \cos(2\pi 5 x)
\end{equation}
and
\begin{equation}\label{Eq:OrigDisc}
	g(x) = \cos(2 \pi x)  + \cos(2 \pi 5 x) \mathcal X_{x\geq 0.5}.
\end{equation}
The first function is very smooth and hence obeys only very few non-zero coefficients in the wavelet domain which all are located in the first levels. In contrast the function $g$ has a discontinuity at $0.5$. This results to non-zero elements in the wavelet domain also in higher levels and hence a larger $M$ for perfect or near perfect reconstruction. The sparsity structure of natural images behaves equally. The more discontinuities a signal has the more higher level coefficients need to be reconstructed. Therefore, a prior analysis on the signals that will be reconstructed is important, even though all natural signals have the same coefficient structure with a few very large ones in the beginning and fewer in the high levels. It only differs in the number of higher level coefficients. This can be seen for the two examples. For both functions we have that the first level is not sparse, and we therefore need to sample $m_1$ fully. However, for $m_k$ with $k \geq 2$ we can use the main theorem and reduce the number of samples $|m| = \sum_{k=1}^r m_k$. The sampling pattern is chosen such that the first level is sampled fully and that the rest of the possible samples is divided equally on all $m_k$ and hence independent on the size of the level.

For the implementation of the optimization problem \eqref{Eq:ReconstructionProblem} we have chosen $L = 2^{12}$ it can be observed that the reconstruction does not change from $L=2^{10}$ onwards. Therefore, it can be assumed that the algorithm has converged at this point and that this finite setting resembles the infinite setting. We can see that the location of the highest level coefficients only depends on the choice of $N$. The code for the numerical experiments relies on the \emph{cww} Matlab package for the reconstruction of wavelet coefficients from Walsh samples by Antun in \cite{GSCS}. The optimization problem is solved with SPGL-1 from Antun available at \cite{SPGL1}.  
We adapted the code to include the infinite setting. It can be found in \cite{CodeExperiments}. The wavelet in use is the Daubechies wavelet with $p=4$ vanishing moments.

In figure \ref{fig:RecDiscont} the reconstruction of $g$ from $|m|=256$ is presented for different choices of $N$. The related sampling pattern are shown in figure \ref{fig:samplingPattern}. The aim of this experiment is to show the impact of the balancing property. Hence, how the choice of $N$ effects $M$ the coefficient bandwidth. To make this more visible we have plotted the wavelet coefficients of the reconstruction next to the reconstructed signals. It can be seen that the coefficients with numbering larger than $N$ cannot be reconstructed with CS. However, it can also be seen that the square relation between $N$ and $M$ in the balancing property in \eqref{Eq:MainTheoAssumptionNM} is not sharp as we are able to reconstruct coefficients larger than $\sqrt{N}$. The decay of the error terms with increasing $N$ can be seen in figure \ref{fig:NMrelationDisc}. Without increasing the number of samples we can decrease the error term in the reconstruction for CS and improve over the directed inversion of the samples.

In figure \ref{fig:CSCont8}, \ref{fig:RecWaveCoeffCont} and \ref{fig:NMrelationCont} we see the related experiments for the continuous signal $f$. The signal only obtains a few non-zero coefficients in the first levels. Therefore, the increased number of $N$ does not change the error or reconstruction which is already for the smallest choice $N=2^6$ nearly invisible. That is the reason why we did not show the reconstruction for $N=2^7, 2^8$ as there is no visible difference. But also in this case it can be seen that we only reconstruct coefficients indexed smaller than $N$ but still larger than $\sqrt{N}$ which implies that the bound for the balancing property is also not sharp in the continuous and very sparse setting. 

After the analysis of the impact of the balancing property on the reconstruction, we show how the number of samples influences the reconstruction. As it has been shown in the main theorem we have to control two parts the balancing property which mainly aims at the bandwidth of the samples $N$ and the related bandwidth of the reconstructed coefficients $M$. Here, the number of coefficients in each level is not considered. This is part of the second requirement on the number of samples in each level $m_k$. It has been seen that beside other logarithmic factors this mainly depends on the sparsity $s_k$ in that level and with exponentially decaying impact also on the sparsity in the neighbouring levels. It can be seen that the discontinuous function $g$ has wavelet coefficients in higher levels. However, in every level there are only one or two and hence the signal is very sparse in those higher levels. This can be exploited to reduce the number of samples drastically as in figure \ref{fig:CSDiscont64}. Even though the error increases a small amount we were able to reduce the number of samples by to $64$ and this is only $25\%$ of the previous number of samples. In this case the contrast to the truncated Walsh transform gets even more obvious as in figure \ref{fig:TWDiscont64}. The continuous function in contrast has only non-zero coefficients until $M=64$ however it is not very sparse up to this number. Therefore, the reconstruction from even less samples $|m|=32$ obeys more artefacts than before, see figure \ref{fig:CSCont32}. Nevertheless, the general structure is still more visible than with the direct Walsh inverse transform in figure \ref{fig:TWCont32}.

\begin{figure}
	\subfloat[CS reconstruction of $f$ with $\epsilon = 0.048$]{
		\includegraphics[width=0.48\textwidth]{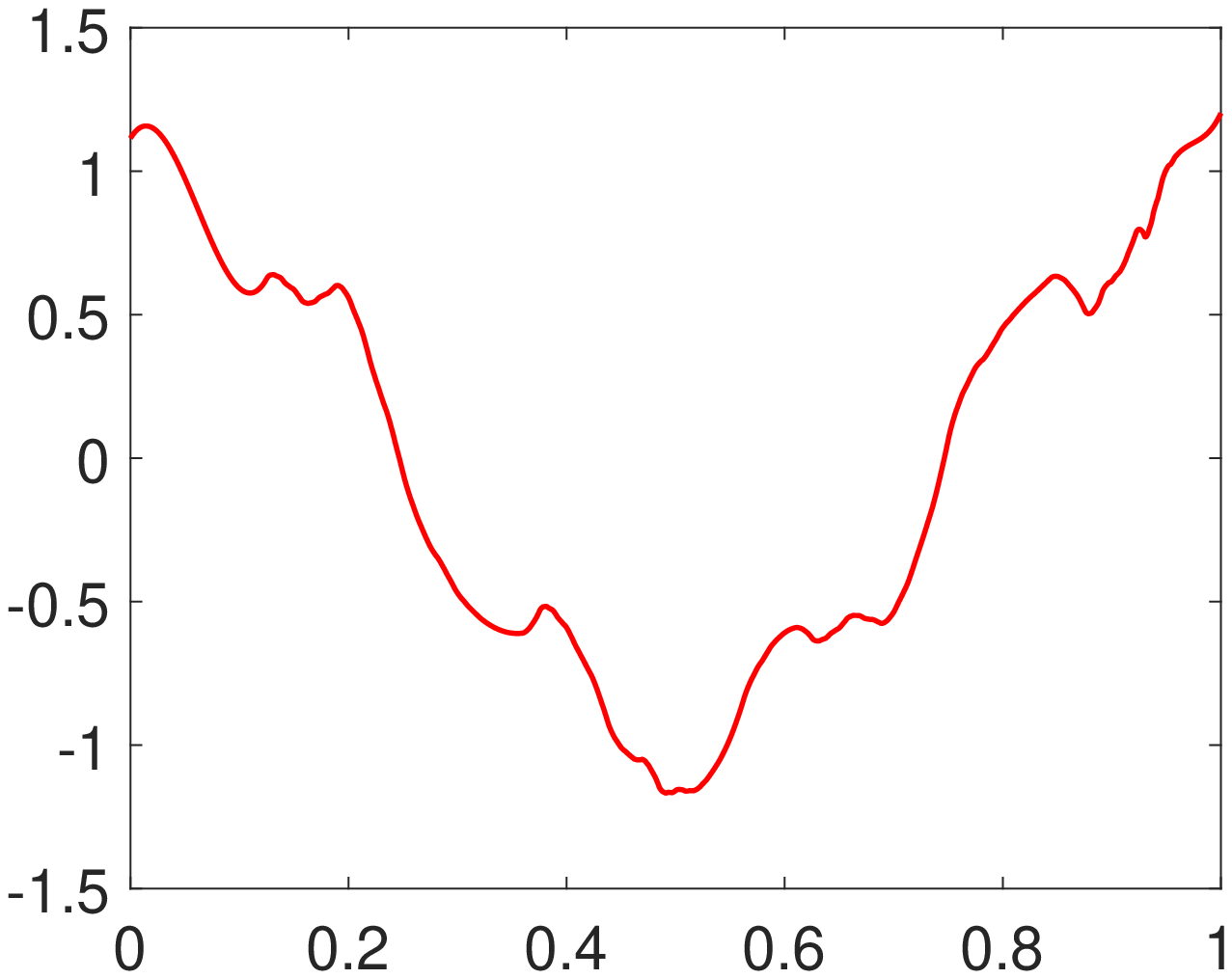}
		\label{fig:CSCont32}}
	\subfloat[CS reconstruction of $g$ with $\epsilon = 0.0298$]{
		\includegraphics[width=0.48\textwidth]{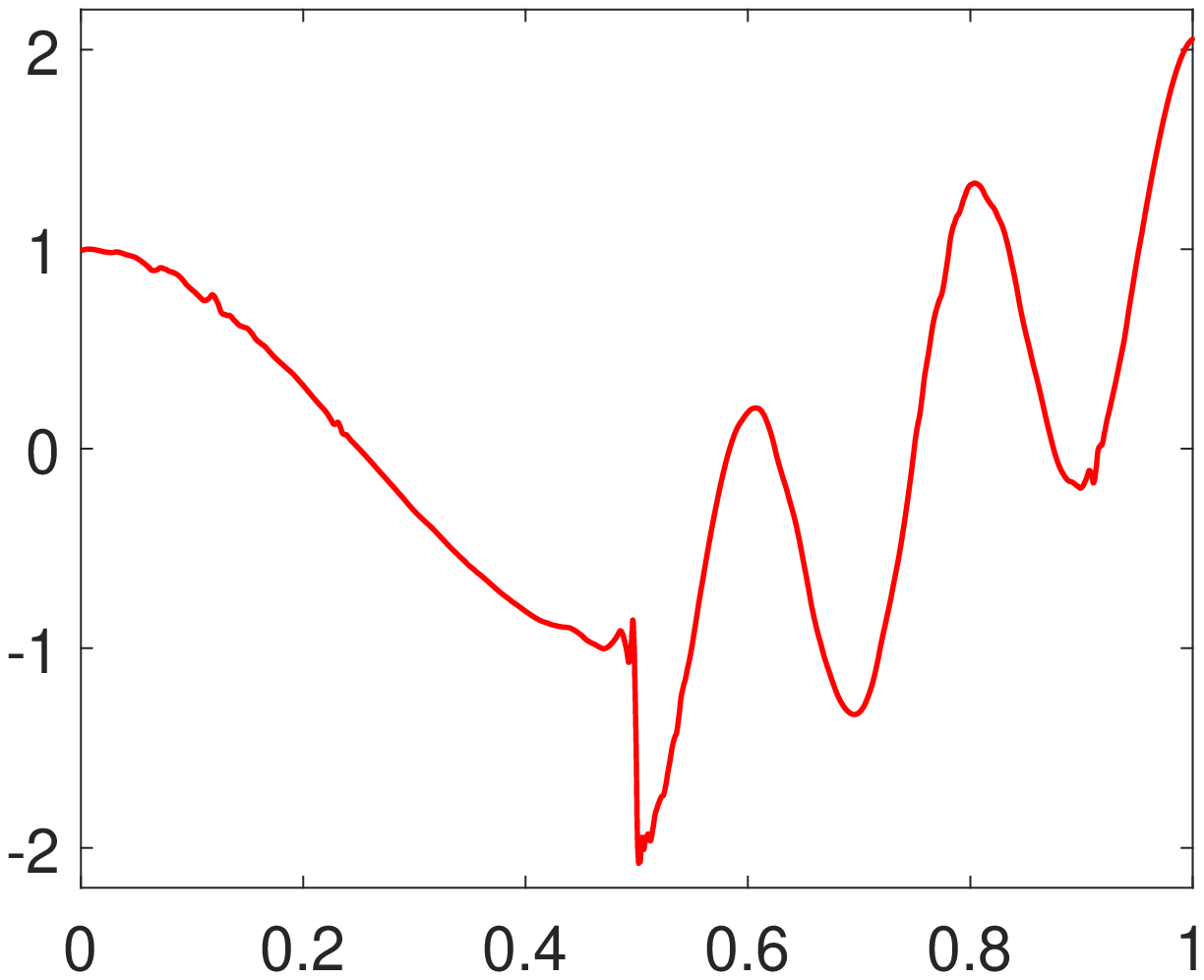}
		\label{fig:CSDiscont64}}
	\quad
	\subfloat[TW reconstruction of $f$ with $\epsilon = 0.078$]{
		\includegraphics[width=0.48\textwidth]{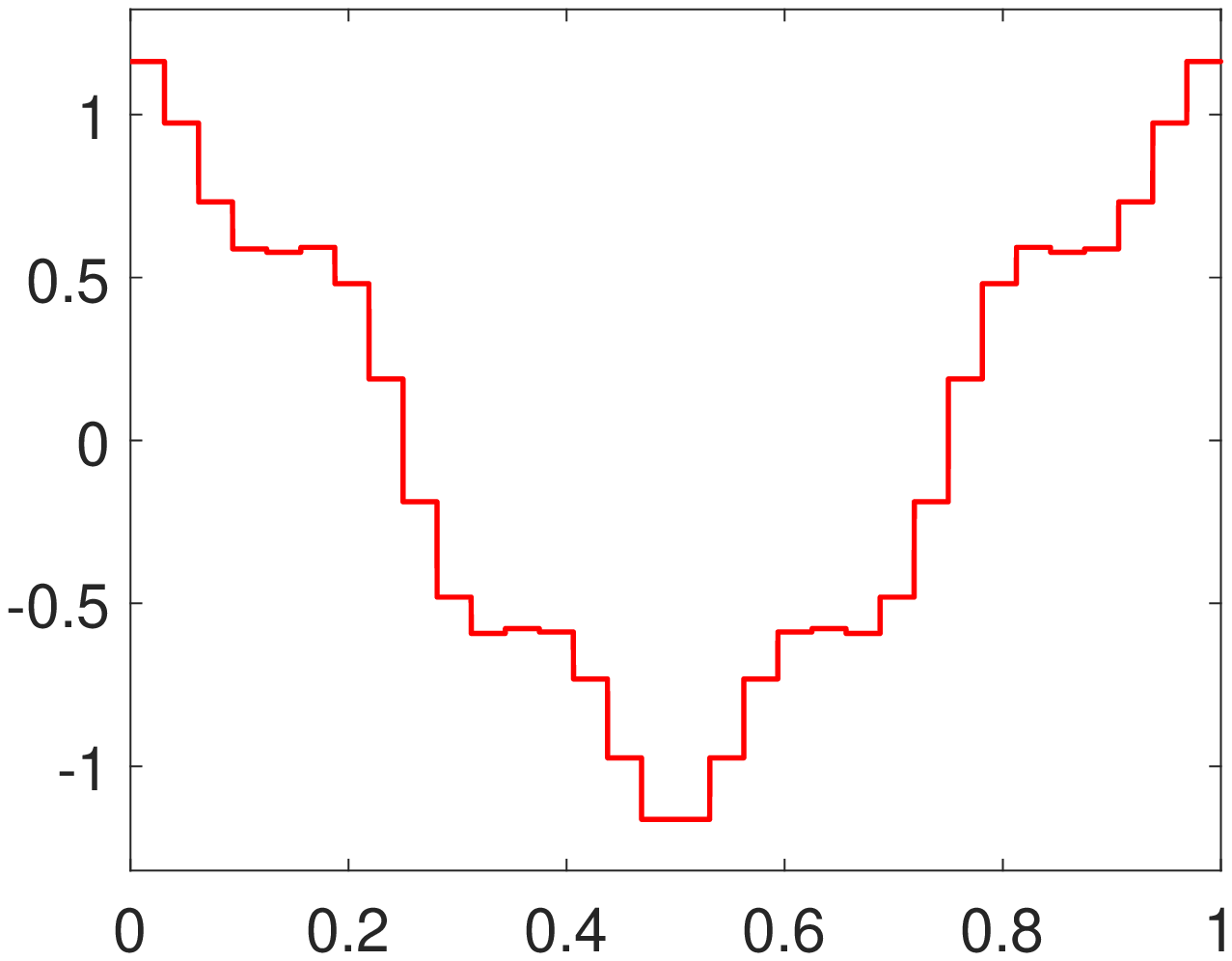}
		\label{fig:TWCont32}}
	\subfloat[TW reconstruction of $g$ with $\epsilon = 0.085$]{
		\includegraphics[width=0.48\textwidth]{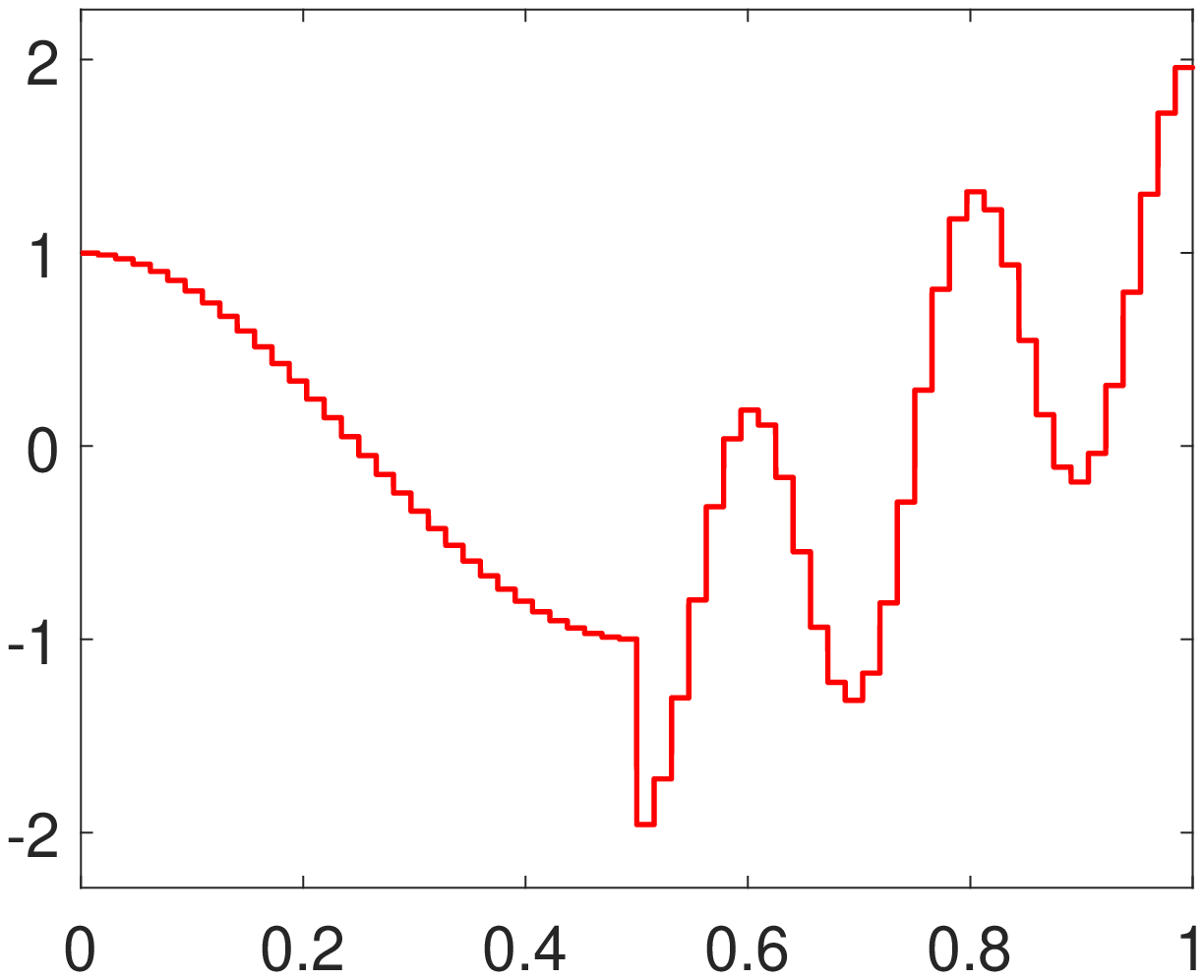}
		\label{fig:TWDiscont64}}
	\quad
	\subfloat[Sampling pattern for $f$ with $N=2^6$]{
		\includegraphics[width=0.48\textwidth]{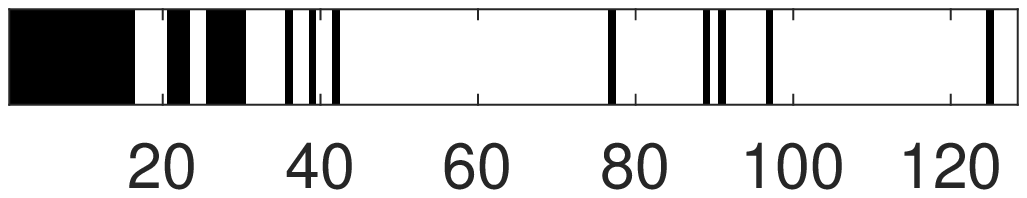}
		\label{fig:SampCont32}}
	\subfloat[Sampling pattern for $g$ with $N=2^8$]{
		\includegraphics[width=0.48\textwidth]{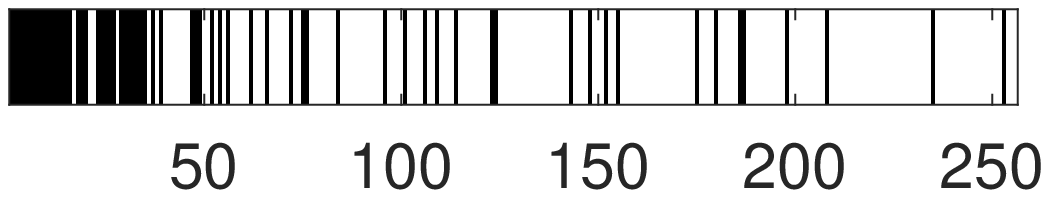}
		\label{fig:SampDiscont64}}
	\caption{CS and truncated Walsh series error values with $|m| = 32$ for $f$ and $|m| = 64$ for $g$} 
\end{figure}

\begin{figure}
	\subfloat[Fourier sampling with $|m|=128$ and $p=2$ with $\epsilon = 0.019$]{
		\includegraphics[width=0.42\textwidth]{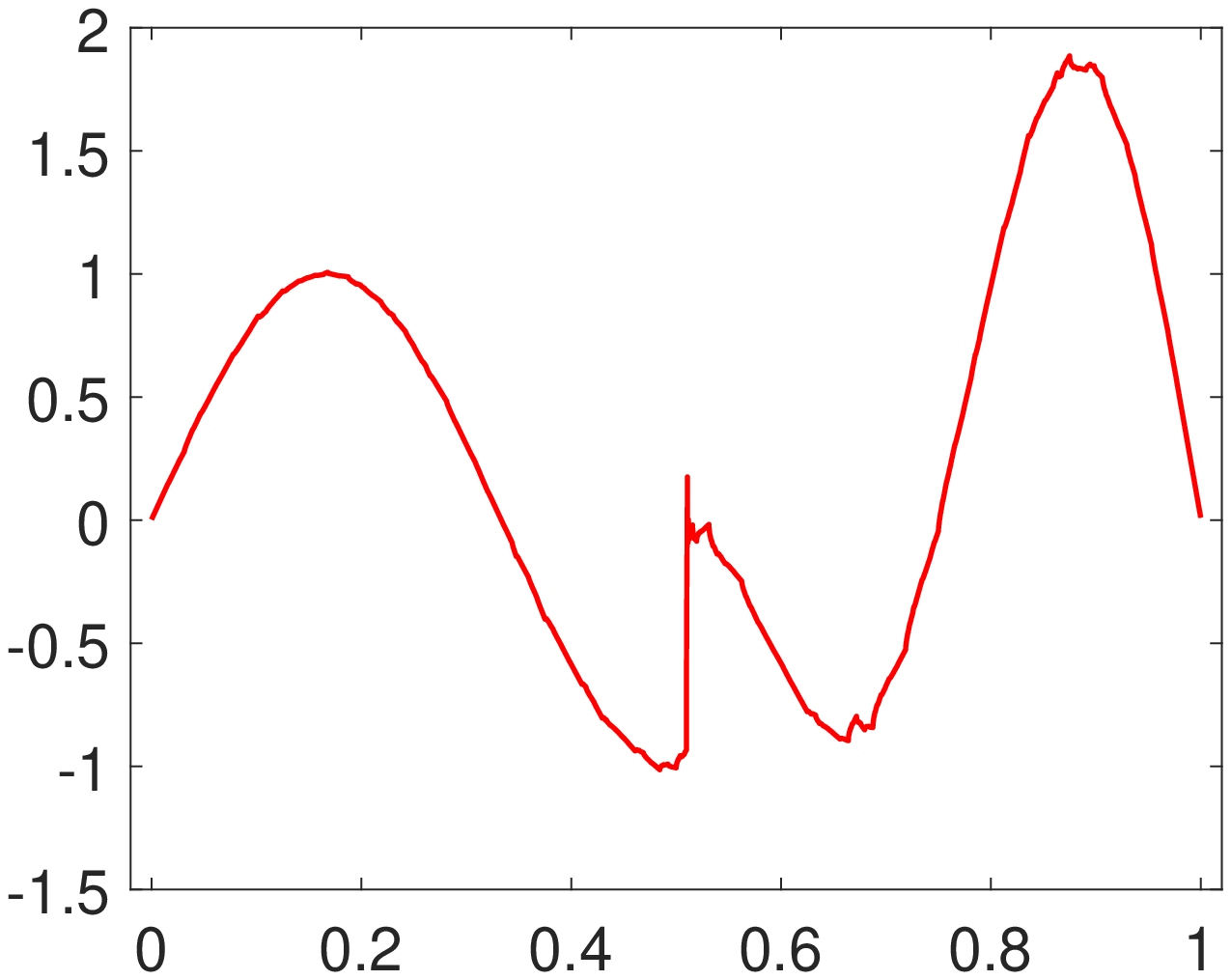}
		\label{fig:Four_vm2_128}}
	\subfloat[Fourier sampling with $|m|=128$ and $p=6$ with $\epsilon = 0.015$]{
		\includegraphics[width=0.41\textwidth]{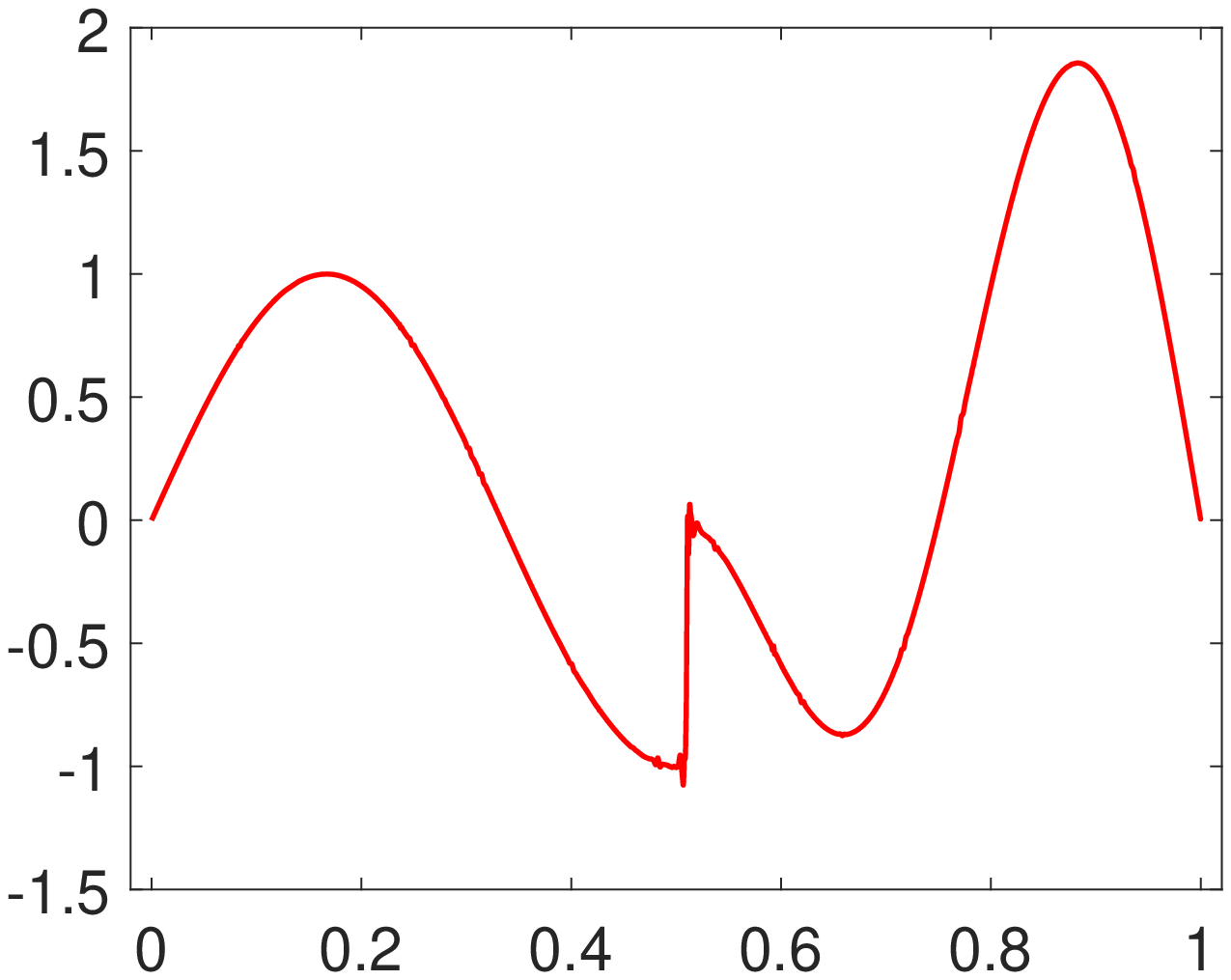}
		\label{fig:Four_vm6_128}}
	\quad
	\subfloat[Fourier sampling with $|m|=256$ and $p=2$ with $\epsilon = 0.009$]{
		\includegraphics[width=0.41\textwidth]{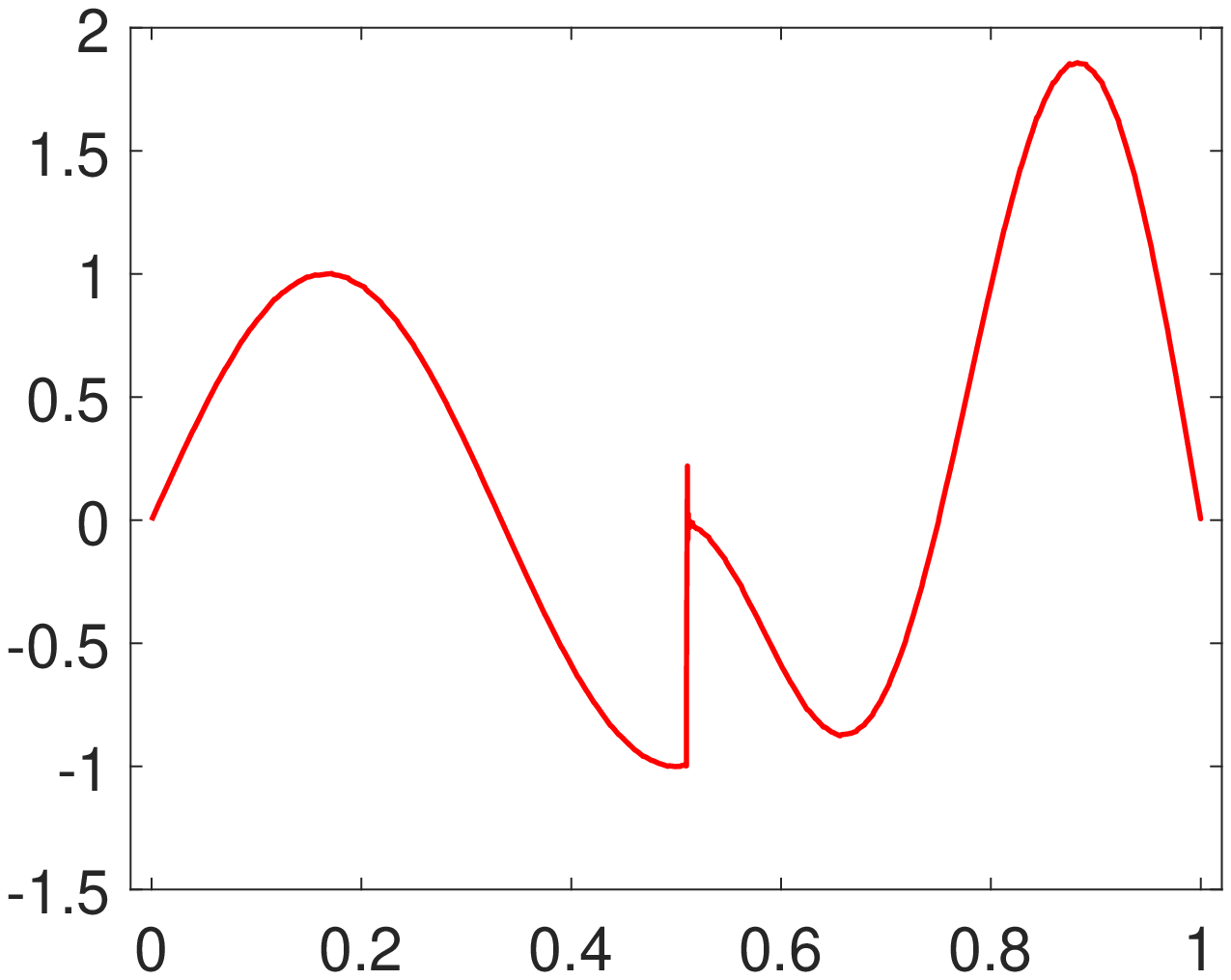}
		\label{fig:Four_vm2_256}}
	\subfloat[Fourier sampling with $|m|=256$ and $p=6$ with $\epsilon = 0.013$]{
		\includegraphics[width=0.41\textwidth]{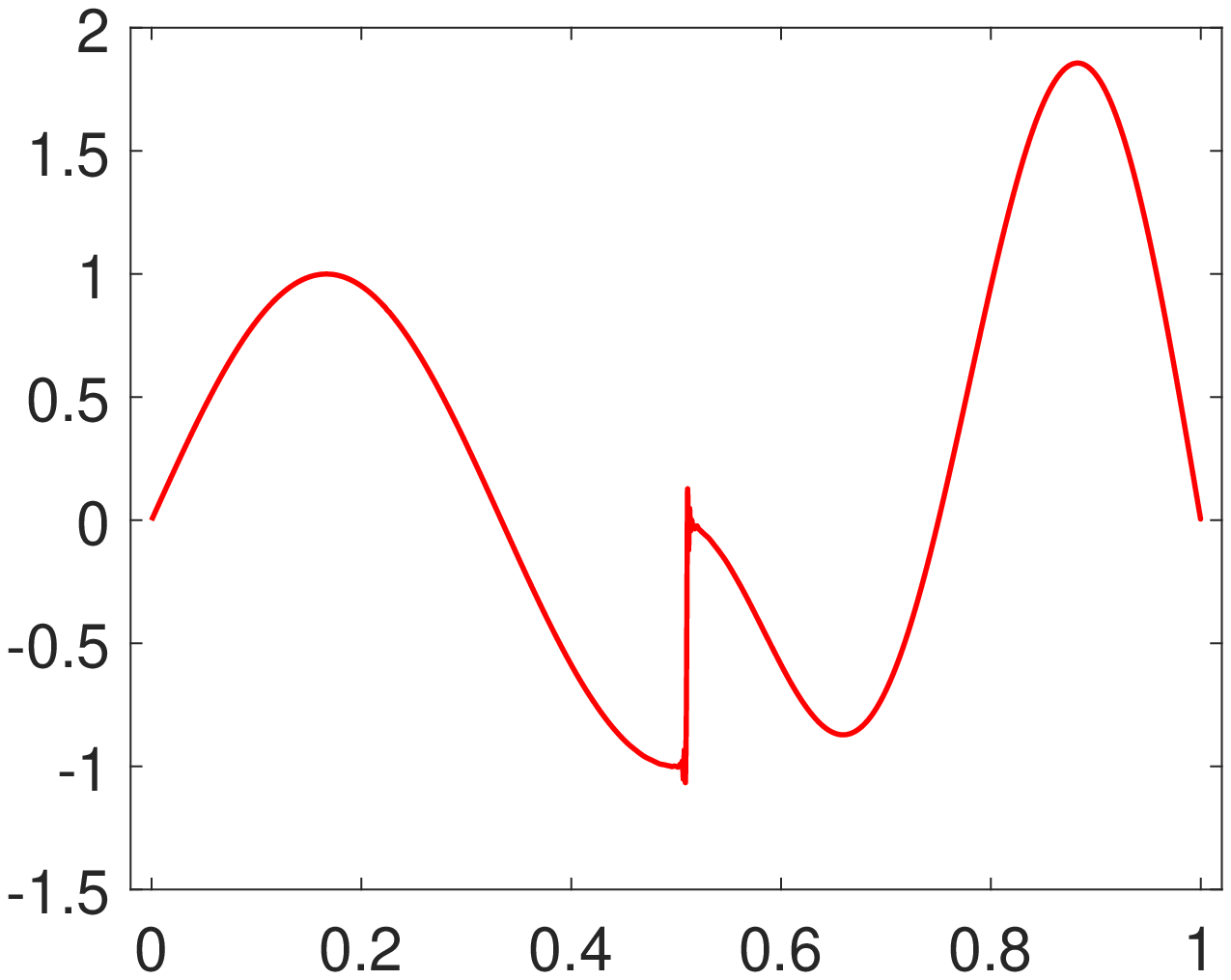}
		\label{fig:Four_vm6_256}}
	\quad
	\subfloat[Walsh sampling with $|m|=128$ and $p=2$ with $\epsilon = 0.023$]{
		\includegraphics[width=0.41\textwidth]{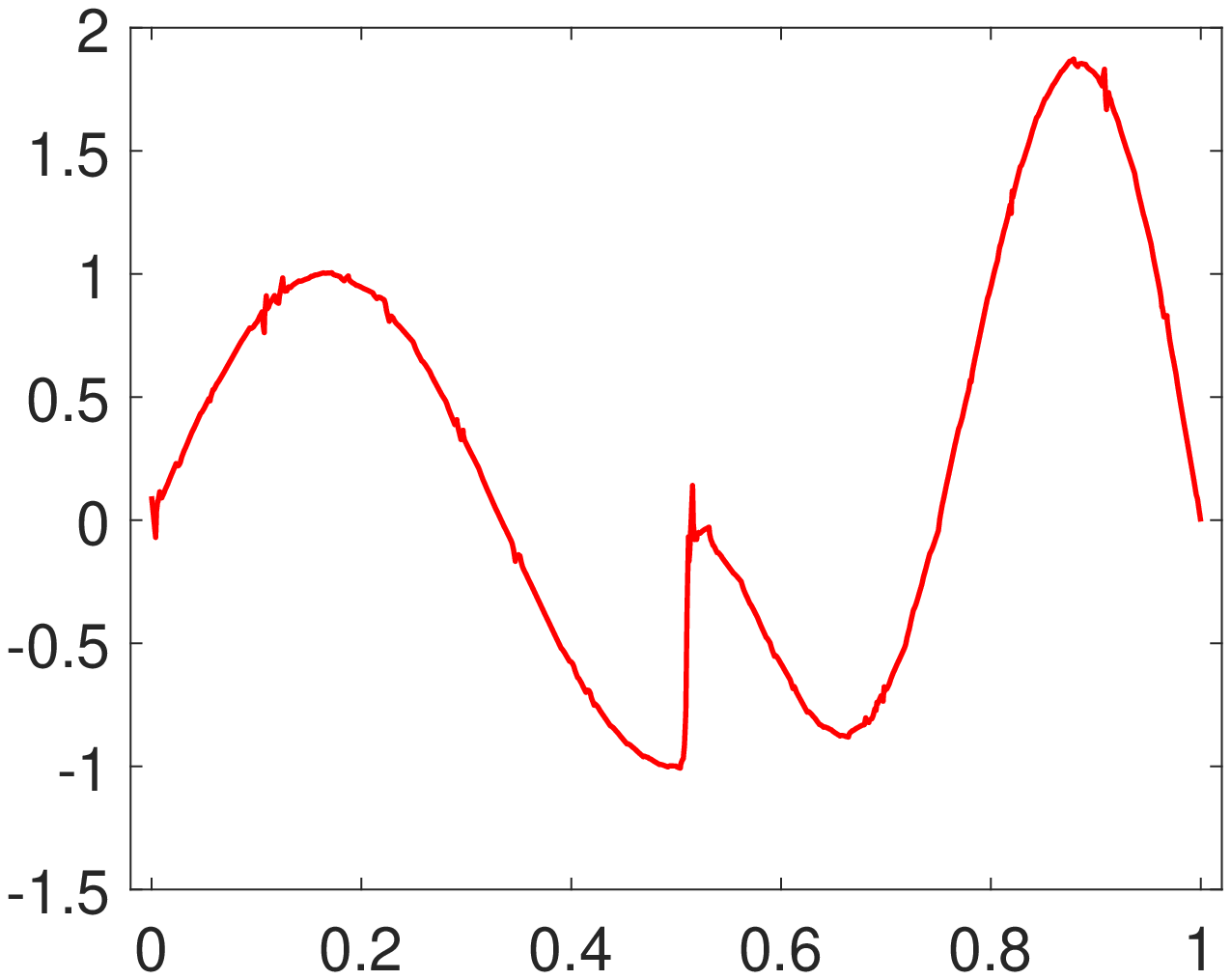}
		\label{fig:Walsh_vm2_128}}
	\subfloat[Walsh sampling with $|m|=128$ and $p=6$ with $\epsilon = 0.019$]{
		\includegraphics[width=0.41\textwidth]{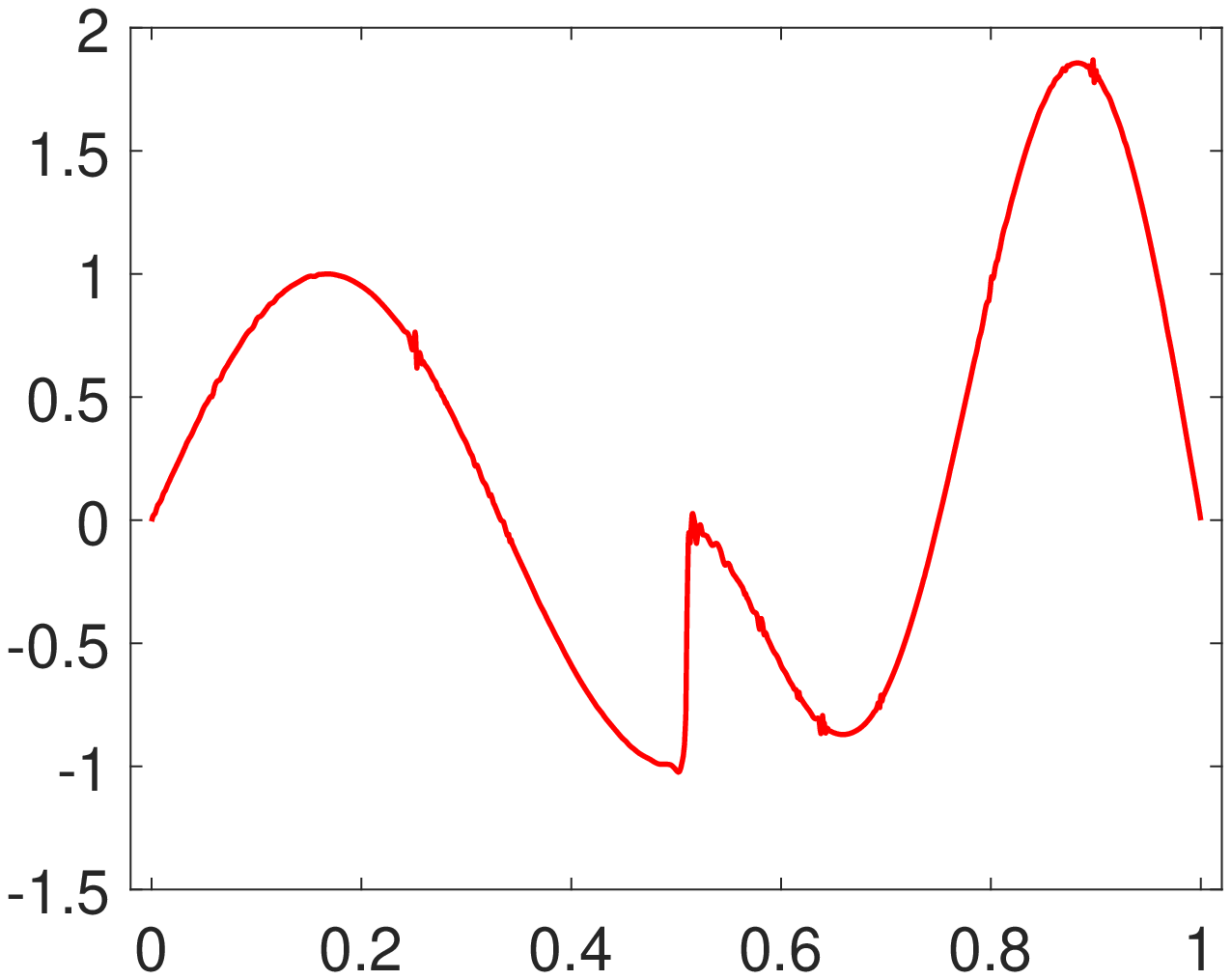}
		\label{fig:Walsh_vm6_128}}
	\quad
	\subfloat[Walsh sampling with $|m|=256$ and $p=2$ with $\epsilon = 0.009$]{
		\includegraphics[width=0.41\textwidth]{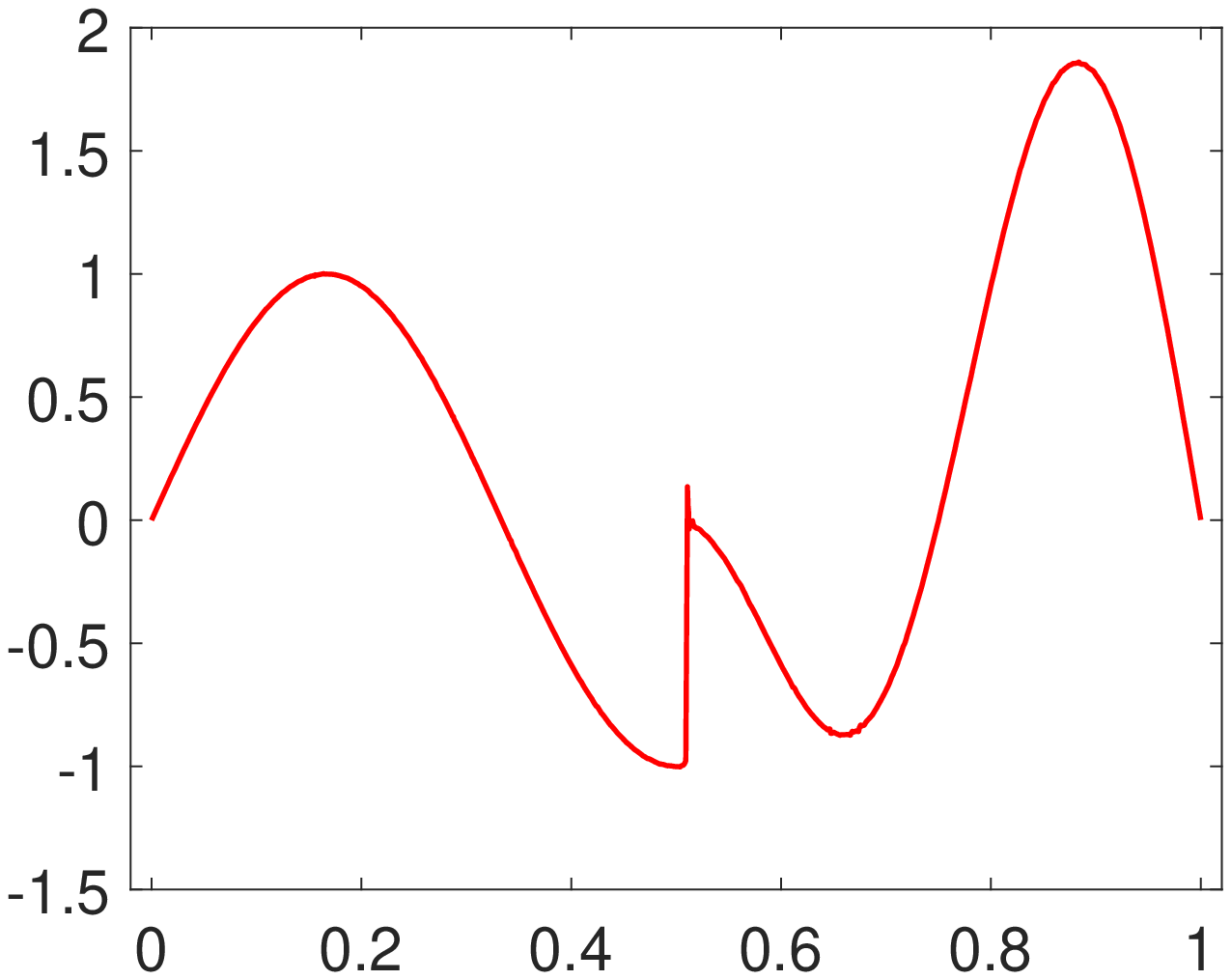}
		\label{fig:Walsh_vm2_256}}
	\subfloat[Walsh sampling with $|m|=256$ and $p=6$ with $\epsilon = 0.011$]{
		\includegraphics[width=0.41\textwidth]{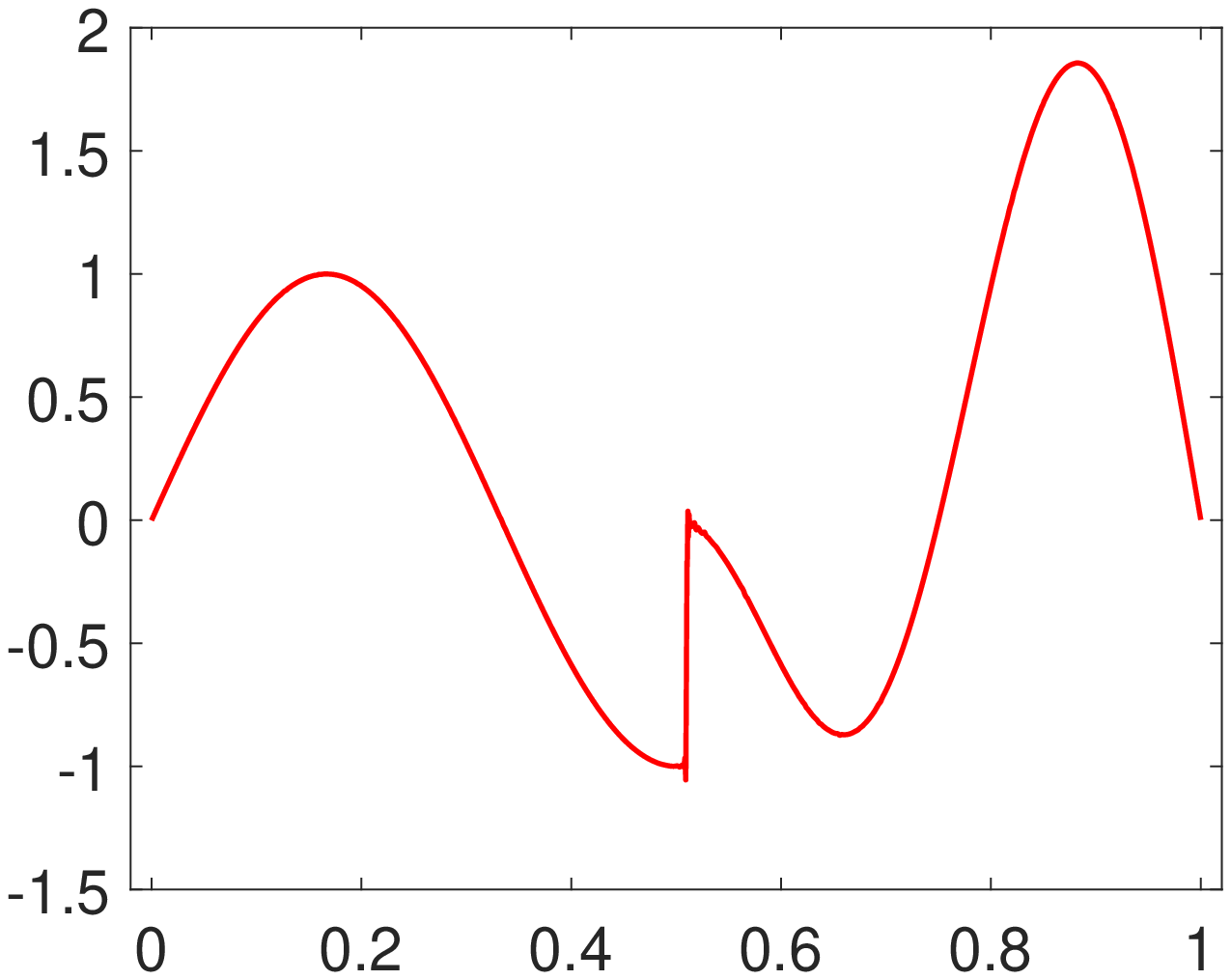}
		\label{fig:Walsh_vm6_256}}
	\caption{CS reconstruction from Fourier samples and Walsh samples for Daubechies wavelets with $2$ and $6$ vanishing moments. The experiment is set up with $N=2^{10}$, $L=2^{13}$ and $|m| = 128$ and $|m|=256$. The relative $\ell_2$-error is denoted by $\epsilon$.} 
	\label{fig:FourComparison}
\end{figure}

The conducted experiments can be compared to the Fourier setting. In Remark \ref{Rem:RelationFourier} we have discussed the relationship between the theoretical results for both sampling modalities. It is important to recall that the theorems offer sufficient conditions for the reconstruction. Therefore, the theoretical differences do not have to become visible in the experiments. Nevertheless, we want to discuss the different aspects of the theory and what we can observe in numerical experiments. The code to produce the Fourier experiments is an adaptation of the one discussed in \cite{MilanaClarice} and available at \cite{CodeExperiments}. The most obvious difference is the squared versus linear relationship in the estimate of the balancing property. It has been shown in the previous examples that these bounds are not sharp for the Walsh setting. Hence, the difference in the balancing property cannot be observed in numerical experiments. Next, we know from the theory that the smoothness of the wavelet impacts the decay rate under the Fourier transform but not for the Walsh transform. This difference is clearly visible in the reconstruction matrix, see Figure \ref{fig:RecMatrix}. Finally, there is also a difference in the number of necessary samples per level because of the number of vanishing moments. With more vanishing moments the impact of the other levels decreases exponentially for the Fourier case. Additionally, the location of the non-zero elements in the coefficient vector changes with the choice of the wavelet. In Figure \ref{fig:FourComparison} we have illustrated the results of the reconstruction of the same function for Fourier measurements and Walsh measurements with wavelets of two different vanishing moments and also changing number of samples. The illustration shows us that for a small number of samples the reconstruction with Fourier gives better results and gives smaller relative error as well as less visible artefacts. However, for a large enough sampling size the reconstructions are in both cases very close to the original function. Moreover, it is possible to see that the different wavelets tend to different artefact types. It can be observed that the wavelets with $2$ vanishing moments are able to recover the discontinuity very well and have less ringing artefacts around it in contrast to the case of wavelets with $6$ vanishing moments. However, in the smoother areas the reconstruction with fewer vanishing moments leads to more artefacts as can be seen in Figure \ref{fig:Four_vm2_128} and \ref{fig:Four_vm6_128}. 

Finally, we demonstrate that the structure of the signal coefficients and the change of basis matrix are very important. For this sake we conducted the same experiment for the continuous function $f$ with a flipped sampling pattern, see figure \ref{fig:samplingFlip}. The reconstruction is nowhere close to perfect and the original signal is not even identifiable. Hence, it is important to take the structure into account.

The experiments show that both parts the number of samples per level and the balancing property are important for the success of the method. For the practical application it is very helpful to be able to estimate beforehand the size of $M$ as well as the sparsity per level.

\begin{figure}
	\subfloat[CS reconstruction]{
		\includegraphics[width=0.48\textwidth]{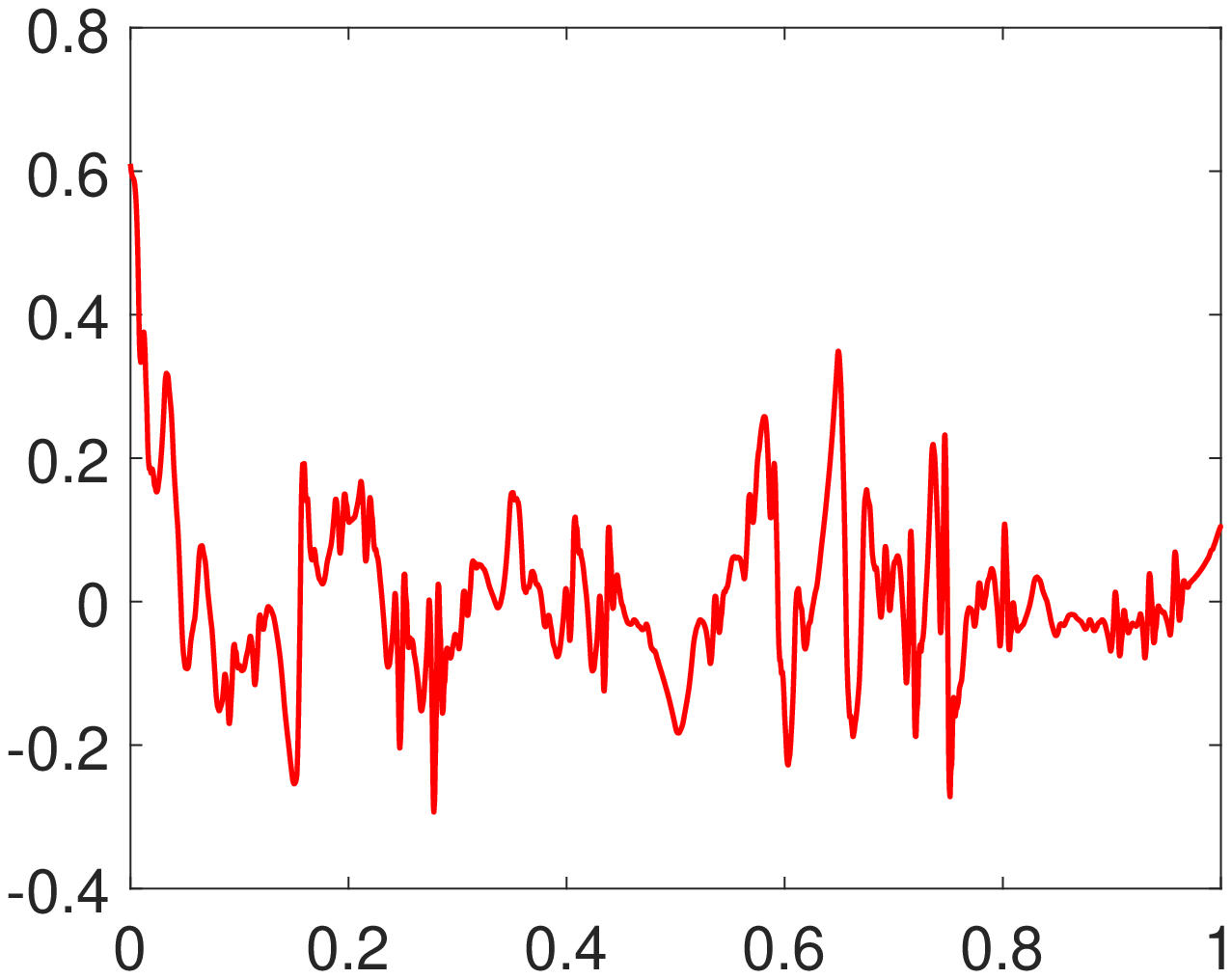}
		\label{fig:RecFlip}}
	
	\subfloat[Flipped sampling pattern]{
		\includegraphics[width=0.48\textwidth]{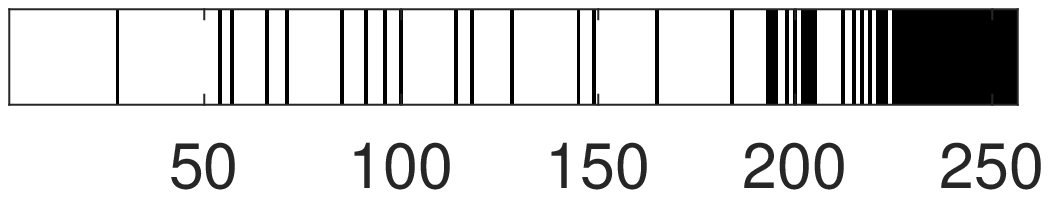}
		\label{fig:samplingFlip}}
	\caption{Reconstruction of $f$ with flipped samples with $N=2^8$ and $|m|=64$}
	\label{fig:FlipTest}
\end{figure}

\section{Conclusion}\label{Ch:Conclusion}

In this paper we have analysed the reconstruction of a one-dimensional signal from binary measurements with structured CS. We gave non-uniform recovery guarantees for the reconstruction with boundary corrected Daubechies wavelets. 
Additionally, we showed the numerical gains and the problems that arise when the theory is not taken into account.

These result fit nicely in the theory of non-linear reconstruction. We now have knowledge about uniform and non-uniform recovery guarantees for two of the major measurement types: Fourier and binary with Daubechies wavelets. For the future it would be of interest to investigate the theory for Radon measurements. Additionally, the bounds for the recovery guarantees are not tight and hence it might be possible to improve them. This relates closely to the question which wavelet type is most suitable for the reconstruction from binary measurements. This, however, is so far not known. Therefore, a further investigation of the relationship between the different wavelet types and Walsh functions would be of interest as numerical experiments suggest that there exists a difference.

\section*{Acknowledgement}

LT  acknowledges  support  by  the  UK  Engineering  and  Physical  Sciences  Research  Council  (EPSRC) grant EP/L016516/1 for the University of Cambridge Centre for Doctoral Training, the Cambridge Centre for Analysis. AH acknowledges support from a Royal Society University Research Fellowship. This is a post-peer-review, pre-copyedit version of an article published in Journal of Fourier Analysis and Applications. The final authenticated version is available online at: http://dx.doi.org/10.1007/s00041-021-09813-6.


\begin{thebibliography}{10}
	
	\bibitem{VegardUniform}
	B.~Adcock, V.~Antun, and A.~C. Hansen.
	\newblock {U}niform {R}ecovery in {I}nfinite-{D}imensional {C}ompressed {S}ensing and
	{A}pplications to {S}tructured {B}inary {S}ampling.
	\newblock {\em arXiv preprint arXiv:1905.00126}, 2019.
	
	\bibitem{GS2}
	B.~Adcock and A.~Hansen.
	\newblock {S}table {R}econstructions in {H}ilbert {S}paces and the {R}esolution of the {G}ibbs {P}henomenon.
	\newblock {\em Appl. Comput. Harmon. Anal.}, 32(3):357--388, 2012.
	
	\bibitem{2DCase}
	B.~Adcock, A.~Hansen, G.~Kutyniok, and J.~Ma.
	\newblock {L}inear {S}table {S}ampling {R}ate: {O}ptimality of 2d {W}avelet {R}econstructions from {F}ourier {M}easurements.
	\newblock {\em SIAM J. Math. Anal.}, 47(2):1196--1233, 2015.
	
	\bibitem{sharpBounds}
	B.~Adcock, A.~Hansen, and C.~Poon.
	\newblock {B}eyond {C}onsistent {R}econstructions: {O}ptimality and {S}harp {B}ounds for
	{G}eneralized {S}ampling, and {A}pplication to the {U}niform {R}esampling {P}roblem.
	\newblock {\em SIAM J. Math. Anal.}, 45(5):3132--3167, 2013.
	
	\bibitem{linearity}
	B.~Adcock, A.~Hansen, and C.~Poon.
	\newblock {O}n {O}ptimal {W}avelet {R}econstructions from {F}ourier {S}amples: {L}inearity
	and {U}niversality.
	\newblock {\em Appl. Comput. Harmon. Anal.}, 36(3):387--415, 2014.
	
	\bibitem{GS}
	B.~Adcock and A.~C. Hansen.
	\newblock {A} {G}eneralized {S}ampling {T}heorem for {S}table {R}econstructions in {A}rbitrary {B}ases.
	\newblock {\em J. Fourier Anal. Appl.}, 18(4):685--716, 2010.
	
	\bibitem{GSinfCS}
	B.~Adcock and A.~C. Hansen.
	\newblock {G}eneralized {S}ampling and {I}nfinite-{D}imensional {C}ompressed {S}ensing.
	\newblock {\em Foundations of Computational Mathematics}, 16(5):1263--1323,
	2016.
	
	\bibitem{AHPWavelet}
	B.~Adcock, A.~C. Hansen, and C.~Poon.
	\newblock {O}n {O}ptimal {W}avelet {R}econstructions from {F}ourier {S}amples: {L}inearity
	and universality of the {S}table {S}ampling {R}ate.
	\newblock {\em Appl. Comput. Harmon. Anal.}, 36(3):387--415, 2014.
	
	\bibitem{breaking}
	B.~Adcock, A.~C. Hansen, C.~Poon, and B.~Roman.
	\newblock {B}reaking the {C}oherence {B}arrier: {A} {N}ew {T}heory for {C}ompressed {S}ensing.
	\newblock In {\em Forum of Mathematics, Sigma}, volume~5. Cambridge University
	Press, 2017.
	
	\bibitem{NoteOnHaarFourier}
	B.~Adcock, A.~C. Hansen, and B.~Roman.
	\newblock {A} {N}ote on {C}ompressed {S}ensing of {S}tructured {S}parse {W}avelet
	{C}oefficients from {S}ubsampled {F}ourier {M}easurements.
	\newblock {\em IEEE Signal Process. Letters}, 23(5):732--736, 2016.
	
	\bibitem{SamplingTranslates}
	A.~Aldroubi and M.~Unser.
	\newblock {A} {G}eneral {S}ampling {T}heory for {N}onideal {A}cquisition {D}evices.
	\newblock {\em IEEE Trans. Signal Process.}, 42(11):2915--2925, 1994.
	
	\bibitem{Vegard}
	V.~Antun.
	\newblock {C}oherence {E}stimates between {H}adamard {M}atrices and {D}aubechies	{W}avelets.
	\newblock Master's thesis, University of Oslo, 2016.
	
	\bibitem{SPGL1}
	V.~Antun.
	\newblock Spgl1.
	\newblock \url{https://github.com/vegarant/spgl1}, 2017.
	
	\bibitem{GSCS}
	V.~Antun.
	\newblock cww - {G}eneralized {S}ampling with {W}alsh {S}ampling.
	\newblock \url{https://github.com/vegarant/cww}, 2019.
	
	\bibitem{deVore1}
	M.~Bachmayr, A.~Cohen, R.~DeVore, and G.~Migliorati.
	\newblock {S}parse {P}olynomial {A}pproximation of {P}arametric {E}lliptic {P}des. {P}art ii:
	{L}ognormal {C}oefficients.
	\newblock {\em ESAIM: Mathematical Modelling and Numerical Analysis},
	51(1):341--363, 2017.
	
	\bibitem{deVore3}
	P.~Binev, A.~Cohen, W.~Dahmen, R.~DeVore, G.~Petrova, and P.~Wojtaszczyk.
	\newblock {D}ata {A}ssimilation in {R}educed {M}odeling.
	\newblock {\em SIAM/ASA Journal on Uncertainty Quantification}, 5(1):1--29,
	2017.
	
	\bibitem{FiniteSection1}
	A.~B{\"o}ttcher.
	\newblock {I}nfinite {M}atrices and {P}rojection {M}ethods: in lectures on {O}perator
	{T}heory and its {A}pplications, fields inst. monogr.
	\newblock {\em Amer. Math. Soc.}, (3):1--72, 1996.
	
	\bibitem{CS1}
	E.~Cand\`es, J.~Romberg, and T.~Tao.
	\newblock {R}obust {U}ncertainty {P}rinciples: {E}xact {S}ignal {R}econstruction from
	{H}ighly {I}ncomplete {F}ourier {I}nformation.
	\newblock {\em IEEE Trans. Inform. Theory}, (52):489--509, 2006.
	
	\bibitem{Curvelets}
	E.~J. Cand\`es and L.~Demanet.
	\newblock {C}urvelets and {F}ourier {I}ntegral {O}perators.
	\newblock {\em C. R. Acad. Sci.}, 336(1):395--398, 2003.
	
	\bibitem{Curvelets3}
	E.~J. Cand\`es and D.~Donoho.
	\newblock {N}ew {T}ight {F}rames of {C}urvelets and {O}ptimal {R}epresentations of {O}bjects
	with {P}iecewise $c^2$ {S}ingularities.
	\newblock {\em Comm. Pure Appl. Math.}, 57(2):219--266, 2004.
	
	\bibitem{Curvelets2}
	E.~J. Cand\`es and L.~Donoho.
	\newblock {R}ecovering {E}dges in {I}ll-{P}osed {I}nverse {P}roblems: {O}ptimality of
	{C}urvelet {F}rames.
	\newblock {\em Ann. Statist.}, 30(3):784--842, 2002.
	
	\bibitem{Stanford_CT}
	K.~Choi, S.~Boyd, J.~Wang, L.~Xing, L.~Zhu, and T.-S. Suh.
	\newblock {{C}ompressed {S}ensing {B}ased {C}one-{B}eam {C}omputed {T}omography {R}econstruction with a {F}irst-{O}rder {M}ethod}.
	\newblock {\em Medical Physics}, 37(9), 2010.
	
	\bibitem{compressiveHolography}
	P.~Clemente, V.~Dur{\'a}n, E.~Tajahuerce, P.~Andr{\'e}s, V.~Climent, and
	J.~Lancis.
	\newblock {C}ompressive {H}olography with a {S}ingle-{P}ixel {D}etector.
	\newblock {\em Optics Letters}, 38(14):2524--2527, 2013.
	
	\bibitem{boundaryWavelets}
	A.~Cohen, I.~Daubechies, and P.~Vial.
	\newblock {W}avelets on the {I}nterval and {F}ast {W}avelet {T}ransforms.
	\newblock {\em Comput. Harmon. Anal.}, 1(1):54--81, 1993.
	
	\bibitem{Shearlets2}
	S.~Dahlke, G.~Kutyniok, P.~Maass, C.~Sagiv, H.-G. Stark, and G.~Teschke.
	\newblock {T}he {U}ncertainty {P}rinciple {A}ssociated with the {C}ontinuous {S}hearlet
	{T}ransform.
	\newblock {\em Int. J. Wavelets Multiresolut. Inf. Process.}, 2(6):157--181,
	2008.
	
	\bibitem{Shearlets3}
	S.~Dahlke, G.~Kutyniok, G.~Steidl, and G.~Teschke.
	\newblock {S}hearlet {C}oorbit {S}paces and {A}ssociated {B}anach {F}rames.
	\newblock {\em Appl. Comput. Harmon. Anal.}, 2(27):195--214, 2009.
	
	\bibitem{deVore2}
	R.~DeVore, G.~Petrova, and P.~Wojtaszczyk.
	\newblock {D}ata {A}ssimilation and {S}ampling in {B}anach {S}paces.
	\newblock {\em Calcolo}, 54(3):963--1007, 2017.
	
	\bibitem{Donoho}
	D.~L. Donoho.
	\newblock {C}ompressed {S}ensing.
	\newblock {\em IEEE Transactions on Information Theory}, 52(4):1289--1306,
	2006.
	
	\bibitem{RobustConsistentSampling}
	T.~Dvorkind and Y.~C. Eldar.
	\newblock {R}obust and {C}onsistent {S}ampling.
	\newblock {\em IEEE Signal Process. Letters}, 16(9):739--742, 2009.
	
	\bibitem{eldar2003FAA}
	Y.~C. Eldar.
	\newblock {S}ampling with {A}rbitrary {S}ampling and {R}econstruction {S}paces and
	{O}blique {D}ual {F}rame {V}ectors.
	\newblock {\em J. Fourier Anal. Appl.}, 9(1):77--96, 2003.
	
	\bibitem{RobustConsistentSampling3}
	Y.~C. Eldar.
	\newblock {S}ampling without {I}nput {C}onstraints: {C}onsistent {R}econstruction in
	{A}rbitrary {S}paces.
	\newblock {\em Sampling, Wavelets and Tomography}, 2003.
	
	\bibitem{RobustConsistentSampling4}
	Y.~C. Eldar and T.~Werther.
	\newblock {G}eneral {F}ramework for {C}onsistent {S}ampling in {H}ilbert {S}paces.
	\newblock {\em Int. J. Wavelets Multiresolut. Inf. Process.}, 3(4):497--509,
	2005.
	
	\bibitem{SinglePixel}
	S.~Foucart and H.~Rauhut.
	\newblock {\em {A} {M}athematical {I}ntroduction to {C}ompressive {S}ensing}.
	\newblock Springer Science+Business Media, New York, 2013.
	
	\bibitem{MilanaClarice}
	M.~Gataric and C.~Poon.
	\newblock {A} {P}ractical {G}uide to the {R}ecovery of {W}avelet {C}oefficients from
	{F}ourier {M}easurements.
	\newblock {\em SIAM J. Sci. Comput.}, 38(2):A1075--A1099, 2016.
	
	\bibitem{deutschWalsh}
	E.~Gauss.
	\newblock {\em {W}alsh {F}unktionen f{\"u}r {I}ngenieure und {N}aturwissenschaftler}.
	\newblock Springer Fachmedien, Wiesbaden, 1994.
	
	\bibitem{FiniteSectionGroechnig}
	K.~Gr{\"o}chenig, Z.~Rzeszotnik, and T.~Strohmer.
	\newblock {Q}uantitative {E}stimates for the {F}inite {S}ection {M}ethod and {B}anach
	{A}lgebras of {M}atrices.
	\newblock {\em Integral Equations and Operator Theory}, 2(67):183--202, 2011.
	
	\bibitem{Unser_MRI}
	M.~Guerquin-Kern, M.~H{\"{a}}berlin, K.~Pruessmann, and M.~Unser.
	\newblock {A} {F}ast {W}avelet-{B}ased {R}econstruction {M}ethod for {M}agnetic {R}esonance
	{I}maging.
	\newblock {\em {IEEE} Transactions on Medical Imaging}, 30(9):1649--1660, 2011.
	
	\bibitem{FiniteSection3}
	A.~C. Hansen.
	\newblock {O}n the {A}pproximation of {S}pectra of {L}inear {O}perators on {H}ilbert
	{S}paces.
	\newblock {\em J. Funct. Anal.}, 8(254):2092--2126, 2008.
	
	\bibitem{Hansen_JAMS}
	A.~C. Hansen.
	\newblock {O}n the {S}olvability {C}omplexity {I}ndex, the {$n$}-{P}seudospectrum and
	{A}pproximations of {S}pectra of {O}perators.
	\newblock {\em J. Amer. Math. Soc.}, 24(1):81--124, 2011.
	
	\bibitem{WalshSSR}
	A.~C. Hansen and L.~Thesing.
	\newblock {O}n the {S}table {S}ampling {R}ate for {B}inary {M}easurements and {W}avelet
	{R}econstruction.
	\newblock {\em Appl. Comput. Harmon. Anal.}, 48(2):630--654, 2020.
	
	\bibitem{PolySSR}
	T.~Hrycak and K.~Gr{\"o}chenig.
	\newblock {P}seudospectral {F}ourier {R}econstruction with the {M}odified {I}nverse
	{P}olynomial {R}econstruction {M}ethod.
	\newblock {\em J. Comput. Phys.}, 229(3):933--946, 2010.
	
	\bibitem{LenslessImaging}
	G.~Huang, H.~Jiang, K.~Matthews, and P.~Wilford.
	\newblock {L}ensless {I}maging by {C}ompressive {S}ensing.
	\newblock In {\em 2013 IEEE International Conference on Image Processing},
	pages 2101--2105. IEEE, 2013.
	
	\bibitem{Nature_sci_rep}
	A.~Jones, A.~Tamt{\"o}gl, I.~Calvo-Almaz{\'a}n, and A.~C. Hansen.
	\newblock {C}ontinuous {C}ompressed {S}ensing for {S}urface {D}ynamical {P}rocesses with
	{H}elium {A}tom {S}cattering.
	\newblock {\em Nature Sci. Rep.}, 6:27776 EP --, 06 2016.
	
	\bibitem{KutyniokLimShearletFourier}
	G.~Kutyniok and W.-Q. Lim.
	\newblock {O}ptimal {C}ompressive {I}maging of {F}ourier {D}ata.
	\newblock {\em SIAM Journal on Imaging Sciences}, 11(1):507--546, 2018.
	
	\bibitem{leary2013etcs}
	R.~Leary, Z.~Saghi, P.~A. Midgley, and D.~J. Holland.
	\newblock {C}ompressed {S}ensing {E}lectron {T}omography.
	\newblock {\em Ultramicroscopy}, 131(0):70--91, 2013.
	
	\bibitem{li2017uniformRecoveryWavFourier}
	C.~Li and B.~Adcock.
	\newblock {C}ompressed {S}ensing with {L}ocal {S}tructure: {U}niform {R}ecovery {G}uarantees
	for the {S}parsity in {L}evels {C}lass.
	\newblock {\em Appl. Comput. Harmon. Anal.}, 2017.
	
	\bibitem{FiniteSection4}
	M.~Lindner.
	\newblock {\em {I}nfinite {M}atrices and {T}heir {F}inite {S}ections: {A}n {I}ntroduction to
		the {L}imit {O}perator {M}ethod}.
	\newblock Birkh{\"a}user Verlag, Basel, 2006.
	
	\bibitem{MRI}
	M.~Lustig, D.~Donoho, and J.~M. Pauly.
	\newblock {S}parse {M}ri: {T}he {A}pplication of {C}ompressed {S}ensing for {R}apid {M}R
	{I}maging.
	\newblock {\em Magnetic Resonance in Medicine}, 58(6):1182--1195, 2007.
	
	\bibitem{Ma}
	J.~Ma.
	\newblock {G}eneralized {S}ampling {R}econstruction from {F}ourier {M}easurements {U}sing
	{C}ompactly {S}upported {S}hearlets.
	\newblock {\em Appl. Comput. Harmon. Anal.}, 2015.
	
	\bibitem{maday3}
	Y.~Maday, T.~Anthony, J.~D. Penn, and M.~Yano.
	\newblock {P}bdw {S}tate {E}stimation: {N}oisy {O}bservations; {C}onfiguration-{A}daptive
	{B}ackground {S}paces; {P}hysical {I}nterpretations.
	\newblock {\em ESAIM: Proceedings and Surveys}, 50:144--168, 2015.
	
	\bibitem{maday2}
	Y.~Maday and O.~Mula.
	\newblock {A} {G}eneralized {E}mpirical {I}nterpolation {M}ethod: {A}pplication of {R}educed
	{B}asis {T}echniques to {D}ata {A}ssimilation.
	\newblock In {\em Analysis and numerics of partial differential equations},
	pages 221--235. Springer, 2013.
	
	\bibitem{PBDW}
	Y.~Maday, A.~T. Patera, J.~D. Penn, and M.~Yano.
	\newblock {A} {P}arameterized-{B}ackground {D}ata-{W}eak {A}pproach to {V}ariational {D}ata
	{A}ssimilation: {F}ormulation, {A}nalysis, and {A}pplication to {A}coustics.
	\newblock {\em International Journal for Numerical Methods in Engineering},
	102(5):933--965, 2015.
	
	\bibitem{WaveletDef}
	S.~Mallat.
	\newblock {\em {A} {W}avelet {T}our of {S}ignal {P}rocessing}.
	\newblock Academic Press, San Diego, 1998.
	
	\bibitem{moshtaghpour2020close}
	A.~Moshtaghpour, J.~M. Bioucas-Dias, and L.~Jacques.
	\newblock {C}lose {E}ncounters of the {B}inary {K}ind: {S}ignal {R}econstruction {G}uarantees
	for {C}ompressive {H}adamard {S}ampling with {H}aar {W}avelet {B}asis.
	\newblock {\em IEEE Transactions on Information Theory}, 2020.
	
	\bibitem{fluorescence}
	M.~M{\"u}ller.
	\newblock {\em {I}ntroduction to {C}onfocal {F}luorescence {M}icroscopy}.
	\newblock SPIE, Bellingham, Washington, 2006.
	
	\bibitem{Clarice}
	C.~Poon.
	\newblock {A} {C}onsistent and {S}table {A}pproach to {G}eneralized {S}ampling.
	\newblock {\em J. Fourier Anal. Appl.}, (20):985--1019, 2014.
	
	\bibitem{PoonFrames}
	C.~Poon.
	\newblock {S}tructure {D}ependent {S}ampling in {C}ompressed {S}ensing: {T}heoretical
	{G}uarantees for {T}ight {F}rames.
	\newblock {\em Appl. Comput. Harmon. Anal.}, 42(3):402--451, 2017.
	
	\bibitem{quinto2006xrayradon}
	E.~T. Quinto.
	\newblock {A}n {I}ntroduction to {X}-{R}ay {T}omography and {R}adon {T}ransforms.
	\newblock In {\em The Radon Transform, Inverse Problems, and Tomography},
	volume~63, pages 1--23. American Mathematical Society, 2006.
	
	\bibitem{Shannon}
	C.~Shannon.
	\newblock {A} {M}athematical {T}heory of {C}ommunication.
	\newblock {\em Bell Syst. Tech. J.}, (27):379--423, 623--656, 1948.
	
	\bibitem{Candes_PNAS}
	V.~Studer, J.~Bobin, M.~Chahid, H.~S. Mousavi, E.~Candes, and M.~Dahan.
	\newblock {C}ompressive {F}luorescence {M}icroscopy for {B}iological and {H}yperspectral
	{I}maging.
	\newblock {\em Proceedings of the National Academy of Sciences},
	109(26):E1679--E1687, 2012.
	
	\bibitem{laserBasedFailureAnalysis}
	T.~Sun, G.~Woods, M.~F. Duarte, K.~Kelly, C.~Li, and Y.~Zhang.
	\newblock {O}bic {M}easurements without {L}asers or {R}aster-{S}canning {B}ased on
	{C}ompressive {S}ensing.
	\newblock In {\em Int. Symposium for Testing and Failure Analysis (ISTFA), San
		Jose, CA}, pages 272--277, 2009.
	
	\bibitem{CodeExperiments}
	L.~Thesing.
	\newblock infCS.
	\newblock \url{https://github.com/laurathesing/infCS}, 2020.
	
	\bibitem{WalshHaar}
	L.~Thesing and A.~C. Hansen.
	\newblock {L}inear {R}econstructions and the {A}nalysis of the {S}table {S}ampling {R}ate.
	\newblock {\em Sampling Theory in Image Processing}, 17(1):103--126, 2018.
	
	\bibitem{Shannon50}
	M.~Unser.
	\newblock {S}ampling - 50 {Y}ears after {S}hannon.
	\newblock {\em Proc. IEEE}, 4(88):569--587, 2000.
	
	\bibitem{ConsistentSampling}
	M.~Unser and J.~Zerubia.
	\newblock {A} {G}eneralized {S}ampling {T}heory without {B}and-{L}imiting {C}onstraints.
	\newblock {\em IEEE Trans. Circuits Syst. II.}, 45(8):959--969, 1998.
	
\end{thebibliography}
\end{document}